\documentclass[a4paper]{siamonline220329} 
\usepackage{matlab-prettifier}
\usepackage{damacros}
\usepackage{mwe}

\usepackage{pict2e}
\makeatletter
\DeclareRobustCommand{\bigtimes}{%
  \mathop{\vphantom{\sum}\mathpalette\@bigtimes\relax}\slimits@
}
\newcommand{\@bigtimes}[2]{\vcenter{\hbox{\make@bigtimes{#1}}}}
\newcommand{\make@bigtimes}[1]{%
  \sbox\z@{$\m@th#1\sum$}%
  \setlength{\unitlength}{\wd\z@}%
  \begin{picture}(1,1)
  \roundcap
  \linethickness{.17ex}
  \Line(.1,.1)(.9,.9)
  \Line(.1,.9)(.9,.1)
  \end{picture}%
}
\makeatother

\newsiamremark{example}{Example}
\crefname{example}{Example}{Examples}
\Crefname{example}{Example}{Examples}
\newsiamthm{problem}{Problem}
\crefname{problem}{Problem}{Problems}
\Crefname{problem}{Problem}{Problems}
\renewcommand{\phi}{\varphi}
\DeclareMathOperator{\dist}{dist}

\newcommand{\epsi}{\varepsilon}

\newcommand{\TheTitle}{Stochastic collocation schemes for Neural Field Equations with random
data}
\newcommand{\TheShortTitle}{Stochastic collocation for Neural Field Equations}
\newcommand{\TheAuthors}{D. Avitabile, F. Cavallini, S. Dubinkina, G.~J. Lord}
\headers{\TheShortTitle}{\TheAuthors}

\title{{\TheTitle}}
 
\author{%
    Daniele Avitabile\thanks{%
    Amsterdam Centre of Mathematics and Computation,
    Department of Mathematics,
    Vrije Universiteit Amsterdam, 
    De Boelelaan 1111,
    1081 HV Amsterdam, The Netherlands.
    MathNeuro Team,
    Inria branch of the University of Montpellier,
    860 rue Saint-Priest
    34095 Montpellier Cedex 5
    France.
    (\email{d.avitabile@vu.nl}).
  }
  \and Francesca Cavallini\thanks{%
    Amsterdam Centre of Mathematics and Computation
    Department of Mathematics,
    Vrije Universiteit Amsterdam, 
    De Boelelaan 1111,
    1081 HV Amsterdam, The Netherlands.
  }
  \and Svetlana Dubinkina\footnotemark[2]
  \and Gabriel~J. Lord\thanks{%
    Department of Mathematics,
    Radboud University, 
    Postbus 9010,
    6500 GL Nijmegen,
    The Netherlands.
  }
}

\graphicspath{{Figures/}}

\usepackage[
  createShortEnv,
  conf={my text link},
]{proof-at-the-end}

\pgfkeys{
  /prAtEnd/.cd,
  my text link/.style={
    text link={
      \begin{proof}
      See \hyperref[proof:prAtEnd\pratendcountercurrent]{proof}
      on page~\pageref{proof:prAtEnd\pratendcountercurrent}.
      \end{proof}
    },
  },
}

\renewtheorem{proposition}{Proposition}[section]

\renewtheorem{corollary}{Corollary}[section]

\renewtheorem{lemma}{Lemma}[section]

\begin{document} 
\maketitle

\begin{abstract}
  We develop and analyse numerical schemes for uncertainty quantification in neural
  field equations subject to random parametric data in the synaptic kernel, firing
  rate, external stimulus, and initial conditions. The schemes combine a generic
  projection method for spatial discretisation to a stochastic collocation scheme for
  the random variables. We study the problem in operator form, and derive estimates
  for the total error of the schemes, in terms of the spatial projector. We give
  conditions on the projected random data which guarantee analyticity of the
  semi-discrete solution as a Banach-valued function. We illustrate how to verify
  hypotheses starting from analytic random data and a choice of spatial projection.
  We provide evidence that the predicted convergence rates are found in various
  numerical experiments for linear and nonlinear neural field problems.
\end{abstract}


\section{Introduction}\label{sec:introduction} 
Modelling brain dynamics and comprehending how uncertainties in the inputs affect
quantities of interest (QOI) is a fundamental question in neuroscience. This field
faces several challenges, including epistemic uncertainty arising from imperfect
models, which are often phenomenological in nature. Additionally, the nonlocality of
these models necessitates specialized numerical approaches.

In this paper we study uncertainty quantification (UQ) in a class of nonlinear brain
activity models known as \textit{neural fields}, which are integro-differential
equations used as large-scale descriptions of neuronal
activity. They model the cortex as a continuum, and provide a versatile tool to understand pattern
formation on a variety of spatial domains
\cite{wilson1973mathematical,
amari1977dynamics,Ermentrout.1998qno,
Bressloff.2014k0p,
coombes2014neural}. 

Neural fields are amenable to functional and nonlinear analysis in simple cortices,
and thus can be used to investigate fundamental mechanisms for the generation of
brain activity. For instance, neural field simulations on spherical
domains~\cite{robinson1997propagation,bressloff2003spherical,Visser.2017}, or
more realistic cortices
\cite{martin2018numerical,pangGeometricConstraintsHuman2023,shaw2025radialbasisfunctiontechniques}
support several coherent structures observed experimentally, such as 
waves or stationary localised
structures~\cite{lee2005traveling,huang2010spiral,kimRingAttractorDynamics2017}.

While a large body of work shows that nonlinearity and nonlocality are building blocks
for patterns of neural activity, their robustness to noisy input data is still mostly
unexplored. In the mathematical and computational neuroscience communities, Monte
Carlo sampling is a popular approach to estimate mean and variance of
QOI. Methods offering faster convergence rates, such as Stochastic Finite Elements or
Stochastic Collocation, have been developed for applications
in other branches of physical and life sciences (see 
\cite{smithUncertaintyQuantificationTheory2014, Lord:2014ir,
adcockSparsePolynomialApproximation2022} for textbooks discussing
these methods). The literature on these schemes,
however, focuses predominantly on Partial Differential Equations (PDEs)~\cite{
xiuModelingUncertaintySteady2002,
ghanemStochasticFiniteElements2003,
babuskaGalerkinFiniteElement2004a,
babuskaStochasticCollocationMethod2007,
frauenfelderFiniteElementsElliptic2005,
xiuHighOrderCollocationMethods2005,
romanStochasticGalerkinMethod2006,
nobileSparseGridStochastic2008a,
nobileAnisotropicSparseGrid2008,
nobileAnalysisImplementationIssues2009,
hoangSparseTensorGalerkin2013a}, and is not immediately applicable to neural
field equations, even though UQ techniques for ODEs have recently been applied to
connectomic ODE models of reaction--diffusion processes for
neurodegenration~\cite{cortiUncertaintyQuantificationFisherKolmogorov2024}.

A classical neural field problem is written in terms of an activity variable
$u(x,t)$, modelling voltage or firing rate at time $t$ and position $x$ of a cortical
domain $D \subset \RSet^3$. We consider (for now informally) a neural field subject
to the following independent random data:  an initial condition $v(x,\omega_v)$, an
external forcing $g(x,t,\omega_g)$, a synaptic kernel $w(x,x',\omega_w)$, modelling
connections from point $x'$ to $x$ in $D$, and a firing rate function
$f(u,\omega_f)$, modelling how neurons transform input currents into spiking rates.
We distinguish between a random linear neural field (RLNF, henceforth indicated with
the index $\LSet = 1$) in
which the firing rate is deterministic and linear $f(u,\omega_f) \equiv u$, and a
random nonlinear neural field (RNNF, $\LSet = 0$). Neural field problems with random input data read as
follows: for fixed $v$, $w$, $g$, and $f$, we seek for a mapping $u \colon D\times J
\times \Omega \to \RSet$ such that for $\prob$-almost all $\omega \in
\Omega$ it holds
\begin{equation}\label{eq:nonlinRNF}
  \begin{aligned}
    &\partial_t u(x,t,\omega)  =  -u(x,t,\omega) + g(x,t,\omega_g) 
     + \int_D w(x,x',\omega_w) f(u(x',t,\omega),\omega_f)\, dx',  \\
    &u(x,0,\omega)  = v(x,\omega_{v}).
\end{aligned}
\end{equation}

\begin{figure}\label{fig:ExpVar}
\centering
	\includegraphics{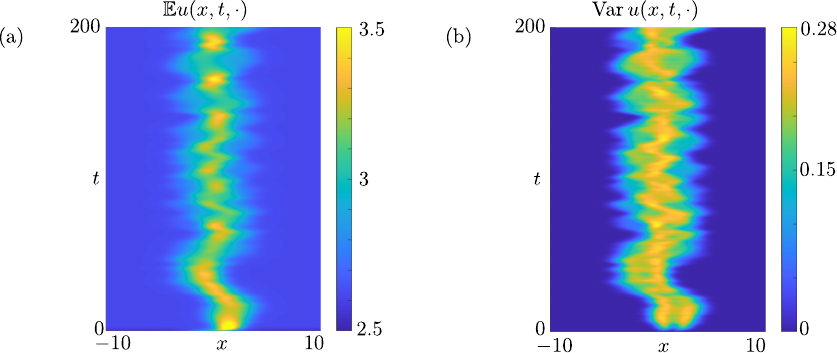}
  \caption{Expectation and variance of a solution $u$ to the nonlinear
    neural field \cref{eq:nonlinRNF} with random forcing. The problem is posed on a 1-dimensional
    ring $D = \RSet/L \ZSet$ of width $L = 22$, and time interval $J = [0,200]$.
    Deterministic data is specified through 
    an excitatory-inhibitory kernel $w(x,y) = W(x-y)$ with $W(z) = (2-z^2)\exp(-z^2)$, 
    a sigmoidal firing rate $f(u) = [ 1+ \exp( -20(u-10) ) ]^{-1}$, and initial
    condition $v(x) = 2.5 + 0.5/\cosh^2(0.5 x)$. The external input is an oscillating
    pulse with random instantaneous speed $c$ depending on $6$ random variables:
    the forcing is given by $g(x,t,Y(\omega_g))$, where $g(x,t,y) = 1.4
    \exp(-(x-c(t,y))^2)$, $y = (c_1,c_2,c_3,f_1,f_2,f_3)$, $c(t,y) = \sum_{k=1}^3 c_k
    \sin(2\pi t /f_k)$. The random variables $Y_i$ are independently distributed with $Y_i \sim
    \mathcal{U}[\alpha_i,\beta_i]$. The model is discretised in space using a
    spectrally convergent Fourier collocation scheme with $n = 100$ gridpoints.
    Expectation and variance are computed with a stochastic collocation scheme, with
    a dense, tensor-product grid of 10 Gauss-Legendre points in each of
    the intervals $[\alpha_i,\beta_i]$, for a total $10^6$ points. Parameters
    $[\alpha_1,\beta_1] = [0,4]$,
    $[\alpha_2,\beta_2] = [1/6,2/3]$,
    $[\alpha_3,\beta_3] = [1/10,4/5]$,
    $[\alpha_4,\beta_4] = [40,60]$,
    $[\alpha_5,\beta_5] = [10,50/3]$,
    $[\alpha_6,\beta_6] = [100,200]$.
  }
\end{figure}

We defer to later sections the probabilistic setting for data and solutions to the
problem, and we view \cref{eq:nonlinRNF} as a prototype for studying forward UQ for
spatio-temporal brain activity, as it encapsulates inputs and mathematical structure
present in more realistic and involved neurobiological models. Recently,
two stepping stones for investigating UQ problems in the form
\cref{eq:nonlinRNF} have been laid. Firstly, spatial discretisations
of \textit{deterministic problems} have been investigated in abstract form
\cite{avitabileProjectionMethodsNeural2023}: following up from work by Atkinson
\cite{atkinson2005theoretical}, it is possible to derive convergence estimates for
Collocation and Galerkin schemes, of Finite-Elements or Spectral type, as generic
projection methods. Secondly, in \cite{avitabile2024NeuralFields} we studied
\cref{eq:nonlinRNF} as a Cauchy problem with random data, posed on
Banach spaces; in particular, we have established conditions on the input data that
guarantee the $L^p$-regularity of the mapping $\omega \mapsto
u(\blank,\blank,\omega)$ in the natural function spaces for neural fields
and for their spatially-discretised versions; further, the theory in
\cite{avitabile2024NeuralFields} covers the case of finite-dimensional noise,
which deals with inputs parametrised by a finite, possibly large, number of
random variables.

We are thus in the position of analysing schemes that approximate statistics of
QOI, which is the main goal of the present article. An example of
these types of computations is given in \Cref{fig:ExpVar}, where the mean and
variance of $\omega \mapsto u(x,t,\omega)$ is approximated for a neural field
equation posed on a simple 1D cortex (a ring), subject to a random forcing. The
deterministic data used for the simulation (synaptic kernel, firing rate, and initial
value) are frequently used in the mathematical neuroscience literature, and are such
that the neural field supports a bump of localised activity (a steady state). The
computation shows how this stationary pattern is affected by a random, pulsatile,
sloshing external forcing. This setup is relevant to cortical processes associated to
working memory associated to the brain navigational system, or
oculomotor responses (we refer to \cite{avitabileBumpAttractorsWaves2023} and
references therein for further modelling and experimental work).

In \cref{fig:ExpVar} the system is discretised in space using a Chebyshev
interpolatory projector from \cite{avitabileProjectionMethodsNeural2023}, and QOI are
estimated with a Stochastic Collocation scheme in the spirit of the classical work by
Babu\v ska, Nobile, and Tempone~\cite{babuskaStochasticCollocationMethod2007}. This
method is part of a family of schemes analysed here, that we call
\textit{spatial-projection, stochastic-collocation schemes}, in which
stochastic collocation is paired to an arbitrary projection scheme, thus retaining the
flexibility and generality of the framework in
\cite{avitabileProjectionMethodsNeural2023,avitabile2024NeuralFields}.

The analysis of generic spatial-projection stochastic collocation schemes for
neural fields is the main contribution of this paper. In the PDE literature cited
above, the spatial projector is often an orthogonal
projector, leading prevalently to Galerkin Finite Elements schemes. This
approach, which has also been extended to time-dependent, linear parabolic PDEs by
Zhang and Gunzburger~\cite{Zhang.2012}, stems naturally from the deterministic
functional setup, where schemes on nontrivial geometries are presented using weak
formulations. Adapting this approach to neural fields comes with different
challenges, and naturally brings a generalisation: it is possible to study in
parallel neural fields as dynamical systems on the Banach phase spaces $\XSet =C(D)$
and $ \XSet = L^2(D)$
\cite{Potthast:2010kb,faugeras2008absolute}; this dichotomy shows up also in their
numerical treatment, which features a
single projector $P_n$ on an $n$-dimensional subspace of $\XSet \in \{ C(D),
L^2(D)\}$, encompassing at once schemes in strong and weak form for collocation and
Galerkin methods, respectively. The convergence rates of such schemes is determined
in a unified manner, studying how $P_n$ approximates an element $v$ of the phase
space, via an estimate of $\| P_n v - v \|$~\cite{avitabileProjectionMethodsNeural2023}. 

Taking the literature on elliptic and parabolic PDEs as a guiding
example~\cite{babuskaStochasticCollocationMethod2007,Zhang.2012}, we
estimate the error $ \|  u - u_{n,q} \|$, measured in an appropriate norm, between
the solution $u$ and an approximation $u_{n,q}$ which combines an $s(n)$-dimensional
(interpolatory or orthogonal) spatial projection with a $d(q)$-dimensional
interpolatory projection for the random variables, for some $s(n)$ and $d(q)$ with
$s(n) \to \infty$ as $n \to \infty $ and $d(q) \to \infty $ as $q \to \infty $, respectively. We derive an error bound in terms
of a spatial component $\| u - u_n \|$ and a stochastic collocation component $\| u_n -
u_{n,q}\|$ with constants that are homogeneous in $q$ and $n$, respectively.

However, since the neural field theory is built around a generic projected dynamical
system with random data, this leads to some differences from the PDE
literature. For instance, no assumption can be made on the convergence of $\|  P_n
\|$ as $n \to \infty$, in contrast to existing literature using orthogonal projectors for
which $\| P_n \|$ is bounded homogeneously in $n$. Keeping generality is valuable for
neuroscience applications, because efficient schemes for neural fields use projectors
that are not orthogonal, or for which the sequence $\{ \| P_n \| \}_n$ is unbounded (the one
used in \cref{fig:ExpVar} being one of them). 

Existing literature on elliptic and parabolic PDEs derives exponential convergence
rates in $q$ for $\| u_n - u_{n,q} \|$ by proving that analyticity of the random input data
implies analyticity of $u$~\cite{babuskaStochasticCollocationMethod2007,Zhang.2012}.
Our theory recovers analogous results for neural fields when $\{ \| P_n \| \}_n$ is bounded
in $n$, but is applicable to cases in which $\{ \| P_n \| \}_n$ is unbounded, provided the
random input data has sufficient \textit{spatial} regularity. In particular, we give
conditions on the \textit{projected random input data} (for instance $P_n g $ and
$P_n v$) which guarantee the analyticity of the spatially discretised solution $u_n$
as a function of the random variables. Because the derived estimates on the
derivatives are homogeneous in $n$, the same occurs for the stochastic collocation
error $\| u_n - u_{n,q} \|$, as required.

In addition, the treatment presented here does not assume Hilbert structure of the
phase space $\XSet$ of the dynamical system with random data, leading to strong norms
when $\XSet = C(D)$, a case that is frequent in neural field applications. To this end
we use definitions and tools for Banach-space valued
functions~\cite{buffoni_analytic_2003}.

We derive error bounds for linear neural fields, without assuming dissipative
dynamics, that is, without assuming that the underlying semigroup operator is
contractive, in a sense that we make precise later. We link the lack of contractivity
to time-dependent analyticity radii that shrink as time increases. We argue that,
while analyticity results are in place for the nonlinear case, the sharp gradients
induced by non-contractive dynamics is to be expected for nonlinear problems near
bifurcation points. Finally, we provide evidence of exponential convergence rates in
linear and nonlinear problems, and signpost a slower convergence rates near a
bifurcating solution. Throughout the paper, we provide illustrative examples on
how to check hypotheses on random data with affine and non-affine dependence on
parameters, with various spatial projectors.

The paper is organised as follows: we set our notation in \cref{sec:notation}, and
summarise hypotheses and results for neural fields with random data in
\cref{sec:mainHypotheses}. Our numerical scheme is introduced in
\cref{sec:numericalScheme}, and analysed for linear problems in
\cref{sec:linearAnalysis}. \Cref{sec:nonlinear} contains comments on the nonlinear
problem, while \cref{sec:numerics} presents numerical experiments. We conclude in
\cref{sec:conclusions}.

\section{Notation}\label{sec:notation} 
We set $\NSet_n = \{ 1,\ldots,n \}$ and $\NSet_{> n} = \NSet \setminus \NSet_n$ (we
similarly use $\NSet_{\geq n}$, and $\ZSet_{\geq 0}$, where $\ZSet$ denotes the
integers). The symbol $\XSet$ denotes a Banach space of functions defined on some $D
\subset \RSet^\textrm{dim}$, while $J = [0,T] \subset \RSet$ is used for
time intervals. We denote by $C^k(\Gamma,\BSet)$ the space of $k$-times continuously
differentiable functions on a
subspace $\Gamma \subseteq \RSet^m$ to a Banach space $\BSet$, with norm $\| v
\|_{C^k(\Gamma,\BSet)} = \max_{j \in \NSet_k}\sup_{y \in \Gamma} \| v^{(j)}(y)
\|_{\BSet}$. We use the shortcuts $C(\Gamma,\BSet) :=C^0(\Gamma,\BSet)$ and
$C^k(\Gamma) := C^k(\Gamma,\RSet)$.

The set of bounded linear operators on $\Gamma$ to $\BSet$, with standard operator
norm, is denoted by $BL(\Gamma,\BSet)$. When it is important to emphasise that
derivatives are understood in the sense of Fréchet, we indicate 
the $j$th Frechét derivative of $v \colon \Gamma \to \BSet$ at $y$ by $d^{\,j}v[y]
\colon \Gamma \to BL(\Gamma^k,\BSet)$. 

We use the symbol $u$ for different, interrelated mappings, the most basic being
$u(x,t,\omega)$ for the neural field solution as a mapping from $D \times J \times
\Omega$ to
$\RSet$. We also progressively ``slice" this function from the
leftmost argument, and so $u(t,\omega)$ is a mapping on $J \times \Omega$ to $\XSet$,
so that, with a little abuse of notation, $u(t,\omega)(x) = u(x,t,\omega)$; further,
we may use $u(\omega)$ for the corresponding mapping on $\Omega$ to $\BSet$, where
$\BSet$ is a suitable Banach space of functions on $D \times J$, for instance $\BSet =
C(J,\XSet)$, and so $u(\omega)(x,t) = u(x,t,\omega)$. In addition, it is necessary
to consider finite-dimensional noise random fields, so that $u(x,t,\omega) = \tilde
u(x,t,Y(\omega))$ for suitable choices of the function $\tilde u$ and the
$\RSet^m$-valued random variables $Y$. We often omit the tilde, but we warn the
reader when necessary, to help disambiguating the overloaded symbol
$u$.
 
For a Banach space $\BSet$, and a weight function $\sigma \colon \Gamma \to
\RSet_{\geq 0}$ we make use of the weighted function spaces, 
\begin{equation}\label{eq:WeightSpace}
	C^0_\sigma(\Gamma,\BSet)=\Big\{v:\Gamma\rightarrow \BSet, \quad v \in
    C^0(\Gamma,\BSet), \quad 
	\| u \|_{C^0(\Gamma,\BSet)} := \max_{y\in\Gamma} \Vert \sigma(y) v(y) \Vert_\BSet<\infty \Big\},
\end{equation}
and
\begin{equation}\label{eq:WeightCkSpace}
  \begin{aligned}
  C^k_\sigma(\Gamma,\BSet) =  \Big\{v:\Gamma\rightarrow \BSet, & \text{ $v$ is $k$ times
    differentiable, } \quad  \\
                     & \max_{y\in\Gamma} \bigl\Vert \, \sigma(y) \, d^{\,j}v[y] \,
                     \bigr\Vert_{BL(\Gamma^k,\BSet)}<\infty \quad
	\text{for all}\quad j\in \NSet_k \Big\},
  \end{aligned}
\end{equation}
with norm $\| v \|_{C^k_\sigma(\Gamma,\BSet)} = \max_{j \in \NSet_k} \max_{y \in
\Gamma} \| \sigma(y) d^{\,j}v[y] \|_{BL(\Gamma^k,\BSet)}$.

We denote Bochner spaces of $\BSet$-valued random variables on a probability space
$(\Omega,\calF,\prob)$ as $L^p(\Omega,\calF,\prob; \BSet)$, or simply
$L^p(\Omega,\BSet)$, where $p \in [0,\infty]$; these spaces are the equivalence classes of
strongly $\BSet$-valued
random variables endowed with norms
\[
  \begin{aligned}
  & \| u \|_{L^p(\Omega,\BSet)} = ( \mean \| u \|^p)^{1/p} = 
    \bigg(
    \int_\Omega \| u(\omega) \|^p_\BSet \, d\prob(\omega)
    \bigg)^{1/p} \qquad p \in [0,\infty), \\
  & \| u \|_{L^\infty(\Omega,\BSet)} = \pesssup_{\omega \in \Omega} \| u(\omega)
  \|_\BSet. \\
  \end{aligned}
\]
For a given probability density $\rho \colon \Gamma \subseteq \RSet^m \to
\RSet_{\geq 0}$ we set $L^p_\rho(\Gamma,\BSet) := L^p(\Gamma,\calB(\RSet^m),\rho
dy,\BSet)$.

\section{Hypotheses and results for neural fields with random data}\label{sec:mainHypotheses} 
We review the functional setup and standing
hypotheses for the study of neural fields with random data, as obtained
in~\cite{avitabile2024NeuralFields}. We begin by discussing the spatial and
temporal domains of the neural field equation.

\begin{hypothesis}[Spatio-temporal domain]\label{hyp:domain} It holds $(x,t) \in D
  \times J$, where $D\subset \mathbb{R}^\textrm{dim}$ is a compact domain with
  piecewise smooth boundary, and $J = [0,T] \subset \RSet$.
\end{hypothesis}

We formalise random data in the neural field equations similarly to the linear parabolic PDE
case~\cite{Zhang.2012}, hence we consider the probability spaces
$(\Omega_w,\calF_w,\prob_w)$,
$(\Omega_f,\calF_f,\prob_f)$,
$(\Omega_g,\calF_g,\prob_g)$,
and $(\Omega_{v},\calF_{v},\prob_{v})$ or, compactly
$\{ (\Omega_\alpha,\calF_\alpha,\prob_\alpha) \colon \alpha \in \USet\}$, with $\USet
\subseteq \{ w,f,g,v \}.$
We introduce 
\begin{equation}\label{eq:randomData}
\begin{aligned}
  & w \colon D \times D \times \Omega_w \to \RSet, 
  && f \colon \RSet \times \Omega_f \to \RSet, \\
  & g \colon D\times J \times \Omega_g \to \RSet, 
  && v \colon  D \times \Omega_v \to \RSet,
\end{aligned}
\end{equation}
and we are interested in how uncertainty in the data is propagated by the neural
field model \cref{eq:nonlinRNF}. We recall that the index $\LSet = 1$ indicates
linear and deterministic firing rates,
$f(u) = u$, and $\LSet = 0$ is used for realisations $u \mapsto f(u,\omega_f)$ of the
firing rate are nonlinear.

We assume that sources of noise are independent, as follows.
\begin{hypothesis}[Independence] The random fields $w$, $f$, $g$, $v$ are mutually
  independent: the event space $\Omega$, $\sigma$-algebra $\calF$, and probability
  measure $\prob$ are given by
  \[
    \Omega = \bigtimes_{\alpha \in \USet} \Omega_\alpha, \qquad 
    \calF = \bigtimes_{\alpha \in \USet} \calF_\alpha, \qquad 
    \prob = \prod_{\alpha \in \USet} \prob_\alpha, \qquad
    \quad  
    \mathbb{U} = 
    \begin{cases}
      \{w,g,v\} & \text{if $\LSet = 1$,} \\[0.5em]
    \{w,f,g,v\} & \text{if $\LSet = 0$.}
  \end{cases}
\]
\end{hypothesis} 

The theory in \cite{avitabileProjectionMethodsNeural2023} casts the neural field problem in operator form as
an ODE on a Banach space $\XSet$ with random data. Following this approach, we deal
concurrently with two common functional setups of this problem, as is seen in the
next hypothesis.

\begin{hypothesis}[Phase space]\label{hyp:phaseSpace} The phase space is either
  $\XSet$ = $C(D)$, the space of continuous functions on $D$ endowed with the supremum
  norm $\| \blank \|_\infty$, or $\XSet = L^2(D)$, the Lebesgue space of
  square-integrable functions on $D$, endowed with the standard Lebesgue norm $\|
  \blank \|_2$. We will compactly write $\XSet \in \{ C(D), L^2(D)\}$.
\end{hypothesis}
 
With reference to \cref{eq:nonlinRNF}, the Banach space $\XSet$ sets the function space
for realisations of $v(\blank,\omega_v)$, the initial condition,
$g(\blank,t,\omega_g)$, the forcing at time $t$, and $u(\blank,t,\omega)$, the
solution at time $t$. Realisations of the synaptic kernel $w(\blank,\blank,\omega_w)$
are bivariate functions, and therefore require a separate function space dependent on
the choice of $\XSet$, denoted by $\WSet(\XSet)$, or simply $\WSet$, and defined as
\[
  \WSet(\XSet) := 
  \begin{cases}
    \displaystyle{
    \Big\{ 
      k \in C(D,L^1(D)) \colon 
      \lim_{h \to 0} \nu(h; k) = 0
    \Big\},
  }
    & \text{if $\XSet = C(D)$,} \\[1em] 
    L^2(D \times D) & \text{if $\XSet = L^2(D)$,} \\
  \end{cases}
\]
in which
\[
  \nu(h;k) = \max_{x,z \in D} \max_{\| x-z \|_2 \leq h } 
      \int_D |k(x, x') - k(z, x') |\, dx', \qquad h \in \RSet_{ \geq  0},
\]
with norm 
\[
  \| k \|_\WSet = 
  \begin{cases}
   \| k \|_{C(D,L^1(D))} & \text{if $\XSet = C(D)$,} \\
   \| k \|_{L^2(D \times D)} & \text{if $\XSet = L^2(D)$.} \\
  \end{cases}
\]

With these preparations \cref{eq:nonlinRNF} is written in operator form, as
as a Cauchy problem on $\XSet$ with random data, as follows:
\begin{equation}\label{eq:NRNFOp}
  \begin{aligned}
  & u'(t , \omega) = N(t, u(t , \omega) , \omega_w,\omega_f,\omega_g), \qquad t \in J,\\
  & u(0 , \omega) = v(\omega_v),
  \end{aligned}
\end{equation}
in which the operator $N$ is defined as
\[
  N \colon J \times \XSet \times \Omega_w \times \Omega_f \times \Omega_g  \to \XSet
  \qquad 
  N(t, u, \omega_w,\omega_f,\omega_g) = -u + W(\omega_w)F(u,\omega_f) + g(t,\omega_g),
\]
and realisations $W(\omega_w)$ are bounded linear operators associated to
the kernel realisations $w(\blank,\blank,\omega_w)$. Specifically one sets
$W(\omega_w) = H(w(\blank ,\blank ,\omega_w))$, where
\begin{equation}\label{eq:HDef}
  H \colon \WSet \to K(\XSet) \subset BL(\XSet), \qquad H(k)(v) = \int_{D}k(\blank,x')v(x') \,dx',
\end{equation}
in which $K(\XSet)$ is the space of compact operators on $\XSet$ to itself, hence
\begin{equation}\label{eq:WOp}
  W \colon \Omega_w \to K(\XSet) \subset BL(\XSet), \qquad  W(\omega_w)v
    = \int_{D}w(\blank,x',\omega_w)v(x') \,dx'.
\end{equation}
Further the realisations $F(\blank,\omega_f)$ come from the firing rate realisations
$f(\blank,\omega_f)$, as follows
\begin{equation}\label{eq:Nemytskii}
  F \colon \XSet \times \Omega_f \to \XSet,
  \qquad  
  F(u,\omega_f)(x) := 
  \begin{cases}
    u(x) & \text{if $\LSet = 1$,} \\
  f(u(x),\omega_f)& \text{if $\LSet = 0$.}
  \end{cases}
\end{equation}
In passing, we note that we often write $N(\blank,\blank,\omega)$, instead of
$N(\blank, \blank, \omega_w,\omega_f,\omega_g)$, for simplicity.

We work with $L^p$-regular input data, as indicated by the next hypothesis.
\begin{hypothesis}[$L^p$-regularity of random data]\label{hyp:randomDataLp} It holds
  that either
  \[
    \LSet = 1, 
    \quad 
    w \in L^\infty(\Omega_w,\WSet), 
    \quad 
    g \in L^p(\Omega_g,C^0(J,\XSet)), 
    \quad 
    v \in L^p(\Omega_v,\XSet),
  \]
  or
  \[
    \LSet = 0, 
    \quad 
    w \in L^p(\Omega_w,\WSet), 
    \quad 
    f \in L^p(\Omega_f,BC^1(\RSet)), 
    \quad 
    g \in L^p(\Omega_g,C^0(J,\XSet)), 
    \quad 
    v \in L^p(\Omega_v,\XSet),
  \]
\end{hypothesis}

Because \cref{hyp:randomDataLp} implies \cite[Hypothesis
3.5]{avitabile2024NeuralFields}, we can use directly~\cite[Theorems 4.2 and
4.5]{avitabile2024NeuralFields} on the well-posedness of \cref{eq:NRNFOp}, and
the measurability and regularity of its solutions.
\begin{theorem}[Abridginng {\cite[Theorems 4.2 and 4.5]{avitabile2024NeuralFields}}
  on solutions to neural fields with random data]
  Under \crefrange{hyp:domain}{hyp:randomDataLp} there exists a $\prob$-almost unique $u \in L^p(\Omega,C^1(J,\XSet))$
  satisfying \cref{eq:NRNFOp} $\prob$-almost surely, that is, satisfying $\prob( B_0 \cap B_J )=1$, with
    \[
        B_0 = \{ \omega \in \Omega \colon u(0, \omega) = v(\omega_v) \}, 
        \qquad  
        B_J = \{ \omega \in \Omega \colon 
            u'(\blank , \omega) = N(\blank , u(\blank , \omega) , \omega_w,\omega_f,\omega_g) \text{ on $J$} 
          \}.
    \]
\end{theorem}

As in the related PDE literature~\cite{
xiuModelingUncertaintySteady2002,
babuskaGalerkinFiniteElement2004a,
babuskaStochasticCollocationMethod2007,
xiuHighOrderCollocationMethods2005, smithUncertaintyQuantificationTheory2014,
Lord:2014ir, adcockSparsePolynomialApproximation2022}, 
we are interested in exploring problems with finite-dimensional random data.
Therefore, we assume that, for instance, the initial
condition $v(x,\omega_v)$ depends on $m_v$ random parameters, which we denote $Y_v$,
and similar for other sources of randomness in the model. We define
finite-dimensional noise as in \cite[Definition 9.38]{Lord:2014ir}, and we state our
hypothesis for finite-dimensional noise in neural fields.
\begin{definition}[$m$-dimensional, $p$th-order, $\BSet$-valued noise] 
  \label{def:finDimNoise}
  Let $m,k \in
  \NSet$, and $\BSet$ be a Banach space. Further, let $\{ Y_k \}$, $k \in \NSet_m$ be
  a collection of $m$ random variables $Y_k \colon \Omega \to \Gamma_k \subset
  \RSet$. A random variable $f \in L^p(\Omega,\BSet)$ of the form
  $f(\blank,Y(\omega))$, where
  $Y = (Y_1,\ldots,Y_m) \colon \Omega \to \Gamma = \Gamma_1 \times \dots \times
  \Gamma_m$, is called an
  $m$-dimensional, $p$th-order, $\BSet$-valued noise. We abbreviate this by saying that
  $f \in L^p(\Omega,\BSet)$ is $m$-dimensional noise.
\end{definition}

\begin{hypothesis}[Finite-dimensional noise random data]\label{hyp:finDimNoise}
  The random data in \cref{hyp:randomDataLp} is finite-dimensional noise of the form
  \begin{align*}
      & 
      w(\blank,\blank,\omega_w) = \tilde w(\blank,\blank,Y_w(\omega_w)),
      && 
      Y_w \colon \Omega_w \to \Gamma_w \subset \RSet^{m_w},
      && 
      Y_w \sim \rho_w,
      \\
      & 
      f(\blank,\omega_f) = \tilde f(\blank,Y_f(\omega_f)),
      && 
      Y_f \colon \Omega_f \to \Gamma_f \subset \RSet^{m_f},
      && 
      Y_f \sim \rho_f,
      \\
      & 
      g(\blank,\blank,\omega_g) = \tilde g(\blank,\blank,Y_g(\omega_g)),
      && 
      Y_g \colon \Omega_g \to \Gamma_g \subset \RSet^{m_g},
      && 
      Y_g \sim \rho_g,
      \\
      & 
      v(\blank,\omega_v) = \tilde v(\blank,Y_v(\omega_v)),
      && 
      Y_v \colon \Omega_v \to \Gamma_v \subset \RSet^{m_v},
      && 
      Y_v \sim \rho_v.
  \end{align*}
\end{hypothesis}


The finite-dimensional noise assumption may hold for affine- or non-affine
parameter dependence~\cite{babuskaStochasticCollocationMethod2007}. We present two
examples for neural fields, which are later used for checking hypotheses in
\cref{sssec:checkAnaliticityOfData} and presenting numerical experiments in
\cref{sec:numerics}.

\begin{example}[Affine parameter dependence] \label{ex:affineRandomData}
  Consider initial conditions in the form of a
  Karhunen-Loeve expansion of the form
\begin{equation}\label{eq:vAnaExample}
  v(x,\omega_v) = b_0(x) + \sum_{j \in \NSet_{m_v}}  b_j(x) Y_j(\omega_v),
\quad Y_j \stackrel{i.i.d}{\sim}\mathcal{N}(0,1), 
\quad j \in \NSet_{m_v}, 
\quad x \in D = [-1,1] ,
\end{equation}
which is $m_v$-dimensional noise with $(y_1,\ldots,y_{m_v}) \in \Gamma_v = \RSet^m$, and in
which the spatial regularity is controlled by $\{b_j\}_{j \in \ZSet_m} \subset
\XSet$. In this case the mapping $y \mapsto \tilde v(\blank, y)$ is affine.
\end{example}

\begin{example}[Non-affine parameter dependence] \label{ex:nonAffineRandomData}
  Consider the random external input
\begin{equation}\label{eq:gAnaExample}
  g(x,t,\omega_g) = e^{t Y(\omega_g)} \bigl[ (Y(\omega_g)+1) \sin(4\pi x) + x/(2 \pi)\bigr]
  \quad  x \in D = [-1,1], \quad t \in J = [0,T],
\end{equation}
with $Y \sim \mathcal{U}[\alpha,\beta]$ or $Y \sim \mathcal{N}(0,\sigma)$. This
forcing is $1$-dimensional noise with $y \in \Gamma_g = [\alpha,\beta] \subset
\RSet$ or $\Gamma_g = \RSet$, and is used in the numerical example in
\cref{ssec:NumericalEx1}.
\end{example}




In addition to finite-dimensionality of the noise, \cref{hyp:finDimNoise} prescribes
that each random variable $ Y_\alpha $ has density $\rho_\alpha$, $\alpha \in
\USet$. It is possible to reformulate the
neural field problem under the finite-dimensional noise assumption: all random
variables are collected in a vector,
\[
  Y \colon \Omega \to \Gamma = \Gamma_1 \times \ldots \times \Gamma_m \subset
  \RSet^m, \qquad m = \sum_{\alpha \in \USet} m_\alpha,  
  \qquad
  Y \sim \rho = \prod_{\alpha \in \USet} \rho_\alpha,
\]
and one can express the neural field problem, in operator form, in terms of a
deterministic function $\hat u(\blank,\blank,y)$, using \textit{deterministic inputs} $\tilde w$,
$\tilde f$, $\tilde g$, and $\tilde v$ in place of the original inputs. We refer
to~\cite[Section 5]{avitabile2024NeuralFields} for the reformulation of the neural field
problem with finite-dimensional noise, which we present below. We highlight that
tildes or hats are henceforth omitted for notational simplicity, as the problem formulation below
can be taken as the starting point for investigating the problem in parametric form.
\begin{problem}[Neural field problem with finite-dimensional
  noise]\label{prob:FiniteDimNRNF}
  Fix $\LSet$, $w$, $g$, $v$, and if $\LSet = 0$, f. Given a density $\rho$ for the
  random variable $Y$, find $u \colon J \times \Gamma \to \XSet$ such that
\begin{equation}\label{eq:FinDimNRNFOp}
  \begin{aligned}
  & u'(t , y) = N(t, u(t , y) , y), 
  & t \in J, \\
  & u(0 , y) = v(y),
    &
  \end{aligned}
  \qquad
  \textrm{$\rho \, dy$-a.e. in $\Gamma$.} \\
\end{equation}
\end{problem}
The result below from \cite{avitabile2024NeuralFields}, addresses well-posedness and
regularity of the parametrised problem. The result is given in terms of Banach
space-valued functions on $(\Gamma, \calB(\RSet^m), \rho dy)$, for instance $v
\in L_\rho(\Gamma,\XSet) := L^p(\Gamma,\calB(\RSet^m),\rho dy, \XSet)$.
\begin{corollary}[Abridged {\cite[Corollary 5.1]{avitabile2024NeuralFields}} on
  $L^p_\rho$-regularity with finite-dimensional noise]
  \label{cor:uLRho}
  Under \crefrange{hyp:domain}{hyp:finDimNoise}, it holds $w \in
  L^\infty_\rho(\Gamma,\WSet)$ (if $\LSet=1)$ or $w \in
  L^p_\rho(\Gamma,\WSet)$ (if $\LSet=0)$, $f \in L^p_\rho(\Gamma,BC(\RSet))$, $g \in
  L^p_\rho(\Gamma, C^0(J,\XSet))$, and $v \in
  L^p_\rho(\Gamma,\XSet)$. Further, \cref{prob:FiniteDimNRNF} has a unique solution $u \in
  L^p_\rho(\Gamma,C^1(J,\XSet))$.
\end{corollary}

\section{Spatial-projection stochastic-collocation scheme} 
\label{sec:numericalScheme}
We aim to construct a numerical approximation to the solution $u$ to
\cref{eq:FinDimNRNFOp}
with well-defined mean and variance and, using \cref{cor:uLRho}, we seek for an
approximation in a subspace of $L^2_\rho(\Gamma,C^1(J,\XSet))$.
We intend to treat the spatial coordinate $x$ and the stochastic
coordinate $y$ differently, so it is
natural to work in $C^1(J,\XSet) \otimes L^2_\rho(\Gamma)$, which is isometrically
isomorphic to the original space $L^2_\rho(\Gamma,C^1(J,\XSet))$ \cite[Sections
7.1--7.2]{defantTensorNormsOperator1992}\footnote{Applying strictly Defant and
  Floret's theory to our case, we should write
$C^1(J,\XSet) \hat \otimes_{\Delta_2} L^2_\rho(\Gamma)$. The subscript $\Delta_2$
indicates that norm in use the tensor product space, as we shall describe below; the
hat indicates the completion of the tensor product space $C^1(J,\XSet) \otimes_{\Delta_2}
L^2_\rho(\Gamma)$. Defant and Floret argue that, owing to the natural injection
\[
  \begin{aligned}
    \iota \colon C^1(J,\XSet) \otimes L^2_\rho(\Gamma) & \to L^2_\rho(\Gamma,C^1(J,\XSet)) \\
    A \otimes \psi &\mapsto A \psi(\blank)
  \end{aligned}
\]
one can place the norm $\| A \otimes \psi \|_{\Delta_2} := 
\Vert \iota (A \otimes \psi) \Vert_{L^2_\rho(\Gamma,C^1(J,\XSet))} = 
\Vert A \psi(\blank) \Vert_{L^2_\rho(\Gamma,C^1(J,\XSet))}$ on elements of the
tensor product space. Now, let $C^1(J,\XSet)
\otimes_{\Delta_2} L^2_\rho(\Gamma)$ be the space $C^1(J,\XSet) \otimes L^2_\rho(\Gamma)$
endowed with the norm $\| \blank \|_{\Delta_2}$. Defant and Floret prove that the
completion $C^1(J,\XSet) \hat \otimes_{\Delta_2} L^2_\rho(\Gamma)$ to
$C^1(J,\XSet) \otimes_{\Delta_2} L^2_\rho(\Gamma)$ is isometrically isomorphic to
$L^2(\Gamma,C^1(J,\XSet))$. In the main text, we write $C^1(J,\XSet) \otimes
L^2_\rho(\Gamma)$ for $C^1(J,\XSet) \hat \otimes_{\Delta_2} L^2_\rho(\Gamma)$.}. 
Since the two spaces can be identified,  a typical element of $u \in C^1(J,\XSet) \otimes
L^2_\rho(\Gamma)$, of the form $u = \sum_{j} A_j \otimes \psi_j = \sum_{j} A_j
\psi_j$, can be measured using the norm on $L^2_\rho(\Gamma,C^1(J,\XSet))$. 

We look for an approximation $u_{n,q}$ in $C^1(J,\XSet_n) \otimes
\calP_q(\Gamma)$, where $\XSet_n \subset \XSet$ and $\calP_q(\Gamma) \subset
L^2_\rho(\Gamma)$ are finite-dimensional subspaces, whose definition will be made
precise later, of dimension $s(n)$ and $d(q)$, respectively. The scheme proceeds in
two steps, which are now introduced in detail: 
\begin{itemize}
  \item Step 1: construct an approximation $u_n(\blank,y)$ in $C^1(J,\XSet_n)$ to $u$ in
    $C^1(J,\XSet)$ by a collocation scheme with grid $\{ x_i \colon i \in
    \NSet_{s(n)}\} \subset D$
    \[
      u'_n(t,y)(x_i) = N(t,u_n(t,y))(x_i),
      \qquad 
      u_n(0,y)(x_i)
      \qquad 
      (t,i,y) \in J \times  \NSet_{s(n)} \times \Gamma
    \]
    or by a Galerkin scheme using test functions $\{ \phi_i \colon i \in
    \NSet_{s(n)}\} \subset \XSet_n$,
    \[
      \bigl\langle \phi_i , 
        u'_n(t,y)- N(t,u_n(t,y)) 
      \bigr\rangle_{\XSet} = 0,
      \qquad 
      \langle \phi_i, u_n(0,y) \rangle_{\XSet} = 0,
      \qquad 
      (t,i,y) \in J \times  \NSet_{s(n)} \times \Gamma.
    \]

  \item Step 2: Collocate one of the schemes in step 1 on the zeros of orthogonal
    polynomials $\{ y_k : k \in \NSet_{d(q)} \}$ and build $u_{n,q}$ using the
    collocated solutions
    \[
      u_{n,q}(t,y) = \sum_{k \in \NSet_{d(q)}} u_{n}(t,y_k) l_k(y),
    \]
    where $\{ l_k \} \subset \calP_q(\Gamma)$ are Lagrange polynomials with nodes
    $\{ y_k \}$.

\end{itemize}
\subsection{Step 1: spatial projection}
Schemes for neural fields of collocation and Galerkin type can be expressed using
projectors ~\cite{avitabileProjectionMethodsNeural2023}. We consider a family of subspaces $\{ \XSet_n \colon n \in \NSet \}
\subset \XSet$, $\XSet_n := \spn \{ \phi_j \colon j \in \NSet_{s(n)} \}$ with $\dim
\XSet_n =: s(n) \to \infty $ as $n \to \infty $ and such that $\overline{ \cup_{n \in
\NSet} \XSet_n} = \XSet$. On each $\XSet_n$ we place the norm $\| \blank \|_{\XSet}$.
A corresponding family of projectors $\{ P_n \colon n \in \NSet\}$ satisfies, for any
$n \in \NSet$, $P_n \in BL(\XSet,\XSet_n)$, $\| P_n \| \geq 1$, and $P_n v = v$ for
all $v \in \XSet_n$~\cite{atkinson2005theoretical}, with
\[
  P_n \colon \XSet \to \XSet_n
  \qquad 
  (P_n v)(x) = \sum_{j \in \NSet_{s(n)}} V_j \phi_j(x)
  \qquad 
  (x, n) \in D \times \NSet.
\]

Interpolatory projectors define collocation schemes: we set
$\XSet = C(D)$, $\phi_j = \ell_j$, the $j$th Lagrange polynomial with nodes $\{ x_j \}
\subset D$, leading to $V_j = v(x_j)$. Orthogonal projectors define Galerkin schemes:
we set $\XSet = L^2(D)$ and $V_j = \langle v, \phi_j
\rangle_{\XSet}$. For details and examples we refer
to~\cite{atkinson1997,Atkinson:1992du,avitabileProjectionMethodsNeural2023}. A
collocation or Galerkin scheme for \cref{eq:FinDimNRNFOp} is written in abstract form
as~\cite{avitabileProjectionMethodsNeural2023}
\begin{equation}
  \label{eq:unEquation}
  \begin{aligned}
  & u_n'(t , y) = P_nN(t, u_n(t , y) , y), 
  & t \in J, \\
  & u_n(0 , y) = P_n v(y),
  \end{aligned}
  \qquad
  \textrm{$\rho \, dy$-a.e. in $\Gamma$.} 
\end{equation}
Since the hypotheses of ~\cite[Corollary 6.2]{avitabile2024NeuralFields} are
satisfied, problem \cref{eq:unEquation} admits a unique solution $u_n \in
L^p_{\rho}(\Gamma,C^1(J,\XSet))$. We conclude this section with a result providing a
bound on the error $u - u_n$ in terms of the error made by the projector on the
random data. In the nonlinear case, such result follows directly from~\cite[Theorems
3.2-3.3]{avitabileProjectionMethodsNeural2023}. The result in the nonlinear case
relies on the boundedness of $f(u)$, which does not hold when $f(u)= u$. We present
below a similar result for the linear case, obtained by a small modification of the
results in~\cite{avitabileProjectionMethodsNeural2023}.

\begin{thmE}\label{thm:determErrorBound}
  If \crefrange{hyp:domain}{hyp:finDimNoise} hold for $\LSet = 1$, then for
  $\rho dy$-almost every $y \in \Gamma$
  \[
  \begin{aligned}
    & \| u(\blank ,y) - P_n u(\blank ,y) \|_{C^0(J,\XSet)} \leq  \alpha_n(y), \\
    & \| u(\blank ,y) - u_n(\blank ,y) \|_{C^0(J,\XSet)} \leq  e^{\beta_n(y_w)} 
            \| u(\blank ,y) - P_n u(\blank ,y) \|_{C^0(J,\XSet)},
  \end{aligned}
  \]
  with $\beta_n(y_w) = T \| P_n W(y_w) \|_{BL(\XSet)}$ and
  \[
    \begin{aligned}
    \alpha_n(y) = 
    \| v(y_v) -  P_n v(y_v)\|_{\XSet} 
    & + T \| W(y_w) - P_n W(y_w) \|_{BL(\XSet)}  \|  u(\blank ,y) \|_{C^0(J,\XSet)}\\
    & + T \| g(\blank, y_g) - P_n g(\blank , y_g)\|_{C^0(J,\XSet)}.
    \end{aligned}
  \]
\end{thmE}
\begin{proofE}
  Proceeding as in ~\cite[Equation 3.8]{avitabileProjectionMethodsNeural2023} with
  $f(u)= u$ we obtain
  \[
    u(t,y) - u_n(t,y) = u(t,y) - P_n u(t,y) 
      + \int_{0}^{T}e^{-(t-s)} P_n W(y_w) [u(s,y) - u_n(s,y)] \,ds.
  \]
  Taking norms we estimate
  \[
    \|  u(t,y) - u_n(t,y) \|_{\XSet} \leq  \|u(t,y) - P_n u(t,y) \|_{\XSet} 
    + \beta_n(y_w) \int_{0}^{T} \|u(s,y) - u_n(s,y)\|_{\XSet} \,ds,
  \]
  and Gr\"onwall inequality leads to
  \[
    \|  u(t,y) - u_n(t,y) \|_{\XSet} \leq  \|u(t,y) - P_n u(t,y) \|_{\XSet} 
    + \beta_n(y_w) \int_{0}^{T} e^{\beta_n(y_w) (t-s)} 
    \|u(s,y) - P_n u(s,y) \|_{\XSet}  \,ds,
  \] 
  hence taking the maximum over $t \in [0,T]$ we get
  \begin{equation}\label{eq:uMinusUn}
    \| u(\blank ,y) - u_n(\blank ,y) \|_{C^0(J,\XSet)}
    \leq 
      e^{\beta_n(y_w)}\| u(\blank ,y) - P_n u(\blank ,y) \|_{C^0(J,\XSet)}.
  \end{equation}
  Further, proceeding again as in~\cite[Proof of Theorem
  3.2]{avitabileProjectionMethodsNeural2023} for $f(u) = u$ we obtain
  \[
    \begin{aligned}
    u(t,y) - P_n u(t,y) = v(y) - P_n v(y_v) 
      & + \int_{0}^{T}e^{-(t-s)} \bigl(W(y_w) - P_nW(y_w)\bigr) u(t,y) \,ds\\
      & + \int_{0}^{T}e^{-(t-s)} g(s,y_g) - P_n g(s,y_g) \,ds,
    \end{aligned}
  \]
  and taking the $\| \blank  \|_{C^0(J,\XSet)}$ norm we estimate
  \[
    \| u_n(\blank ,y ) - P_n u(\blank ,y) \|_{C^0(J,\XSet)} \leq \alpha_n(y).
  \]
\end{proofE}

We conclude this section by recalling an $n$-dependent bound on the projected random
data, which holds because \cref{hyp:randomDataLp} implies~\cite[Hypothesis
3.5]{avitabile2024NeuralFields}, hence the statement is proved in \cite[Proposition
6.1]{avitabile2024NeuralFields}. 
\begin{proposition}[Adapted from {\cite[Proposition 6.1]{avitabile2024NeuralFields}}]\label{prop:kappaEst_Pn}
  Assume \crefrange{hyp:domain}{hyp:randomDataLp} and let
  \[
  \begin{aligned}
    & \kappa_{w,n}(\omega_w) :=
    \| P_n W(\omega_w) \|_{BL(\XSet,\XnSet)}, 
    && \kappa_{w}(\omega_w) :=
    \| w(\omega_w) \|_{\WSet}, 
    \\
    & \kappa_{g,n}(\omega_g) :=
    \| P_n g(\blank,\omega_g) \|_{C^0(J,\XSet)}, 
    && \kappa_{g}(\omega_g) :=
    \| g(\blank,\omega_g) \|_{C^0(J,\XSet)}, 
    \\
    & \kappa_{v,n}(\omega_v) :=
    \| P_n v(\omega_v) \|_{\XSet},
    && \kappa_{v}(\omega_v) :=
    \| v(\omega_v) \|_{\XSet}.
  \end{aligned}
  \]
  For any $n \in \NSet$ and  $\alpha \in \{w,g,v\}$ it holds
  \[
    \kappa_{\alpha,n} \leq \| P_n \| \kappa_\alpha \qquad \text{$\prob_\alpha$-almost surely}.
  \] 
\end{proposition}

\begin{remark}\label{rem:finDiffNoisePn}
\Cref{prop:kappaEst_Pn} can also be
expressed under the finite-dimensional noise assumption. By the finite-dimensional
noise assumption, for any $\alpha \in \{w,g,v\}$ and $n \in \NSet$ there exist
deterministic functions $\tilde \kappa_\alpha, \tilde \kappa_{\alpha,n} \colon \Gamma
\to \RSet_{\geq 0}$ such that $\kappa_\alpha(\omega_\alpha)= \tilde
\kappa_\alpha(Y_\alpha(\omega_\alpha))$ and $\kappa_{\alpha,n}(\omega_\alpha)= \tilde
\kappa_{\alpha,n}(Y_\alpha(\omega_\alpha))$, respectively. For instance it holds
$\kappa_w(\omega_w) = \| P_n v(\omega_v) \|_{\XSet} =
\| P_n \tilde v(Y_v(\omega_w)) \|_{\XSet} =: \tilde \kappa_w(Y_w(\omega_w)) $
and similar identities are true for the other random data. 
\end{remark}

\subsection{Step 2: stochastic collocation} After a spatial projection has been
derived, a stochastic collocation method is constructed using an interpolatory
projector (and hence collocation) in the variable $y = (y_1,\ldots,y_m) \in \Gamma
\subset \RSet^m$. To describe, and later
analyse, this method we follow closely 
\cite{babuskaStochasticCollocationMethod2007}. We aim to construct a basis for the
space $\calP_q(\Gamma) \subset L^2_\rho(\Gamma)$. The function space
$\calP_q(\Gamma)$ is the tensor product of spans of polynomials of degree at most
$q=(q_1,\ldots,q_m)$,
\[
  \calP_q(\Gamma) = \bigotimes_{i \in \NSet_m} \calP_{q_i}(\Gamma_i), 
  \qquad
  \calP_{q_i}(\Gamma_i) = \spn\{ 1, y_i, y_i^2, \ldots, y_i^{q_i}\} = \spn\{y_i^j
  \colon j \in \ZSet_{q_i+1} \},
\]
from which we deduce
\[
  \dim \calP_{q_i}(\Gamma_i) = q_i+1, \quad i \in \NSet_m, \qquad 
  \dim \calP_{q}(\Gamma) = \prod_{i \in \NSet_m} (q_i+1) =: d(q).
\]
A basis for $\calP_{q}(\Gamma)$ is expressed in terms of bases for the spaces
$\calP_{q_i}(\Gamma_i)$. For each $i\in\mathbb{N}_m$ we introduce $q_i+1$ points 
\(
 \{y_{i,0},y_{i,2},...,y_{i,q_i}\} \subset \Gamma_i,
\)
and corresponding Lagrange polynomials
\[
  l_{i,j} \in \calP_{q_i}(\Gamma_i), \qquad l_{i,j}(y_{i,r}) = \delta_{j,r} \qquad
  j,r \in \ZSet_{q_i+1}.
\]
We aim to enumerate collocation points $\{ y_k \}_k \subset \Gamma$ with a single index $k
\in \NSet_{d(q)}$: to any vector of $m$ indices $(k_1, \ldots, k_m)$, where $k_i \in \ZSet_{q_i}$
and $i\in\mathbb{N}_m$, we associate the global index
\[
  k = (k_1+1) + q_1 k_2 + q_1q_2k_3 + \dots + q_1 \cdots q_{m-1}k_m,
\]
and define points $\{  y_k \}_k \subset \Gamma$ by
\[
  y_k := (y_{1,k_1}, \ldots, y_{m,k_m}) \in \Gamma, \qquad k\in\mathbb{N}_{d(q)}.
\]

Since $\calP_q(\Gamma)$ is the tensor product of span of polynomials in
$\calP_{q_i}(\Gamma_i)$ we set\footnote{Note that
  we use $\ell_j$ for the spatial Lagrange basis on $D$, and $l_j$
for the stochastic Lagrange basis on $\Gamma$, because their domain and interpolation
points differ.}
\[
  l_{k}(y) := \prod_{i \in \NSet_m} l_{i,k_i}(y_i).
\]
For a function $v \colon \Gamma \to \RSet$, we consider interpolation operator
\[
  I_q v = \sum_{k \in \NSet_{d(q)}} v(y_k) l_k(y),
\]
and hence the sought approximation $u_{n,q}$ to $u_n$ is of the form
\begin{equation}\label{eq:unqDef}
  u_{n,q}(x,t,y) := (I_q u_n(x,t,\blank))(y) = \sum_{k \in \mathbb{N}_{d(q)}}u_n(x,t,y_k) l_k(y),
\end{equation}
where $u_n(x,t,y_k)$ is the solution to \cref{eq:unEquation} for $y = y_k$. 


\subsection{Discrete scheme}\label{ssec:DiscScheme}
Finally, for an implementable scheme one must select:
\begin{enumerate}
  \item A quadrature scheme to approximate the integrals in the variable $y$. In
    \cite[Section 2.1]{babuskaStochasticCollocationMethod2007} a Gauss quadrature
    scheme associated with the interpolatory projector $I_q$ is discussed. 
  \item A quadrature scheme to approximate the integrals in the variable $x$,
    associated to the spatial projector $P_n$. A discussion on several pairs of
    quadrature schemes and interpolatory or orthogonal projectors is given in
    \cite{avitabileProjectionMethodsNeural2023}.
  \item A time stepper for realisations of the neural field equation. Error
    estimates combining the effects of spatial and temporal discretisations are
    possible (see for instance ~\cite[Section
    5]{avitabileProjectionMethodsNeural2023}).
\end{enumerate}
Selecting quadratures and time steppers lead to \textit{discrete
schemes}. In this paper we present a convergence analysis for the scheme in operator
form, and supporting numerical simulations in which spatial quadrature and
time-stepping errors are negligible or of the same magnitude with respect to the
other sources of errors. 

\subsection{Mean, variance, and considerations on errors}\label{ssec:meanErrors}
Once a pointwise approximation $u_{n,q}$ to $u$ is available, we can approximate
quantities of interests associated to it. For instance the mean is estimated using
\begin{equation}\label{eq:ERhoU}
  (\mean_\rho u)(x,t) = \int_\Gamma u(x,t,y) \rho(y) dy \approx 
  \int_\Gamma u_{n,q}(x,t,y) \rho(y) dy =
  (\mean_\rho u_{n,q})(x,t).
\end{equation}
In passing, we note that $\mean_\rho u$ is the expectation of
$u(\blank,\blank,\omega)$, because
\[
  (\mean u)(x,t) = \int_\Omega u(x,t,\omega)\, d\prob(\omega) 
  = 
  \int_\Gamma \hat u(x,t,y) \rho(y)\, dy = (\mean_\rho \hat u)(x,t)
\]
and we have been omitting the hat in \cref{eq:ERhoU} (and ever since we introduced
them, see discussion after \cref{ex:nonAffineRandomData}). The functions $\mean_\rho u$ and
$\mean_\rho u_{n,q}$ are $C^1(J,\XSet)$ and $C^1(J,\XSet_n)$,
respectively, but for simplicity we derive bounds the error of the numerical scheme
in the $C^0$ norm, namely $\| \mean_\rho (u - u_{n,q}) \|_{C^0(J,\XSet)}$ and 
$\| \vari_\rho u - \vari_\rho u_{n,q} \|_{C^0(J,\XSet)}$. In \cref{ssec:totalError}
we show that both errors are controlled by $\| u - u_{n,q} \|_{C(J,\XSet) \otimes
L^2_\rho(\Gamma)}$, which is thus the focus of the analysis in
\cref{sec:linearAnalysis}.

\section{Error analysis for RLNFs ($\LSet =1$)}\label{sec:linearAnalysis}
We aim to study the convergence error for the scheme in \cref{sec:numericalScheme},
as described in \cref{ssec:meanErrors}. Since $u - u_{n,q} = u - u_n + u_n -
u_{n,q}$, we expect that the error splits into a spatial projection error, and a stochastic
collocation error. For the former we make use of results in
\cite{avitabileProjectionMethodsNeural2023}, from which we inherit convergence rates as $n \to
\infty$. For the latter we prove that $u_n-u_{n,q}$ decays
asymptotically exponentially in $q$ using an argument along the lines of
\cite{Zhang.2012}, which uses \cite[Section 4]{babuskaStochasticCollocationMethod2007} and relies on the analyticity of $u_n$
with respect to the different components of the vector $y \in \Gamma$.

As in \cite{babuskaStochasticCollocationMethod2007,Zhang.2012} we introduce a
suitable weighted space of continuous functions on $\Gamma =
\bigtimes_{i \in  \NSet_m} \Gamma_i$, in which each $\Gamma_i$ can be bounded or
unbounded. 
introduce the weights $\sigma(y)=\prod_{i=1}^m \sigma_i(y_i)$ where, 
\begin{equation}\label{eq:weights}
	\sigma_i(y_i)=
	\begin{cases}
		1  & \textrm{if $\Gamma_i$ is bounded,} \\ 
		e^{-\eta_i|y_i|}  & \textrm{if $\Gamma_i$ is unbounded,} \\
	\end{cases},
  \qquad \eta_i \in \RSet_{ \geq 0}, \qquad i \in \NSet_m,
\end{equation}
we recall \cref{eq:WeightSpace,eq:WeightCkSpace} for the definition of the weighted
function spaces $C^0_\sigma(\Gamma,\BSet)$ and $C^k_\sigma(\Gamma,\BSet)$, respectively.

To establish an error bound of the semidiscrete solution in elliptic problems on
bounded as well as unbounded $\Gamma_i$, the results in
\cite{babuskaStochasticCollocationMethod2007} make use of the
continuous embedding of $C^0_\sigma(\Gamma,\BSet)$ into $L^2_\rho(\Gamma,\BSet)$, for
which it is necessary to demand that the joint density $\rho$ decays sufficiently
fast as $|y| \to \infty$ \cite[page 1015]{babuskaStochasticCollocationMethod2007}. We
use this result in the main theorem at the end of the section,
\cref{thm:totalBound}, but we anticipate here the necessary assumption for that step.
\begin{hypothesis}[Sub-gaussianity of the joint density]\label{hyp:subgaussian}
    The joint probability density $\rho(y)$ satisfies
      \[
        \rho(y) \leq \kappa_\rho \exp \bigg( - \sum_{i \in \NSet_m} (\zeta_i y_i)^2 \bigg),
      \]
      where $\kappa_\rho > 0$, and for any $i \in \NSet_m$ the constant $\zeta_i$ is
      positive if $\Gamma_i$ is unbounded, and null otherwise.
\end{hypothesis}

In the rest of the section we first show that continuity of the data with respect to
the variable $y$ implies continuity of the solution $u_n$. Then, adding further
regularity assumptions on the random fields $w$, $g$ and $v$, we prove analyticity of
$u_n$. The proofs in this
section combine results for neural fields with random data, developed
in~\cite{avitabile2024NeuralFields}, with analyticity results
in~\cite{babuskaStochasticCollocationMethod2007}. 
\subsection{$C^0_\sigma$-regularity of the solution to the RLNF}\label{ssec:C0RegularLRNF}

\begin{lemma}[$C^0_\sigma$-regularity for linear problems with finite-dimensional
  noise] \label{lemma:C0SigmaRegularity} 
  Assume 
  \crefrange{hyp:domain}{hyp:randomDataLp} 
  (general hypotheses) and \cref{hyp:finDimNoise} (finite-dimensional noise) hold for
  $\LSet=1$. If the random data satisfies $w \in
  BC(\Gamma,\WSet)$, $g \in C^0_\sigma(\Gamma,C^0(J,\XSet))$, and $v \in
  C^0_\sigma(\Gamma,\XSet)$, then the solution $u$ to \cref{eq:NRNFOp} is in
  $C^0_\sigma(\Gamma,C^1(J,\XSet))$ and the solution $u_n$ to \cref{eq:unEquation} is
  in $C^0_\sigma(\Gamma,C^1(J,\XSet_n))$.
\end{lemma}
\begin{proof} We prove the statement for the solution $u_n$ to
  \cref{eq:unEquation}, which for $\LSet=1$ reads
  \[
    \begin{aligned}
      & u'_n(t,y) = A_n(y) u_n(t,y) + P_n g(t,y), \qquad t \in [0,T], \\
      & u_n(0,y) = P_n v(y),
    \end{aligned}
    \qquad \textrm{$\rho dy$-a.e. in $\Gamma$,}
  \]
  where $A_n(y) = -\id + P_n W(y)$. All the steps below can be straightforwardly
  adapted for the solution $u$ to \cref{eq:NRNFOp} (and in fact the formal
  substitutions $P_n \mapsto \id$, $\XSet_n \mapsto \XSet$, and $u_n \mapsto u$ work to that effect).
  For any fixed $y \in \Gamma$ and $n \in  \NSet$, the operator $A_n(y)$ is in
  $BL(\XSet_n)$, and therefore it generates the uniformly continuous semigroup $(e^{t
  A_n(y)})_{t\geq0}$ on the Banach space $\XSet_n \subset \XSet$. Further, 
  for any fixed $y\in \Gamma$, $n \in \NSet$, the linear problem above admits a
  unique strong solution \cite[Theorem 13.24]{vanneervenFunctionalAnalysis2022} of the form
  \[
    u_n(t,y) = S_n(t,y)P_n v(y) + \int_{0}^{t} S_n(t-s,y)P_n g(s,y) \,ds, \qquad
    S_n(t,y):= e^{tA_n(y)} \quad t \in J.
  \]
  Further $u_n(t,y)$ is a classical solution because $t \mapsto u_n(t,y)$ is
  differentiable with derivative $A_n(y)u_n(t,y) + P_n g(t,y)$. We thus introduce the
  mapping 
  \begin{equation}\label{eq:psi_nMap}
    \psi_n \colon \Gamma \to C^1(J,\XSet_n), \qquad y \mapsto u_n(\blank,y),
  \end{equation}
  and we aim to prove that, for any $n\in \NSet$, the mapping $\psi_n$
  is in $C^0_\sigma(\Gamma,C^1(J,\XSet_n))$. To achieve this we show that: $\psi_n$ is
  continuous on $\Gamma$ to $C^1(J,\XSet_n)$ (Step 1), and then that $\sigma(y) \| \psi_n(y)
  \|_{C^1(J,\XSet)}$ is bounded on $\Gamma$ (Step 2).
%

  \textit{Step 1: the mapping $\psi_n$ is continuous on $\Gamma$ to
  $C^1(J,\XSet_n)$.} To prove that $\psi_n$ is continuous on $\Gamma$ we show that
  for any  $y \in \Gamma$ and $n \in \NSet$
  \[
    \lim_{y_0 \to y} \| \psi_n(y_0) - \psi_n(y) \|_{C^1(J,\XSet)} = 0.
  \]
  Starting from the bound
  \[
    \| \psi_n(y_0) - \psi_n(y) \|_{C^1(J,\XSet)} \leq (I) + (II) + (III)
  \]
  with
  \[
    \begin{aligned}
      & (I) := \sup_{t \in J}\| u_n(t,y_0) - u_n(t,y) \|_{\XSet}, \\
      & (II) := \sup_{t \in J}\| A_n(y_0) u_n(t,y_0) - A_n(y) u_n(t,y) \|_{\XSet}, \\
      & (III) := \sup_{t \in J}\| P_n g(t,y_0) - P_n g(t,y) \|_{\XSet}.
    \end{aligned}
  \]
  it suffices to prove that, for fixed $y \in \Gamma$ and $n \in \NSet$, it holds $(I), (II),
  (III) \to 0$ as $y_0 \to y$.
  For the first term we estimate
  \begin{align*}
      (I)  & \leq 
      \begin{aligned}[t]
        \sup_{t \in J} \| S_n(t,y_0) &P_n v(y_0)  - S_n(t,y) P_n v(y) \| \\
        & + \sup_{t \in J} \int_{0}^{T}\| S_n(t-s,y_0) P_n g(s,y_0) - S_n(t-s,y) P_n
        g(s,y) \| \,d s
      \end{aligned} \\
           & =: (I_1)+(I_2).
  \end{align*}
  Using \cref{lem:semigroupPert} and the fact that $\|  e^{tA} \| \leq e^{t \| A
  \|}$, if $A \in BL(\XSet)$, we obtain
  \begin{align*}
    (I_1) & \leq
       \bigg( \sup_{t \in J} \| S_n(t,y_0) \|\bigg) \| P_n v(y_0) - P_n v(y) \|_{\XSet} \\
      & \qquad + \bigg( \sup_{t \in J} \| S_n(t,y_0)- S_n(t,y) \|\bigg) \| P_n v(y)\|_{\XSet} \\
      & \leq e^{T \| A_n(y_0) \|} \, \| P_n \| \, \|v(y_0) - v(y) \|_{\XSet} \\
      & \qquad + T \| A_n(y_0) - A_n(y)\| \,  \| P_n v(y)\|_{\XSet} q_{n}(T,y_0,y) ,
  \end{align*}
  with 
  \[
  q_n(T,y_0,y) = \exp\big[ T \| A_n(y_0) - A_n(y) \| + T \min(\| A_n(y_0) \|,\| A_n(y)\|) \big] 
  \]
  To show $A_n \in C^0(\Gamma,BL(\XSet_n)$ we note that $W(y) =
  H(w(\blank,\blank,y))$, where $w(\blank,\blank,y) \in \WSet$ is the prescribed
  synaptic kernel and
  \[
    H \colon \WSet \to K(\XSet) \subset BL(\XSet), \qquad H(k)(v) = \int_{D}k(\blank,x')v(x') \,dx',
  \]
  in which $K(\XSet)$ is the space of compact operators on $\XSet$ to itself. Since
  $w \in C^0_{\sigma}(\Gamma,\WSet)$ then $w
  \in C^0(\Gamma_w,\WSet)$. Further, the operator $H$ is continuous
  by~\cite[Proposition 3.1]{avitabile2024NeuralFields}. Therefore, the composition
  $W(y) = H(w(\blank,\blank,y))$ satisfies $W \in C^0(\Gamma_w,H(\WSet))$.
  This implies $A_n \in C^0(\Gamma_w,BL(\XSet_n))$, because $\| A_n(y_0) - A_n(y) \|
  \leq  \| P_n\| \| W(y_0) - W(y) \| \to 0$ as $y_0 \to y$. We use the continuity of $A_n$, and the hypothesis
  $v \in C^0(\Gamma_v,\XSet)$ to conclude $(I_1) \to 0$ as $y_0 \to y$.
  Further
  \begin{align*}
    (I_2) & \leq
       \sup_{t \in J} \int_{0}^{T} \| S_n(t-s,y_0) \| \| P_n g(s,y_0) - P_n g(s,y)
       \|_{\XSet}\,d s \\
          & \qquad  
        + \sup_{t \in J} \int_{0}^{T} \| S_n(t-s,y_0) - S_n(t-s,y) \| \| P_n g(s,y)
        \|_{\XSet}\,d s \\
          & \leq 
          T e^{ T \| A_n(y_0) \| } \, \| P_n \| \, \| g(\blank,y_0) -g(\blank,y) \|_{C^0(J,\XSet)} \\
          & \qquad  
          + T^2 \| A_n(y_0) - A_n(y) \| 
          \| P_n \| \, \| g(\blank,y) \|_{C^0(J,\XSet)}
        q_n(T,y_0,y), 
  \end{align*}
  hence the hypothesis $g \in C^0_\sigma(\Gamma,C^0(J,\XSet)) \subset C^0(\Gamma,C^0(J,\XSet))$ and the continuity of $A_n$
  give $(I_2) \to 0$ as $y_0 \to y$. 

  So far we have proved that $(I) \to 0$ as $y_0 \to y$, that is, $\psi_n \in
  C^0(\Gamma,C^0(J,\XSet_n))$, which in turn implies $(II) \to 0 $ as $y_0 \to y$,
  owing to
  \begin{align*}
    (II) \leq & \| A_n(y_0) \| \, \| \psi_n(y_0) - \psi_n(y) \|_{C^0(J,\XSet)} \\
              & + \| A_n(y_0) - A_n(y)\| \, \| \psi_n(y) \|_{C^0(J,\XSet)}.
  \end{align*}

  Finally, the hypothesis $g \in C_\sigma^0(J,C^0(J,\XSet))$ implies $(III) \to 0$ as
  $y_0
  \to y$, because
  \[
    (III) \leq \| P_n \| \| g(\blank,y_0) - g(\blank,y) \|_{C^0(J,\XSet)}.
  \]
  
  \textit{Step 2: boundedness of $\sigma(y) \| \psi_n(y) \|_{C^1(J,\XSet)}$.} To
  estimate the weighted norm of $\psi_n(y)$ we must estimate first 
  $\| \psi_n(y) \|_{C^1(J,\XSet)} = \| u_n(\blank,y) \|_{C^1(J,\XSet)}$. Recall that,
  to be precise, we should write $\| \psi_n(y) \|_{C^1(J,\XSet)} = \| \tilde
  u_n(\blank,y) \|_{C^1(J,\XSet)}$, to reflect the dependence on $y$. While we
  normally drop the tilde, we reinstate it in this proof because we use at the same
  time functions dependent on $y$ and on $\omega$, for which a different symbol is needed.
  By the finite-dimensional noise assumption, \cref{rem:finDiffNoisePn} and~\cite[Theorem
  6.3, see also Theorem 4.2]{avitabile2024NeuralFields} it holds
  \begin{equation}\label{eq:psinBound}
      \| \tilde u_n(\blank,Y(\omega)) \|_{C^r(J,\XSet)} 
      = \| u_n(t,\omega) \|_{C^r(J,\XSet)} \leq M_{r,n}(\omega)
      =: \tilde M_{r,n}(Y(\omega)),
      \qquad r \in \{ 0,1 \}
  \end{equation}
  with
 \[
   \begin{aligned}
     & \tilde M_{0,n}(y) = 
     \big( \tilde \kappa_{v,n}(y_v) + \tilde \kappa_{g,n}(y_g) T \big) \exp( \kappa_{w,n}(y_w)T), \\
     & \tilde M_{1,n}(y) = \tilde \kappa_{g,n} (y_g) + (2 + \tilde\kappa_{w,n}(y_w)) \tilde M_{0,n}(y).
   \end{aligned}
 \]
 Using~\cite[Proposition 6.1]{avitabile2024NeuralFields} and the hypotheses $v \in
 C^0_\sigma(\Gamma,\XSet)$, $g \in C^0_\sigma(\Gamma,C^0(J,\XSet))$, and $w \in BC(\Gamma, \WSet))$, we
 obtain the bounds
 \begin{equation}\label{eq:MTildeBounds}
   \begin{aligned}
     \| \tilde M_{0,n} \|_{C^0_\sigma(\Gamma,\RSet)}  
       & \leq \| P_n \| \exp\Big( T\| P_n \| \, \| w \|_{BC(\Gamma,\WSet)}\Big)
           \Big( \| v \|_{C^0_\sigma(\Gamma,\XSet)} 
           + T \| g \|_{C^0_\sigma(\Gamma,C^0(J,\XSet))} \Big) < \infty   \\
     \| \tilde M_{1,n} \|_{C^0_\sigma(\Gamma,\RSet)}  
       & \leq \| P_n \| \, \| g \|_{C^0_\sigma(\Gamma,C^0(J,\XSet))} +  \Big(2 + \| P_n \| \, 
         \| w \|_{BC(\Gamma,\WSet)} \Big)\| \tilde M_{0,n} \|_{C^0_\sigma(\Gamma,\RSet)}  < \infty 
    \end{aligned}
 \end{equation}

 Combining \cref{eq:psinBound,eq:MTildeBounds} we now estimate
 \[
   \begin{aligned}
   \| \psi_n \|_{C^0_\sigma(\Gamma,C^1(J,\XSet))} 
     & = \sup_{y \in \Gamma} \sigma(y) \| \tilde u_n(y) \|_{C^1(J,\XSet)} \\
     & \leq \sup_{y \in \Gamma} \sigma(y) \tilde M_{1,n} (y) = \| \tilde M_{1,n} \|_{C^0(\Gamma,C^1(J,\XSet))}
     < \infty. 
   \end{aligned}
 \]
\end{proof}

\subsection{Analyticity hypotheses for the projected random
data}\label{ssec:AnalyticytyHypotheses} 
We can now work towards establishing the analyticity of $u_n$ with respect to each
variable $y_i$: we aim to find an analytic extension of $u_n$, as a function of
$y_i$, onto some 
open set of the complex plane containing $\Gamma_i$, namely
\begin{equation}\label{eq:SigmaDef}
  \Sigma(\Gamma_i,\tau_i)=\{z\in\mathbb{C}: \dist(\Gamma_i,z)< \tau_i\}, \qquad \tau_i > 0,
\end{equation}
and, by definition, the extension will coincide with $u_n$ on $\Gamma_i$. 

\begin{remark}[Analyticity radius]
The radius $\tau_i$ is important in the error estimates that follows, as it determines 
the decay rate of the stochastic collocation term. We aim to control this term with a constant 
independent of the spatial projection (namely independent of $n$).
However, the nature of our problem significantly differs
from the ones studied in \cite{babuskaStochasticCollocationMethod2007,Zhang.2012}: a
priori we can not rely on  properties of parabolic or elliptic differential
operators. The radius of analyticity we find shows a dependence on time that seems
natural to expect in the linear (and nonlinear) neural field problem. 
\end{remark}

We recall that $\Gamma=\prod_{j=1}^m \Gamma_j$ where
$m=\sum_{\alpha\in\mathbb{U}}m_{\alpha}$, therefore each vector $y_\alpha\in \Gamma_\alpha$
has exactly $m_\alpha$ components. We prove analyticity of $u_n$ with respect to each variable $y_i \in \Gamma_i$, similarly to what was done in \cite[Section 4.1]{Zhang.2012}, of which we 
adopt the notation
\[
  \Gamma_i^*=\prod_{j=1,j\neq i}^m \Gamma_j,
  \qquad
  \sigma_i^*=\prod_{j=1,j\neq i}^m \sigma_j,
\]
where elements of $ \Gamma_i^*$ are denoted by $y_i^*$. Moreover, we denote
$C^0_{\sigma^*}( \Gamma_i^*, \BSet)$ the function space defined analogously to
$C^0_\sigma(\Gamma, \BSet)$ in \cref{eq:WeightSpace}, with weights
$\sigma_i^*(y_i^*)=\prod_{j=1,j\neq i}^{j=m}\sigma_j(y_j)$.

With a slight abuse of notation we rewrite the semidiscrete solution $u_n(\blank,y)$,
$y \in \Gamma$, as $u_n(\blank,y_i,y_i^*)$, with $(y_i,y_i^*) \in \Gamma_i \times
\Gamma^*_i$ for some $i \in \NSet_m$. 
For each $y_i\in\Gamma_i\subset \mathbb{R}$ we prove that the semidiscrete
solution $u_n(t,y_i,y_i^*)$ seen as a function of $y_i$, namely $u_n: \Gamma_i\to
C_{\sigma^*_i}^0(\Gamma_i^*,C^1(J,\XSet))$, admits an analytic extension in
$\Sigma(\Gamma_i,\tau_i)$, for some $\tau_i>0$.
In order to prove analyticity of $u_n$ one must require sufficient regularity of the
random data as detailed in the next assumption.

\begin{hypothesis}[Analyticity of the projected random data] 
  \label{hyp:analyticityLRNF_Pn} 
  The random data satisfies $w \in BC(\Gamma,\WSet)$, $g \in C^0_\sigma(\Gamma,C^0(J,\XSet)))$, and
  $v \in C^0_\sigma(\Gamma,\XSet)$. Further, there exist a positive constant $\bar
  K_w$ and functions $K_g, K_v \in C^0_\sigma(\Gamma,\RSet_{\geq 0})$ such that
  for any $i \in  \NSet_m$ there exists $\gamma_i > 0$ satisfying
  \[
    \begin{aligned}
      & \|  P_n \partial_{y_i}^k w(\blank, \blank, y) \|_{\WSet} \leq 
            \bar K_w \, k!\, \gamma_i^k, 
      && (y,k,n) \in \Gamma \times \ZSet_{\geq 0} \times \NSet, \\
      & \|  P_n \partial_{y_i}^k g(\blank, y) \|_{C^0(J,\XSet)} 
             \leq K_g(y)\, k!\, \gamma_i^k,  
      && (y,k,n) \in \Gamma \times \ZSet_{\geq 0} \times \NSet, \\
      & \|  P_n \partial_{y_i}^k v(y) \|_{\XSet} 
             \leq K_v(y)\, k!\, \gamma_i^k,  
      && (y,k,n) \in \Gamma \times \ZSet_{\geq 0} \times \NSet.
    \end{aligned}
  \]
\end{hypothesis}

\begin{remark}[Checking \cref{hyp:analyticityLRNF_Pn}]
\Cref{hyp:analyticityLRNF_Pn} requires that $\| P_n w(\blank , \blank ,y) \|_{\WSet}$
is bounded by $\bar K_w$ homogeneously in the spatial discretisation parameter $n$
(case $k =0$), and that similar bounds hold for $P_n g(\blank ,y)$ and $P_n v(y) $ in
the respective norms. The independence on $n$ guarantees analyticity radii that are
also discretisation independent. Because the bounds must be independent on $n$,
\cref{prop:kappaEst_Pn} is not useful to check \cref{hyp:analyticityLRNF_Pn}.
Further, the hypothesis requires the existence of partial derivatives of the
projected input data with respect to $y_i$, and that \(P_n
\partial^k_{y_i}g(\blank,y) \in C^0(J,\XSet)\), \(\partial^k_{y_i}w(\blank,\blank,y)
\in \WSet\), and \(P_n \partial^k_{y_i}v(y) \in \XSet\) with bounds independent on
$n$. \Cref{hyp:analyticityLRNF_Pn} is different from the analyticity hypotheses
in existing literature on
PDEs~\cite{babuskaStochasticCollocationMethod2007,Zhang.2012}, which concern the
unprojected random data (hence with left-hand sides that are independent on $n$), and
provide explicit choices for the bounding constants. We refer to
\cref{sec:introduction} for a justification of this choice.
We now discuss how to check \cref{hyp:analyticityLRNF_Pn}: the upcoming
\cref{lem:ProjAnalyticity} demonstrates that if the family of projectors satisfies
\(P_n z \to z\) for every \(z \in \XSet\), then \cref{hyp:analyticityLRNF_Pn},
on the projected random data, can be replaced by the more straightforward
\cref{hyp:analyticityLRNF} on the unprojected data, as in the PDE case. In
\cref{sssec:checkAnaliticityOfData} we give working examples on how to check
\cref{hyp:analyticityLRNF_Pn} in cases where $P_n z \not\to z$ for some $z \in \XSet$.
\end{remark}

\begin{hypothesis}[Analyticity of random data] \label{hyp:analyticityLRNF}  
  The random data satisfies $w \in BC(\Gamma,\WSet)$, $g \in
  C^0_\sigma(\Gamma,C^0(J,\XSet)))$, and $v \in C^0_\sigma(\Gamma,\XSet)$. Further,
  for any $i \in \NSet_m$ there exists a positive constant $\gamma_i > 0$ such that 
  \[
  \begin{aligned}
    &
     \frac{ \|\partial^k_{y_i}w(\cdot,\cdot,y)\|_{\WSet} }{1+\|w(\cdot,\cdot,y) \|_{\WSet}} 
     \leq k! \gamma_i^k, 
    && (y,k) \in \Gamma \times \NSet, \\
    &
     \frac{\|\partial^k_{y_i}g(\blank,y)\|_{C^0(J,\XSet)}}{ 1+\|g(\blank,y) \|_{C^0(J,\XSet)}} 
     \leq k!\, \gamma_i^k, 
    && (y,k) \in \Gamma \times \NSet, \\
    \label{eq:analyticityV}
    &
     \frac{\| \partial^k_{y_i}v(y)\|_{\XSet}}{ 1+\|v(y) \|_{\XSet}} \leq k!\, \gamma_i^k. 
    && (y,k) \in \Gamma \times \NSet.
  \end{aligned}
\]
\end{hypothesis}

\begin{lemma}\label{lem:ProjAnalyticity} 
  Assume \cref{hyp:analyticityLRNF}. If $P_n z \to z $ in $\XSet$ as $n \to \infty $
  for all $z \in \XSet$, then \cref{hyp:analyticityLRNF_Pn} holds.
\end{lemma}
\begin{proof}
  If $P_n z \to z$ for all $z$, then by the Principle of Uniform
  Boundendess~\cite[Theorem 2.4.4]{atkinson2005theoretical} it holds 
$M := \sup_{n \in \NSet} \|  P_n \| < \infty$. We fix $i \in \NSet_m$, derive
\[
  \|  P_n v(y) \|_{\XSet} \leq \|  P_n \| \| v(y) \|_{\XSet}
  \leq M \| v(y) \|_{\XSet} \leq M (1 + \| v(y) \|_{\XSet})
  \qquad 
  \textrm{(bound for $k = 0$),}
\]
and use
\cref{eq:analyticityV} to estimate, for any $(y,k,n) \in \Gamma \times \NSet \times
\NSet$ 
\[
  \| P_n \partial^k_{y_i}v(y)\|_{\XSet} \leq 
  \| P_n \| \, \| \partial^k_{y_i}v(y)\|_{\XSet} \leq 
  M  k! \, \gamma_i^k (1+\|v(y) \|_{\XSet})
  \qquad 
  \textrm{(bound for $k > 0$).}
\]
Since $v \in C^0_\sigma(\Gamma,\XSet)$ then, $y \mapsto M(1 + \| v(y) \|_{\XSet}$ is
in $C^0_\sigma(\Gamma)$, hence the bound on $v$ in \cref{hyp:analyticityLRNF_Pn} holds
with $K_v(y) = M(1 + \| v(y) \|_{\XSet})$. In a similar way we find bounds for
the derivatives of $w$ and $g$, with a suitable norm change.
\end{proof}

\subsubsection{Checking analyticity hypothesis on the projected random
data}\label{sssec:checkAnaliticityOfData} We give examples on checking 
\cref{hyp:analyticityLRNF_Pn,hyp:analyticityLRNF} for two types of random data,
previously introduced in \cref{ex:affineRandomData,ex:nonAffineRandomData}



\begin{example}[Checking the analyticity of \cref{eq:vAnaExample} under a
  Finite-Element collocation scheme] \label{ex:FEMColl}
  Assume we want to check \cref{hyp:analyticityLRNF_Pn} for
  \cref{eq:vAnaExample} when the spatial discretisation is done with the finite
  element collocation method in \cite[Section
  4.1.1]{avitabileProjectionMethodsNeural2023}, for which $\XSet = \bigl( C(D), \|
  \blank \|_{\infty}\bigr)$, and $P_n$ is an interpolatory projector at nodes $\{ x_j
\}_{j \in \ZSet_{n}}$, with $\XSet_n = \spn \{ \ell_0, \cdots, \ell_n\}$, where $\{
\ell_j \}$ are the classical tent (piecewise linear) functions. In this case
$P_n z \to z$ as $n \to \infty $ for all $z \in \XSet$ (see \cite[Equation
3.2.39]{atkinson1997} and \cite[Section 3.2.3]{atkinson2005theoretical}), therefore
we can check \cref{hyp:analyticityLRNF} and apply \cref{lem:ProjAnalyticity}. It
holds 
\[
  \begin{aligned}
  & \| v(y) \|_{\XSet} \leq 
    \| b_0 \|_{\XSet} + \sum_{i=1}^{m}  \, |y_i|\,  \| b_i \|_{\XSet} 
    < \infty, \\
  & \| v \|_{C^0_\sigma(\Gamma,\XSet)} \leq 
    \sup_{y \in \Gamma} 
    \biggl[
    \exp\biggl(-\sum_{i =1}^m \sigma_i |y_i|\biggr)
    \biggl(
    \| b_0 \|_{\XSet} + \sum_{i=1}^{m}  \, |y_i|\,  \| b_i \|_{\XSet} 
    \biggr)
    \biggr]
    < \infty,
  \end{aligned}
\]
and $y \mapsto v(y)$ is continuous on $\Gamma$ to $\XSet$, hence $v \in
C^0_\sigma(\Gamma,\XSet)$. Further, we compute
\[
  \partial_{y_i}^{k} v(x,y)
  =
  \begin{cases}
     b_i(x) & k = 1, \\
                        0   & k > 1,
  \end{cases}
  \qquad 
  \frac{\|  \partial_{y_i}^{k} v(y) \|_{\XSet}}{1 + \| v(y) \|_{\XSet}}
  \leq 
  \begin{cases}
    \| b_i\|_{\XSet}  & k = 1, \\
                        0   & k > 1,
  \end{cases}
  \qquad 
  (y,i) \in  \Gamma \times \NSet_m
\]
hence \cref{hyp:analyticityLRNF_Pn} holds with $\gamma_i = \| b_i \|_{\XSet}$ for all
$i \in \NSet_m$.
\end{example}
\begin{example}[Checking analyticity of \cref{eq:vAnaExample} with Finite-Element Galerkin
  scheme] We now wish to modify \cref{ex:FEMColl} in the context of a Finite-Element
  Galerkin scheme. In this case $\XSet = \bigl( L^2(-1,1), \| \blank \|_{L^2(-1,1)}\bigr)$ and $P_n$
  is an orthogonal projector on $\XSet$ to $\XSet_n = \spn \{ \ell_0, \ldots,\ell_n
  \}$, where $\{ \ell_j \}$ are defined as in \cref{ex:FEMColl}, but are considered
  as elements in $L^2(-1,1)$ (see \cite[Section 4.2.1, and references
  therein]{avitabileProjectionMethodsNeural2023}). In this case 
   we have
  $P_n z \to z$ for all $z \in \XSet$ (owing to $\| P_n \|=1$) and \cref{hyp:analyticityLRNF}
  is checked as in \cref{ex:FEMColl}, with $\| \blank \|_\XSet = \| \blank
  \|_{L^2(-1,1)}$.
\end{example}
\begin{example}[Checking analyticity of \cref{eq:vAnaExample} with Spectral Collocation
  method]\label{ex:ChebColl}
  Assume we are back to the functional setup of \cref{ex:FEMColl}, hence
  $\XSet = C([-1,1])$, but we use a Lagrange interpolating polynomial with Chebyshev
  node distribution (Chebyshev interpolant), which is spectrally convergent for
  neural fields with sufficiently regular data \cite[Section 4.1.2]{avitabileProjectionMethodsNeural2023}. For this
  scheme $\| P_n \| \in O(\log n)$ and hence there exists $z \in  \XSet$ for which
  $P_n z$ diverges as $n \to \infty $. In deterministic problems this difficulty is
  overcome by noting that convergence of $P_n z$ is restored for all functions with
  sufficiently strong regularity (typically H\"older continuity, see for instance
  \cite{atkinson1997,atkinson2005theoretical}). In this example, we proceed to verify
  \cref{hyp:analyticityLRNF_Pn} directly by using a similar idea.  
  We outline a strategy that is valid for generic
  functions $v(y)$ before using it for the specific choice \cref{eq:vAnaExample},
  with $\{b_j\}_j \in C^r(D)$. The
  main idea is to prove that \cref{hyp:analyticityLRNF_Pn} holds provided
  \cref{hyp:analyticityLRNF} holds for $v(x,y)$ as well as for its $r$th partial
  derivative with respect to $x$, $\partial_{x}^{r} v(x,y)$, for some $r \in \NSet$. More precisely we
  require:
  \begin{remunerate}
    \item[H1]
      \Cref{hyp:analyticityLRNF} holds with constants $\gamma_i = \mu_{0,i}$, for all
      $i \in \NSet_m$.
    \item[H2] It holds $v \in C^k_\sigma(\Gamma,C^r(D))$ for some $r \in \NSet$ and
      all $k \in \ZSet_{\geq 0}$.
    \item[H3] For any $i \in  \NSet_m$ there exists a positive constant $\mu_{r,i}$ such
      that
      \[
       \frac{\| \partial_{y_i}^{k}\partial_{x}^{r} v(y) \|_{\XSet}}%
            { 1 + \| \partial_{x}^{r} v(y) \|_{\XSet},} 
       \leq k! \, \mu_{r,i}^k,
       \qquad 
       (y,k) \in \Gamma \times \NSet.
      \]
  \end{remunerate}
  Applying Jackson's theorem \cite[Theorem
  3.7.2]{atkinson2005theoretical} to the function $x \mapsto \partial_{y_i}^{k}
  v(x,y)$ we estimate\footnote{The referenced Jackson's theorem is given
    for a function $f$ whose $r$th derivative is $\alpha$-H\"older continuous, $f \in
    C^{r,\alpha}([-1,1])$ for some $\alpha \in (0,1]$. Here we fixed $y$, set $f(x) =
    \partial_{y_i}^{k}v(x,y)$, assumed $f \in C^r([-1,1])$, and applied the
    theorem with $C^{r-1,1}([-1,1])$ and H\"older constant $M = \max_{x \in [-1,1]}
    |f^{(r)}(x)|$.} 
  \[
    \bigl\| P_n \partial_{y_i}^{k} v(\blank, y) 
          - \partial_{y_i}^{k} v(\blank, y)\bigr\|_{\XSet}
    \leq d_{r-1,n} c^r \frac{M_{k,r,i}(y)}{n^r} 
    \leq \delta_{r-1} c^r \frac{M_{k,r,i}(y)}{n^r} 
  \]
  in which
  \[
    d_{r-1,n} := \frac{n^r}{n(n-1)\ldots(n-r+1)},
    \qquad 
    c := 1 + \frac{\pi^2}{2},
    \qquad 
    M_{r,k,i} (y) := \| \partial_{y_i}^{k}\partial_{x}^{r} v(\blank, y) \|_{\XSet},
  \]
  and where we have used the fact that $d_{r-1,n} \to 1$ as $n \to \infty $, hence
  $\delta_{r-1} := \sup_{n \in \NSet} d_{r-1,n}$ is well defined. We now conclude that:
  \[
    \bigl\| P_n \partial_{y_i}^{k} v(\blank, y) 
          - \partial_{y_i}^{k} v(\blank, y)\bigr\|_{\XSet}
    \to 0
    \qquad 
    \textrm{as $n \to \infty$}
    \qquad 
    \textrm{for all $(y,k,i) \in \Gamma \times  \ZSet_{\geq 0} \times \NSet_m$},
  \]
  and
  \[
    \bigl\| P_n \partial_{y_i}^{k} v(\blank, y) 
          - \partial_{y_i}^{k} v(\blank, y)\bigr\|_{\XSet}
    \leq \delta_{r-1} c^r M_{k,r,i}(y)
    \qquad 
    \textrm{for all $(y,k,i,n) \in \Gamma \times  \ZSet_{\geq 0} \times \NSet_m
    \times \NSet$}.
  \]
  In passing, we note that the latter bound is homogeneous in $n$, as desired. We now
  use H1 and H3 to estimate
  \[
    \begin{aligned}
    \bigl\| P_n \partial_{y_i}^{k} v(\blank, y) \bigl\|_{\XSet}
    &
    \leq 
    \bigl\| P_n \partial_{y_i}^{k} v(\blank, y) 
          - \partial_{y_i}^{k} v(\blank, y)\bigr\|_{\XSet}
    + \bigl\| \partial_{y_i}^{k} v(\blank, y)\bigr\|_{\XSet} \\
    &
    \leq 
    \delta_{r-1} c^r 
    \bigl\| \partial_{y_i}^{k} \partial_{x}^{r} v(\blank, y)\bigr\|_{\XSet} 
    + \bigl\| \partial_{y_i}^{k} v(\blank, y)\bigr\|_{\XSet} \\
    &
    \leq 
    \delta_{r-1} c^r 
    \bigl(1 + \| \partial_{x}^{r} v(y) \|_{\XSet} \bigr) k!\, \mu_{r,i}^k
    +
    \bigl(1 + \| v(y) \|_{\XSet} \bigr) k! \, \mu_{0,i}^k \\
    &
    \leq 
    \bigl[
    \delta_{r-1} c^r 
    \bigl(1 + \| \partial_{x}^{r} v(y) \|_{\XSet} \bigr)
    +
    \bigl(1 + \| v(y) \|_{\XSet} \bigr)
    \bigr]
    k! \, \bigl[ \max(\mu_{r,i}, \mu_{0,i}) \bigr]^k \\
    &
    =: K_v(y) \, k! \, \gamma_i^k,
    \qquad (y,k,n) \in \Gamma \times \ZSet_{\geq 0} \times \NSet.
    \end{aligned}
  \]
  Further $K_v \in C^0(\Gamma,\RSet_{\geq 0})$ because $ \partial_{x}^{r} v, v
  \in C^0_\sigma(\Gamma,\XSet)$ by H2, hence \cref{hyp:analyticityLRNF_Pn} holds. It
  now remains to show that $v$ satisfies H1--H3. In fact, H1 has been verified in
  \cref{ex:FEMColl}. We further compute
\[
  \partial_{y_i}^{k} \partial_{x}^{r} v(x,y)
  =
  \begin{cases}
    \displaystyle{\frac{d^r b_i}{dx^r}(x)} & k = 1, \\
                        0   & k > 1,
  \end{cases}
  \qquad 
  \frac{\|  \partial_{y_i}^{k} \partial_x^r v(y) \|_{\XSet}}%
       {1 + \| \partial_x^r v(y) \|_{\XSet}}
  \leq 
  \begin{cases}
    \displaystyle{
    \biggl\| \frac{d^r b_i}{dx^r} \biggr\|_{\XSet}  
    }
        & k = 1, \\
    0   & k > 1,
  \end{cases}
  \quad 
  (y,i) \in  \Gamma \times \NSet_m.
\]
Recalling that $\| b \|_{C^r(D)} := \| b \|_{C(D)}+\sum_{j=1}^r \| b^{(j)} \|_{C^j(D)}$, 
we observe that $\| b^{(r)} \|_{\XSet} \leq  \| b \|_{C(D)}$ for all $r$.
Hence H2 holds, and H3 holds with $\mu_{r,i} = \| b_i \|_{C^r(D)}$.
\end{example}

\begin{example}[Checking analyticity of \cref{eq:gAnaExample}]\label{ex:gAnaExample}
  Finally, we discuss how to check analyticity for the non-affine case
  \cref{eq:gAnaExample}. In this case
  $\XSet=C(D)$ with $D = [-1,1]$, and we note that the mapping of
  interest, $y \mapsto g(\blank, \blank, y)$, is on $\Gamma \subseteq
  \RSet$ to $C^0(J,\XSet)$,
  and hence the relevant norms are on $C^0(J,\XSet)$ as opposed to $\XSet$ as in the
  previous examples. The strategy presented in \cref{ex:ChebColl} does not rely on the boundedness of
  $\Gamma$, so that it can be used also when $Y$ is a Gaussian random variable. We
  set $a = 4 \pi$, $b = (2 \pi)^{-1}$, compute
  \[
    g(x,t,y) = e^{ty} \bigl[ (y + 1) \sin(a x) + b x \bigr],
    \qquad 
    \partial_{y}^{k}g(x,t,y) = t^{k-1} e^{ty} \bigl[ (ty + t + k) \sin(a x) + b t x\bigr],
  \]
  and, using \cref{eq:weights}, we deduce $g \in C^0_{\sigma}(\Gamma,C^0(J,\XSet))$ for
  any $\eta$ (if $Y\sim \mathcal{U}[\alpha,\beta]$) or for any $\eta > T$ (if $Y \sim
  \mathcal{N}(0,\beta)$), where $\eta$ is given in \cref{eq:weights}. We let $G_k(y) = \| \partial_{y}^{k} g(\blank,y)\|_{C^0(J,\XSet)}$, for $k \in
  \ZSet_{\geq 0}$, compute 
  \[
    G_0(y) = e^{T y} A(y), \qquad A(y) = \max_{x \in [-1,1]} |(y+1) \sin(ax) + bx|,
  \]
  and estimate
  \[
    \begin{aligned}
    \frac{G_k(y)}{1+G_0(y)} 
    & \leq \frac{T^k e^{Ty} A(y) + T^{k-1} k}{1 + e^{Ty}A(y)}
      \leq T^k\frac{A(y)}{e^{-Ty} + A(y)} + T^{k-1}\frac{k}{e^{-Ty} + A(y)} \\
    & \leq T^k +T^{k-1}k \leq (1+T)^k (1+k).
    \end{aligned}
  \]
  In passing we note that we used the lower bound $e^{-Ty} +A(y) \geq 1$ because for $y
  \leq  0$ it holds $e^{-Ty}+ A(y) \geq e^{-Ty} \geq 1$, and for $y \geq 0$ it holds
  $e^{-Ty}+ A(y) \geq A(y) \geq 1-b$ by the reverse triangle inequality. Also, it was
  necessary to bound $T^k$ and $T^{k-1}$ by $(1+T)^k$ because $T$ can be smaller than
  $1$ and hence it is not true, in general, that $T^{k-1} < T^k$. We now aim to find a
  constant $\gamma$ satisfying
  \[
    (1+T)^k (1+k) \leq \gamma^k k! \qquad k \in \NSet,
  \]
  so that
  \[
     \frac{\|\partial^k_{y_i}g(\blank,y)\|_{C^0(J,\XSet)}}{ 1+\|g(\blank,y) \|_{C^0(J,\XSet)}} 
     =
     \frac{G_k(y)}{1+G_0(y)} \leq \gamma^k k! 
     \qquad (y,k) \in \Gamma \times \NSet,
  \]
  hence \cref{hyp:analyticityLRNF} holds with $\gamma = 2(1+T)$. 

  If $P_n$ is the one used in \cref{ex:FEMColl}, then $P_n z \to z $ for all $z \in
  \XSet$, and we use \cref{lem:ProjAnalyticity} to conclude that
  $\cref{hyp:analyticityLRNF_Pn}$ holds. If, on the other hand, $P_n$ is the one used
  in \cref{ex:ChebColl}, we proceed to check H1-H3 on $g$, with updated norms. In the
  previous step we have checked H1 holds with $\mu_{0} = 2(1+T)$. Further, we differentiate
  $g$ and $\partial_{y}^{k} g$ for $r \geq 2$ times with respect to $x$  and obtain
  \[
    \partial_{x}^{r} g(x,t,y) = a^r e^{ty}(y+1) \sin(r \pi/2 + ax), 
    \qquad 
    \partial_{y}^{k}\partial_{x}^{r} g(x,t,y) = a^r e^{ty}t^{k-1} (ty +t +k) \sin(r \pi/2 + ax)
  \]
  hence $v \in C^k_{\sigma}(\Gamma,C^r(D)))$ for any $k$, and H2 holds. We now set
  \[
    \begin{aligned}
      & G_k(y):= \| \partial_{y}^{k}\partial_{x}^{r} g(\blank,y) \|_{C^0(J,\XSet)} 
        = a^r T^{k-1} e^{Ty} |Ty + T + k| \\
      & G_0(y):= \| \partial_{x}^{r} g(\blank,y) \|_{C^0(J,\XSet)} = a^r e^{Ty} |y + 1|
    \end{aligned}
  \]
  and use the bound $a^{-r}e^{-Ty} + |y+1| \geq 1$ to estimate
  \[
    \frac{G_k(y)}{1+G_0(y)} 
      \leq T^k \frac{|y+1|}{a^{-r}e^{-Ty} + |y+1|} + T^{k-1} \frac{k}{a^{-r}e^{-Ty} + |y+1|} 
      \leq (1+T)^k(1+k).
  \]
  Any constant $\mu_r$ satisfying $(1+T)^k(1+k) \leq \mu_r^k k!$ for all $k \in \NSet$ is
  such that
  \[
    \frac{\| \partial_{y}^{k}\partial_{x}^{r} g(\blank,y) \|_{C^{0}(J,\XSet)}}{1+ 
                      \| \partial_{x}^{r} g(\blank,y) \|_{C^0(J,\XSet)}}
    \leq \mu_r^k k!
    \qquad (y,k) \in \Gamma \times \NSet,
  \]
  and hence H3 holds with $\mu_r = 2(1+T)$. Proceeding as in \cref{ex:FEMColl} we
  apply Jackson's Theorem to get
  \[
    \bigl\| P_n \partial_{y}^{k} g(\blank, y) 
          - \partial_{y}^{k} g(\blank, y)\bigr\|_{C^0(J,\XSet)}
    \leq \delta_{r-1} c^r \frac{M_{k,r}(y)}{n^r} 
    \qquad 
    M_{r,k} (y) = \| \partial_{y}^{k}\partial_{x}^{r} g(\blank, y) \|_{C^0(J,\XSet)}.
  \]
  We conclude $P_n \partial_{y}^{k} g(\blank,y) \to \partial_{y}^{k} g(\blank,y)$ in
  $C^0(J,\XSet)$ for all $(y,k) \in \Gamma \times \ZSet_{\geq 0}$ and
  \[
    \| P_n \partial_{y}^{k} g(\blank,y) \|_{\ZSet^0(J,\XSet)} \leq K_g(y) k!\,
    \gamma^k,
  \]
  with 
  \[
    K_g(y) =
    \delta_{r-1} c^r 
    \bigl(1 + \| \partial_{x}^{r} g(\blank,y) \|_{C^0(J,\XSet)} \bigr)
    +
    \bigl(1 + \| g(\blank,y) \|_{C^0(J,\XSet)} \bigr),
    \quad
  \gamma = \max (\mu_r, \mu_0)= 2(1+T),
  \]
  hence \cref{hyp:analyticityLRNF_Pn} is verified.
\end{example}

\subsection{Bounds on the $k$th derivative of the solution to the RLNF
problem}\label{ssec:kThBounds} 
We now make progress towards proving the analyticity of $u_n$, which is presented
below in \cref{ssec:analyticitySection}. A necessary preliminary result is the
following bound on the $k$th derivative of the solution $u_n$ to the projected
problem.

\begin{theorem}[Bounds on $k$th derivative of $u_n$] \label{thm:AnalytLRNFAlt}
  Assume $\LSet = 1$, \crefrange{hyp:domain}{hyp:randomDataLp} (general hypotheses),
  \cref{hyp:finDimNoise} (finite-dimensional noise), and
  \cref{hyp:analyticityLRNF_Pn} (analyticity of random data). The function
  $u_n(x,t,y)$ is infinitely differentiable at $(x,t,y) \in D \times J \times \Gamma$
  with respect to the variable $y_i$, for any $i \in  \NSet_m$. Further, for each $i \in
  \NSet_m$, there exist functions 
  $\delta_i \in C^0(J)$, $D_{i} \in C^0_\sigma(\Gamma,C^0(J))$ 
  such that 
  \begin{align} 
     & \|\partial^k_{y_i}u_n(t,y)\|_{\XSet}
       \leq D_{i}(t,y) k! \, \delta_i(t)^k, 
     && (t,y,k,n) \in J \times \Gamma \times \NSet \times \NSet, \label{eq:unkBoundX} \\
     & \|\partial^k_{y_i}u_n(t,\blank)\|_{C^0_\sigma(\Gamma,\XSet)} 
       \leq \| D_{i}(t,\blank) \|_{C^0_\sigma(\Gamma)} k! \, \delta_i(t)^k,
     && (t,k,n) \in J \times \NSet \times \NSet, \label{eq:unkBoundC0} 
     \\
     & \|\partial^k_{y_i}u_n\|_{C^0_\sigma(\Gamma,C^0(J,\XSet))} 
     \leq \| D_i \|_{C^0_\sigma(\Gamma,C^0(J))} k! \, 
     \| \delta_i \|^k_{C^0(J)},
     && (k,n) \in \NSet \times \NSet. \label{eq:unkBoundC0J}
  \end{align}
  Further, let $\bar K_w$, $K_v$ and $K_g$ be defined as in
  \cref{hyp:analyticityLRNF_Pn}, and let $\beta = 1 + \bar K_w$. It holds:
\begin{remunerate}
  \item If $y_i$ is a component of $y_w$, then 
  \begin{equation*}\label{eq:D_Gamma_g}
     D_i(t,y)= 
     \bigl(K_v(y) + t K_g(y)\bigr)e^{ t \beta },  
     \qquad
     \delta_i(t)=\bigg( \beta\, \frac{e^{t\beta}-1}{\beta}+1\bigg)\gamma_i.
  \end{equation*}
  \item If $y_i$ is a component of $y_v$, then 
  \begin{equation*}\label{eq:D_Gamma_w}
    D_i(t,y)=e^{t \beta} K_v(y), \qquad
    \delta_i(t) \equiv \gamma_i.
  \end{equation*}
  \item
  If $y_i$ is a component of $y_g$, then 
 \[
  D_i(t,y)=\,  \frac{e^{t\beta}-1}{\beta} K_g(y),
  \qquad \delta_i(t) \equiv \gamma_i.
 \]
\end{remunerate}
\end{theorem}
\begin{remark}[Contractivity and time-independent $\delta_i$]\label{rem:contractivity}
  The bound \cref{eq:unkBoundX} shows that, when uncertainty is placed on
  the kernel, $\delta_i$ grows with $t$. This has the important implication (see
  \cref{thm:C-analyticity-C0}) the radius of analyticity of the solution $u_n$
  shrinks as $t$ increases. While this growth is unavoidable in general, the bound on
  $\delta_i$ turns out to be pessimistic in some cases. On linear problems
  where a contraction occurs, such as the one obtained when linearising the
  nonlinear neural field around stable stationary states, the spectrum of the
  compact operator $A(y_w)$ contains eigenvalues with strictly negative real parts,
  and bounded away from the origin when $y_w$ varies in $\Gamma_w$. In
  \cref{sssec:kthBoundsContractive} we make this notion of contractivity precise, and
  show that, for sufficiently large $n$, $\delta_i$ can be made independent of $t$.
\end{remark}
\begin{proof}[Proof of \cref{thm:AnalytLRNFAlt}]
  The proof relies on establishing the existence of partial derivatives, then
  estimate them by differentiating \cref{eq:unEquation} $k$ times with respect to
  a variable $y_i \in y$ and write $\partial^k_{y_i}u_n(t,y)$ as a solution of a
  Cauchy problem. 
  We pick $m,n \in \NSet $, $k \in \ZSet_{\geq 0}$, and $t \in J$, which will remain fixed throughout the
  proof. In preparation for the main argument, we introduce some useful linear
  operators. Firstly, for all $y_w \in \Gamma_w$ the operator
  \begin{equation}\label{eq:AnDef}
  A_n(y_w) = -\id + P_n W(y_w)
  \end{equation}
  is in $BL(\XSet_n)$, and generates the uniformly continuous semigroup $(e^{t
  A_n(y_w)})_{t\geq0}$ on the Banach space $\XSet_n$. Importantly, owing to
  \cref{hyp:analyticityLRNF_Pn}, each realization of the operator is bounded in norm by a
  positive constant, independent on $n$ and $y_w$, 
  \begin{equation}\label{eq:AnBound}
    \| A_n(y_w)\|_{BL(\XSet)} \leq 1+ \bar K_w = \beta \qquad 
    \text{for all $y_w \in \Gamma_w$.}
  \end{equation}

  Secondly, for fixed $i \in \NSet_m$, and $y_w \in \Gamma_w$ we introduce the
  operators $\{B^l_{n,i}(y_w) \colon l \in \NSet_k\}$,
  \[
    B^l_{n,i}(y_w) \colon v \mapsto P_n H( \partial_{y_i}^{l} w(\blank, \blank, y_w) ) = 
   P_n \int_D \partial_{y_i}^{l} w(\blank,x',y_w)  v(x')\, dx',
  \]    
  whose action can be interpreted using \cref{eq:HDef} as done 
  in~\cite[Section 6]{avitabile2024NeuralFields}. By \cref{hyp:analyticityLRNF_Pn},
  $\partial_{y_i}^{l} w(\blank, \blank, y_w) \in \WSet$ and hence
  $H(\partial_{y_i}^{l} w(\blank, \blank, y_w)) \in K(\XSet) \subset BL(\XSet)$ for
  all $y_w \in \Gamma_w$ and $l \in \NSet_k$. Using the boundedness of $P_n$ we
  conclude $\{  B^l_{n,i}(y_w) \}_l \subset BL(\XSet_n)$ for
  all $y_w$ and $i$. Combining
 \cref{hyp:analyticityLRNF_Pn} with~\cite[Lemma 4.2]{avitabile2024NeuralFields}, we can
 bound the operator estimate
 \begin{equation}\label{eq:BlBound}
   \| B^l_{n,i}(y_w)\|_{BL(\XSet)} 
  = \| H( P_n \partial_{y_i}^{l} w(\blank, \blank, y_w) ) \|_{BL(\XSet)} 
 \leq \| P_n \partial_{y_{i}}^{l} w(\blank, \blank, y_w) \|_{\WSet} 
  \leq \bar K_w \gamma_i^l l!.
 \end{equation}

After these preparations, we fix $y=(y_i,y_i^*)$, where $y_i$ is a component of one
of the vectors $y_w$, $y_g$, $y_v$, and turn to the estimation of
$\partial_{y_i}^k u_n(t,y)$, where $u_n$ is the unique classical solution to (see proof of
\cref{lemma:C0SigmaRegularity}) 
\begin{equation}\label{eq:unEquationLocal}
    u'_n(t,y) = A_n(y_w) u_n(t,y) + P_n g(t,y_g), \qquad t \in [0,T], 
    \qquad u_n(0,y) = P_n v(y_v),
\end{equation}
that is,
\[
u_n(t,y) 
    = e^{t A_n(y_w)} P_n v(y_v) 
    + \int_{0}^{t} e^{(t-s)A_n(y_w)} P_n  g(s,y_g)\,d s 
\]
By main hypothesis the mappings $y_i \mapsto A_n(y_i,y_i^*)$, $y_i \mapsto
g(t,y_i,y_i^*)$, and $y_i \mapsto v(y_i,y^*) $ are of class $C^k$ in their respective
function spaces, hence $\partial_{y_i}^k u_n(t,y)$ is well defined. We derive an
evolution equation for $\partial_{y_i}^k u_n$ by differentiating $k$ times
\cref{eq:unEquationLocal} with respect to $y_i$, using the general Leibiniz
rule for differentiation, and noting that $\partial_{y_i}$ commutes with
$\partial_{t}$ and with $P_n$, obtaining the recursion
\begin{equation}\label{eq:ukEvol}
\begin{aligned}
& \partial_{t} \partial_{y_i}^{k} u_n = A_n \partial_{y_i}^{k} u_n
                + P_n\partial_{y_i}^{k} g
                + \sum_{l = 1}^{k} \binom{k}{l} B^l_{n,i} \,
                \partial_{y_i}^{k-l} u_n
                & \textrm{on $[0,T]$,} \\
& \partial_{y_i}^{k} u_n(0) = P_n \partial_{y_i}^{k} v.
                & 
\end{aligned}  
\end{equation}

For fixed $k$ we regard \cref{eq:ukEvol} as an inhomogeneous Cauchy problem on the Banach space
$\XSet_n$, in the variable $\partial^k_{y_i} u$, with fixed parameter $y$, with forcing
dependent on $\partial_{y_i}^{k} g $, $\partial_{y_i}^{k} w$, $u_n,\partial_{y_i}u_n,
\cdots, \partial_{y_i}^{k-l}u_n$, 
By \cite[Theorem 13.24]{vanneervenFunctionalAnalysis2022}, for all $y \in \Gamma$
such that the forcing term is in $L^1(0,T;\XSet_n)$ the Cauchy problem \cref{eq:ukEvol}
admits a unique strong solution given by
\begin{equation}\label{eq:ukBound}
  \begin{aligned}
    \partial_{y_{i}}^{k}u_n(t,y) 
    & = e^{t A_n(y_w)} P_n \partial_{y_i}^{k} v(y_v) 
    + \int_{0}^{t} e^{(t-s)A_n(y_w)} P_n  \partial_{y_i}^kg(s,y_g)\,d s \\
    & + \sum_{l=1}^k \binom{k}{l} \int_{0}^{t} e^{(t-s)A_n(y_w)} 
            B^l_{n,i}(y_w)
                  \partial_{y_i}^{k-l} u_n(s,y) \, ds
  \end{aligned}
\end{equation}
which is therefore in $C(J,\XSet_n)$.

We will now show that for all $y \in \Gamma$ and $k \in \NSet$ the forcing
term is in $L^1(0,T,\XSet_n)$, hence $\partial_{y_i}^k u_n$ exists and is given by
\cref{eq:ukBound}, and find a bound for $\|
\partial_{y_i}^{k}u_n(t,y) \|_\XSet$. An inspection reveals
that only one of the three terms in the right-hand side of \cref{eq:ukBound} is
nonzero, depending on whether $y_i$ is a component of $y_v$, $y_g$, or $y_w$, so we
proceed on a per-case basis.

 \textit{Case 1: bound \cref{eq:unkBoundX} when $y_i$ is a component of $y_v$.} The Cauchy problem in 
 \cref{eq:ukEvol} becomes homogeneous as $\partial_{y_i}^{k} g=0$ and
 $\partial_{y_{i}}^{k} w =0$ so all forcing terms vanish. Hence, it admits the
 solution
 \[
  \partial_{y_{i}}^{k}u_n(t,y) = e^{t A_n(y_w)} P_n \partial_{y_i}^{k} v(y_v).
 \]
Using the boundedness of $A_n$, the bound \cref{eq:AnBound}, 
and \cref{hyp:analyticityLRNF_Pn} we estimate 
 \[
 	  \| \partial_{y_{i}}^{k}u_n(t,y)\|_\XSet 
    \leq e^{t \|A_n(y_w)\|_{BL(\XSet)}} \| P_n \partial_{y_i}^{k} v(y_v)\|_\XSet \\ 
	  \leq e^{t \beta} K_v(y) k!\, \gamma_i^k.
 \]
that gives \cref{eq:unkBoundX} for all $i \in \NSet_m$ such that $y_i \in y_v$, with
\[
 D_i(t,y)=e^{t \beta} K_v(y), \quad \quad \quad
 \delta_i=\gamma_i.
\]

\textit{Case 2: bound \cref{eq:unkBoundX} when $y_i$ is a component of $y_g$.} Here $\partial_{y_{i}}^{k} w =0$,
$\partial_{y_i}^{k} v = 0$ and the forcing term in \cref{eq:ukEvol} is equal to $P_n\partial_{y_i}^k g$. By
 \cref{hyp:analyticityLRNF_Pn} it holds that $P_n\partial_{y_i}^k g(\blank,y_g)$ is in $ C^0(J,\XSet_n) \hookrightarrow
 L^1(0,T;\XSet_n)$ hence, we can write the solution as
 \[
 	 \partial_{y_{i}}^{k}u_n(t,y) = \int_0^t e^{(t-s)A_n(y_w)}P_n\partial_{y_i}^k g(s,y_g)ds.
 \]
 By taking norms and proceeding as in the argument for Case 1 we obtain
   \[
 	  \| \partial_{y_{i}}^{k}u_n(t,y)\|_\XSet 
    \leq \int_0^t e^{(t-s)\|A_n(y_w)\|_{BL(\XSet)}}\|P_n\partial_{y_i}^k g(s,y_g)\|_{\XSet}ds 
	  \leq  K_g(y) k!\, \gamma_i^k \int_0^t e^{(t-s)\beta}ds .
 \]
 Integrating gives
 \[
 	 \| \partial_{y_{i}}^{k}u_n(t,y)\|_\XSet \leq K_g(y) k!\, \gamma_i^k \, \frac{e^{t\beta}-1}{\beta}
 \]
 hence it follows
 \[
  D_i(t,y)=\,  \frac{e^{t\beta}-1}{\beta}\, K_g(y), \qquad \delta_i=\gamma_i.
 \]
 
 \textit{Case 3: bound \cref{eq:unkBoundX} when $y_i$ is a component of $y_w$.} In this last case, the forcing of
 the Cauchy problem in \cref{eq:ukEvol} is $\sum_{l = 1}^{k} \binom{k}{l}
 B^l_{n,i}(y_w) \partial_{y_i}^{k-l} u_n$. 
 This forcing depends on the first $k$ derivatives of $w$ (via the operators
 $B^l_{n,i}$), as well as on the first $k-1$ derivatives of $u_n$. Hence, 
 we need to argue by induction. We prove two claims.

 \textit{Claim 1. For all $r \in \NSet_k$ it holds $\partial_{y_i}^{r} u_n \in
 C^0(J,\XSet_n)$ on $\Gamma_w$}. We use a strong induction argument to prove this
   claim. The base case $r=0$ reduces to $u_n\in C^0(J,\XSet_n)$ on $\Gamma_w$, which
   holds by \cref{lemma:C0SigmaRegularity}. The induction step
   assumes the first $r-1$ derivatives are continuous functions and differentiating
   \cref{eq:unEquation} $r$ times we find that the evolution equation for
   $\partial_{y_i}^{r} u_n$ is a Cauchy problem with forcing 
 \[
   t \mapsto 
   \sum_{l = 1}^{r} \binom{r}{l} B^l_{n,i}(y_w) \, \partial_{y_i}^{r-l} u_n(t,y),
 \]
 which is in $C^0(J,\XSet_n)$ for any $y_w \in \Gamma_w$ because 
 $t \mapsto \partial_{y_i} u_n(t,\blank), \ldots, t \mapsto \partial^{m-1}_{y_i}
 u_n(t,\blank)$ are in $C^0(J,\XSet_n)$ by induction hypothesis and the operators
 $B^l_{n,i}(y_w)$ preserve continuity in $t$, as one can
 see from
\[
  ( B^l_{n,i}(y_w) v) (t) = P_n \int_D \partial_{y_i}^{k} w(\blank,x',y_w)  v(x',t)\, dx'.
\]
Since the forcing for the Cauchy problem for $\partial_{y_i}^{m} u_n$ is in
$C^0(J,\XSet_n)$,
we can apply again \cite[Theorem 13.24]{vanneervenFunctionalAnalysis2022}, to
conclude the existence of a unique strong solution $\partial_{y_i}^{m} u_n\in
C^0(J,\XSet_n)$.
The induction argument is now complete.

\textit{Claim 2. Bound \cref{eq:unkBoundX} holds}. The solution of the Cauchy problem
\cref{eq:ukEvol} when $y_i \in y_w$ is
\[
 \begin{aligned}
    \partial_{y_{i}}^{k}u_n(t,y) 
    = \sum_{l=1}^k \binom{k}{l} \int_{0}^{t} e^{(t-s)A_n(y_w)} 
            B^l_{n,i}(y_w)
            \partial_{y_i}^{k-l} u_n(s,y) \, ds.
  \end{aligned}
\]
Using the bounds \cref{eq:BlBound,eq:AnBound} for $\| B_{n,i}^l(y_w)\|$ we estimate
 \[
 \begin{aligned}
 \bigg\| \frac{\partial_{y_{i}}^{k}u_n(t,y) }{k!} \bigg\|_\XSet & \leq  \sum_{l=1}^k  \int_{0}^{t} e^{(t-s)\beta} 
            \frac{\| B^l_{n,i}(y_w) \|_{BL(\XSet)} }{l!}\,
                 \frac{\| \partial_{y_i}^{k-l} u_n(s,y) \|_\XSet}{(k-l)!}  \, ds, \\
                 &  \leq  \beta  \sum_{l=1}^k \gamma_i^l \,\int_{0}^{t} e^{(t-s)\beta} \,
                 \frac{\| \partial_{y_i}^{k-l} u_n(s,y) \|_\XSet}{(k-l)!}  \, ds, \\            
  \end{aligned}
 \]
 We set
 \[
  R_r(t,y)=\sup_{\tau\in [0,t]}  \bigg\| \frac{\partial_{y_{i}}^{r}u_n(\tau,y) }{r!}
  \bigg\|_\XSet, \qquad  r \in \ZSet_k
 \]
  and derive the following estimate
\begin{equation}\label{eq:RkEstimate}
    \begin{aligned}
      R_k(t,y) &  \leq \beta \sum_{l=1}^k   \gamma_i^l \, \sup_{\tau\in[0,t]}\int_{0}^{\tau} e^{(\tau-s)\beta} \,
      \frac{\| \partial_{y_i}^{k-l} u_n(s,y) \|_\XSet}{(k-l)!}  \, ds \\
               &  \leq \beta \sum_{l=1}^k    \gamma_i^l \, \sup_{\tau\in[0,t]}\int_{0}^{\tau} e^{(\tau-s)\beta} \,
               R_{k-l}(t,y)  \, ds \\
               &  = \beta \sum_{l=1}^k  \gamma_i^l \,  R_{k-l}(t,y)  \sup_{\tau\in[0,t]}\int_{0}^{\tau} e^{(\tau-s)\beta}\, ds\\
               & = \beta \, \frac{e^{t\beta}-1}{\beta} \sum_{l=1}^k  \gamma_i^l \,  R_{k-l}(t,y) 
               =: \nu(t) \sum_{l=1}^k  \gamma_i^l \,  R_{k-l}(t,y).
    \end{aligned}
  \end{equation}
  Using the bound recursively on the elements of the sum gives, omitting dependence on
  $(y,t)$
  \[
    \begin{aligned}
      R_k 
      & \leq  
      \nu \sum_{l=1}^k  \gamma_i^l \,  R_{k-l} 
      = 
      \nu \biggl[ \gamma_i R_{k-1} + \sum_{l=2}^k  \gamma_i^l \,  R_{k-l} \biggr]
      \\
      & \leq 
      \nu \biggl[  \nu \sum_{r=1}^{k-1}  \gamma_i^{r+1} \,  R_{k-1-r} 
      + \sum_{l=2}^k  \gamma_i^l \,  R_{k-l} \biggr]
      = 
      \nu (\nu+1)\sum_{l=2}^k  \gamma_i^l \,  R_{k-l}
      \\
      & 
      \leq 
      \nu (\nu+1)^2 \sum_{l=3}^k  \gamma_i^l \,  R_{k-l}
      \leq \ldots \leq \nu (\nu+1)^{k-1} \gamma_i^k R_0
    \end{aligned}
  \]
  and therefore 
  \[
    R_k(t,y)\leq  \bigg[ \bigg( \beta \, \frac{e^{t\beta}-1}{\beta}+1\bigg) \gamma_i \bigg]^{k}\,R_0(t,y)
  \]
  Applying~\cite[Theorem 6.3]{avitabile2024NeuralFields} under finite-dimensional noise
  (\cref{hyp:finDimNoise}) we have 
  \[
    R_0(t,y)=\|u_n(\blank,y)\|_{C([0,t],\XSet)} 
    \leq \big(
      \|  P_n v(y) \|_{\XSet} + \| P_n g(\blank ,y) \|_{C^0(J,\XSet)} t 
    \bigr) e^{\| P_nW(y) \| t}
  \]
  and by \cref{hyp:analyticityLRNF_Pn} we have
  \[
    R_0(t,y) \leq \bigl( K_v(y) + K_g(y) t\bigr) e^{\beta t}
  \]
  We conclude that for any $(t, y )\in J \times \Gamma$
  \[
    \frac{ \| \partial_{y_{i}}^{k}u_n(t,y)\|_\XSet }{k!} \leq R_k(t,y)\leq D_i(t,y)
    \delta_i(t)^{k}
  \]
  where the coefficients read
  \[
    D_i(t,y)=\bigl( K_v(y) + K_g(y) t\bigr) e^{\beta t} ,  \quad
    \quad \quad \delta_i(t)=\bigg( \beta\, \frac{e^{t\beta}-1}{\beta}+1\bigg)\gamma_i.
  \]

\textit{Derivation of bounds \cref{eq:unkBoundC0,eq:unkBoundC0J}.} We have so far
established \cref{eq:unkBoundX}. By \cref{hyp:analyticityLRNF_Pn} the functions $K_g$
and $K_v$ (and hence $y \mapsto D_i(t,y)$) have a finite $C^0_\sigma(\Gamma)$ norm.
Taking the norm in $C^0(\Gamma,\XSet)$ on both sides of \cref{eq:unkBoundX} gives
\cref{eq:unkBoundC0}. To derive 
\cref{eq:unkBoundC0J} we start from \cref{eq:unkBoundX} and estimate
\[
  \begin{aligned}
    &
    \| \partial_{y_i}^k u_n (\blank, y) \|_{C^0(J,\XSet)} 
    \leq 
    k! \sup_{t \in J} D_{i}(t,y) \delta_i(t)^k 
    \leq 
    k!
    \bigg( \sup_{t \in J} D_{i}(t,y) \bigg)
    \bigg( \sup_{t \in J} \delta_i(t) \bigg)^k, 
  \end{aligned}
\]
where we have used the fact that either $\delta_i(t) \equiv \gamma_i$, or
$\delta_i(t) = \nu(t) \gamma_i$, with $\nu(t) \geq 1$ for all $t$. Taking the
norm in $C^0_\sigma(\Gamma~C^0(J,\XSet))$ gives \cref{eq:unkBoundC0J}.
\end{proof}

\subsubsection{Bounds on the $k$th derivative of $u_n$ for contractive
RLNFs}\label{sssec:kthBoundsContractive} 
Following from \cref{rem:contractivity}, we now introduce an abstract condition
(\cref{hyp:contractivity}) which describes contractivity of the parameter-dependent
semigroup associated with the linear operators $A_n(y)$ appearing in the proof or
\cref{thm:AnalytLRNFAlt}. The upcoming \cref{cor:contractivity} shows that this condition
is sufficient to make $\delta_i(t)$ in \cref{thm:AnalytLRNFAlt} independent of $t$,
for sufficiently large $n$.

\begin{hypothesis}[Contractivity of discretised linear operators]\label{hyp:contractivity}
  There exist $M \in
  \RSet_{\geq 1}$, $\epsi \in \RSet_{>0}$, and $n_0 \in \NSet$ such that for the
  family of operators $\{ A_n \}_n$ defined in \cref{eq:AnDef} it holds
  \[
    \| e^{t A_n(y)} \| \leq M e^{-\epsi t}
    \qquad 
    (y,t,n) \in \Gamma_w \times \RSet_{\geq 0} \times \NSet_{\geq n_0}.
  \]
\end{hypothesis}
\begin{corollaryE}[Contractive case of
  \cref{thm:AnalytLRNFAlt}]\label{cor:contractivity}
  Assume the hypotheses of \cref{thm:AnalytLRNFAlt} and
  \cref{hyp:contractivity} (contractvity). If $y_i$ is a component of $y_w$, then
  estimates \crefrange{eq:unkBoundX}{eq:unkBoundC0J} hold for $n \geq n_0$ with the
  constant function $\delta_i(t) \equiv (1+M\beta/\epsi)$.
\end{corollaryE}
\begin{proofE}
  The proof is a slight amendment of Case 3 in the proof of \cref{thm:AnalytLRNFAlt},
  which runs identical to this one up to the following bound in Claim 2:
 \[
 \bigg\| \frac{\partial_{y_{i}}^{k}u_n(t,y) }{k!} \bigg\|_\XSet 
     \leq  \sum_{l=1}^k  \int_{0}^{t} 
       \big\| e^{(t-s)A_{n}(y)} \big\|
       \frac{\| B^l_{n,i}(y_w) \|_{BL(\XSet)} }{l!}\,
       \frac{\| \partial_{y_i}^{k-l} u_n(s,y) \|_\XSet}{(k-l)!}  \, ds.
 \]
 As in the proof of \cref{thm:AnalytLRNFAlt} we use the bounds
 \cref{eq:BlBound,eq:AnBound}, but the contractivity hypothesis on the semigroups
 allow us to introduce the negative exponent $-\epsi$ in the bound,
 \[
   \bigg\| \frac{\partial_{y_{i}}^{k}u_n(t,y) }{k!} \bigg\|_\XSet 
     \leq  M \beta \sum_{l=1}^k \gamma_i^l 
     \int_{0}^{t} 
       e^{- \epsi (t-s)} 
     \bigg\| \frac{\partial_{y_i}^{k-l} u_n(s,y) }{(k-l)!} \bigg\|_\XSet  \, ds, 
 \]
 and defining $R_k$ as in the proof of \cref{thm:AnalytLRNFAlt} we estimate
 \[
   \begin{aligned}
     R_k(t,y) 
     & \leq 
       M \beta \biggl( \sup_{\tau \in [0,t]} \int_{0}^{\tau}e^{-\epsi(t-s)} \,ds\biggr) 
       \sum_{l=1}^k \gamma_i^l R_{k-l}(t,y) 
      \leq 
       M \beta \frac{1-e^{-\epsi t}}{\epsi}
       \sum_{l=1}^k \gamma_i^l R_{k-l}(t,y)
       \\
     &  \leq 
       \frac{M \beta}{\epsi} \sum_{l=1}^k \gamma_i^l R_{k-l}(t,y)
       =:
       \nu \sum_{l=1}^k \gamma_i^l R_{k-l}(t,y),
   \end{aligned}
 \]
 which is analogous to \cref{eq:RkEstimate} upon replacing $\nu(t)$ by the
 time-independent $\nu = M\beta/\epsi$. Following \cref{thm:AnalytLRNFAlt} with
 this new $\nu$ gives the assertion.
\end{proofE}

A natural question arises as to how one can check \cref{hyp:contractivity} starting
from conditions on the spectrum of the compact linear operators $W(y)$ when $y$
varies in $\Gamma_w$. Intuitively we expect that \cref{hyp:contractivity} may hold
if the spectrum of $W(y)$ is contained in $\{ z \in \CSet
\colon \real z < 1 \}$, independently of $y$. Hypothesis 2 in
\cref{lem:contractivityCondition} makes this
spectral condition precise. \Cref{lem:contractivityCondition} shows that there are
further conditions required to check \cref{hyp:contractivity} concerning the
compactness of $\Gamma_w$ (Hypothesis 1 in the lemma), and the projector $P_n$ (Hypothesis 2). It
is possible to relax $\Gamma_w$ and obtain contractivity in the absence of
projectors, provided one amends Hypothesis 2 so as to require uniform sectoriality of
the operators. We do not give a theoretical result in which the
hypothesis on $P_n$ is relaxed, albeit we show numerical evidence of convergence in
\cref{sec:numerics}.

\begin{lemmaE}[Checking the contractivity hypothesis \labelcref{hyp:contractivity}]
  \label{lem:contractivityCondition}
  Let $\{ A_n \}_n \subset C^0(\Gamma_w, BL(\XSet))$ be the sequence of functions
  $A_n \colon y \mapsto -\id + P_n W(y)$, where $\{ P_n \}_n$ is a family of
  projectors on $\XSet$ to $\XSet_n$, and $W \in C^0(\Gamma_w, K(\XSet))$
  with $\Gamma_w \subset \RSet^{m_w}$, and $K(\XSet)$ the set of compact operators in
  $BL(\XSet)$. If:
  \begin{remunerate}
  \item the domain $\Gamma_w$ is compact;
  \item there exists $ \mu > 0$ such that
    \begin{equation}\label{eq:spectralGapAy}
      \sup_{y \in \Gamma_w} \sup_{\lambda \in \sigma(W(y))} \real \lambda
      < 1 - \mu;
    \end{equation}
  \item the projectors satisfy $P_n v \to v$ as $n \to \infty$ for all $v \in \XSet$;
  \end{remunerate}
  then \cref{hyp:contractivity} holds.
\end{lemmaE}
\begin{proofE}
  The proof is based on bounds for the semigroup associated with
  a perturbation $A + (A_n - A)$ to the linear operator $A_n$, for which $A$ will be
  defined in the steps below. 

  \textit{Step 1: homogeneous bound on $e^{tA(y)}$.} The mapping $A \colon \Gamma_w \to
  BL(\XSet)$, with $A(y) = -\id + W(y)$, is continuous on the compact $\Gamma_w$, and
  hence the set $\{ A(y) \in
  BL(\XSet) \colon y \in \Gamma_w \}$ is a compact subset of operators in
  $BL(\XSet)$, whose spectra lie in the halfplane $\{ z \in \CSet \colon \real z <
  -\mu\}$ by \cref{eq:spectralGapAy}. By \cite[Chapter III. Lemma
  6.2]{daleckiiStabilitySolutionsDifferential2002}, there exists $M \geq 1$ such that
  \[
     \|  e^{t A(y)} \| \leq M e^{-\mu t} \qquad (y,t) \in \Gamma_w \times \RSet_{\geq 0}.
  \]

  \textit{Step 2: uniform convergence of $P_nW$ to $W$ on $\Gamma_w$.} Consider the
  sequence $\{ W_n \}_{n \in \NSet} \subset C^0(\Gamma_w, BL(\XSet))$
  with elements defined by the functions $W_n(y) := P_n W(y)$. For any  $y \in
  \Gamma_w$, the operator $W(y)$ is compact, hence $W_n(y) \to W(y)$ by \cite[Lemma
  12.1.4]{atkinson2005theoretical}, that is, $W_n$ converges to $W$ pointwise.
  We now show that the convergence is uniform on $\Gamma_w$.

  The hypothesis on $P_n$ guarantees that the functions $W_n$ are uniformly bounded
  by the Principle of Uniform Boundedness \cite[Theorem
  2.4.4]{atkinson2005theoretical}, hence $\kappa := \sup_{n \in \NSet} \| P_n \| <
  \infty$ and
  \[
    \sup_{n \in \NSet} \| W_n \|_{C(\Gamma, BL(\XSet_n))} 
      \leq \bigg( \sup_{n \in \NSet} \| P_n \| \bigg) \| W \|_{C(\Gamma_w, BL(\XSet))}
      \leq \kappa \| W \|_{C(\Gamma_w, BL(\XSet))}
      < \infty. 
  \]
  Further, the functions $W_n$ are equicontinous: fix $\rho >0$; since $W \in C^0(\Gamma_w,
  BL(\XSet))$ there exists $\delta > 0$ such that $\| W(y) - W(y')\| < \rho/\kappa$
  provided $ \| y - y' \| < \delta$; hence $\| y - y' \| < \delta$ implies
  \[
      \| W_n (y) - W_n (y') \| 
      \leq \| P_n \| \| W(y) - W(y')\|
     \leq \rho.
     \qquad n \in \NSet.
  \] 
  By the Arzelà--Ascoli Theorem \cite[Theorem 1.6.3]{atkinson2005theoretical},
  generalised to the case of functions with values in Banach spaces,
 the convergence of $W_n$ to $W$ is uniform on
  the compact set $\Gamma_w$.

  \textit{Step 3: homoegeneous bound on $e^{t A_n(y)}$.} We write $A_n(y) = A(y) + (A_n(y)
  -A(y))$. Owing to the bound for $e^{t A(y)}$ found in Step 1, we can apply
  \cite[Theorem 1.3, page 158]{engelOneparameterSemigroupsLinear2000} and deduce
  that, for any $(y,t,n) \in \Gamma_w \times \RSet_{\geq 0}\times \NSet$, it holds
  \[
    \| e^{t A_n(y)} \| \leq M \exp \big[ (-\mu + M \| A_n(y) - A(y) \|) t \big] 
    = M \exp \big[ (-\mu + M \| P_n W(y) - W(y) \|) t \big]
  \]
  Since $P_n W$ converges uniformly to $W$ on $\Gamma_w$, there exists an $n_0$,
  independent of $y$, such that $\| P_n W(y) -W(y) \| \leq \mu/M$ for all $n \geq  n_0$
  and $y \in \Gamma_w$. We set
  \[
    \epsi := \mu - M \sup_{n \geq  n_0} \sup_{y \in \Gamma_w} \| P_n W(y) - W(y) \| > 0,
  \]
  and derive, for any $(y, t, n) \in \Gamma_w \times \RSet_{\geq 0}\times \NSet_{\geq 
  n_0}$, the estimate
  \[
    \| e^{t A_n(y)} \| \leq M \exp 
    \Bigl[ \Bigl(-\mu + M \sup_{n \geq  n_0} \sup_{y \in \Gamma_w} \| P_n W(y) - W(y) \|\Bigr) t \Bigr] 
    = M e^{-\epsi t}.
  \]
\end{proofE}

\subsection{Analyticity of $u_n$}\label{ssec:analyticitySection} 
In preparation for the analyticity proof, we present a technical statement that
follows from \cref{thm:AnalytLRNFAlt}, and concerns analyticity and Fréchet
differentiability of $u_n$, as a function on $\Gamma_i$. Here and henceforth, we
study analyticity of mappings on Banach spaces, adopting the definition and
theory in \cite[Definition 4.3.1, and Section 4.3]{buffoni_analytic_2003}, according
to which a function $\psi \colon \BSet_1 \to \BSet_2$ a
function is said to be $\RSet$- or $\CSet$-analytic depending on whether $\BSet_1$ and
$\BSet_2$ are Banach spaces over $\RSet$ or $\CSet$. The sketch in
\cref{fig:analyticitySketch} is helpful in navigating the material in this section.
\begin{theorem}[$\RSet$-analyticity of the semidiscrete
  solution with values in $\XSet$]\label{thm:r-analyticity}
  Assume the hypotheses of \cref{thm:AnalytLRNFAlt}, fix $n \in \NSet$, $t \in J$,
  $i \in \NSet_m$, and $y_0\in \Gamma_i$. The mapping
  \[
    \psi_t \colon \Gamma_i \subset \RSet \to C^0_{\sigma_i^*}(\Gamma_i^*,\XSet), \qquad y_i \mapsto
    u_n(t,y_i,\blank).
  \]
  is $\RSet$-analytic at $U_{\tau_i} = \{ y \in \Gamma_i \colon |y-y_0|  \leq \tau_i
  \}$, for any $0 < \tau_i < 1/\delta_i(t)$,
  and its $k$th Fréchet derivative at $y_0 \in U_{\tau_i}$ is given by 
  \begin{equation}\label{eq:ukFrechet}
    d^k \psi_t[y_0] \colon \Gamma_i^k \subset \RSet^k \to C^0_{\sigma_i^*}(\Gamma_i^*,\XSet)
    \qquad (\gamma_1,\ldots,\gamma_k) \mapsto 
    \partial_{y_i}^{k} u_n(t,y_0,\blank) \gamma_1 \cdots \gamma_k.
  \end{equation}
\end{theorem}
\begin{proof} 
Fix an arbitrary $y_0 \in \Gamma_i$, and denote by $\BSet$ the space
$C^0_{\sigma^*}(\Gamma^*_i,\XSet)$ over $\RSet$. In this proof we will adopt the
representations $(y,y_i^*), (y_0,y_i^*)$ for variables in $\Gamma_i \times
\Gamma_i^*$.
By \cref{thm:AnalytLRNFAlt}, the real-valued function $y \mapsto
u_n(t,y,y_i^*)(x)$ is in $C^\infty(\Gamma_i,\RSet)$ for any $x$, $t$, and $y_i^*$. By Taylor's theorem
for real-valued functions of a single variable, there exists a function $R_k$ such that
\[
  u_n(t,y,y_i^*)(x)-
  \sum_{j=0}^{k} \frac{1}{j!}
  \partial^j_{y_i} u_n(t,y_0,y_i^*)(x) (y-y_0)^j =
  R_{k}(y_0,y,t,y^*_i)(x),
\]
We will examine in a moment the asymptotic properties of $R_k$, which we put on a
side for now, and we rewrite the previous identity as
\begin{equation}\label{eq:psiExpansion}
  \psi_t(y)(y_i^*)
  -\sum_{j=0}^{k} \frac{1}{j!}
  \partial^j_{y_i} u_n(t,y_0,y_i^*)(y-y_0)^j =
  R_k(y_0,y,t,y^*_i).
\end{equation}
Crucially, the last identity does not make any statement about the
differentiability of $\psi_t$, nor does it express $d^k \psi_t$, the Fréchet
derivative, in terms of $\partial_{y_i}^k u_n$, the standard partial derivatives, albeit
the notation is suggestive of that. We now elucidate this connection, and prove the
analyticity of $\psi_t$. 

We introduce the symbol $m_0 (y-y_0)^0 = \psi_t(y_0)$, and the family of $k$-linear operators $\{ m_k \}$
defined by
\[
  m_k \colon 
  \Gamma_i^k \subset \RSet^k \to \BSet
  \qquad 
    (\gamma_1,\ldots,\gamma_k) \mapsto 
    \frac{1}{k!}\partial_{y_i}^{k} u_n(t,y_0,\blank) \gamma_1 \cdots \gamma_k,
    \qquad k  \geq 1.
\]
By \cref{thm:AnalytLRNFAlt} $\partial^k_{y_i} u_n(t,\blank,\blank)$ is in
$C^0_\sigma(\Gamma,\XSet)$, hence $\partial^k_{y_i} u_n(t,y_0,\blank)$ is in $\BSet$.
This, together with the fact that the product $\gamma_1 \cdots \gamma_k$
is in the reals, shows that effectively $m_k$ is on $\Gamma^k_i$ to $\BSet$.
 
The operator norm of $m_k$ is estimated as follows:
\begin{equation}\label{eq:FrechDerEstimate}
  \begin{aligned}
  \| m_k \| 
  & = \sup_{|\gamma_1|, \ldots, |\gamma_k|  \leq 1} 
    \frac{1}{k!}\| \partial_{y_i}^{k} u_n(t,y_0,\blank)  \gamma_1 \cdots \gamma_k \|_{\BSet}
  \leq 
    \frac{1}{k!}\| \partial_{y_i}^{k} u_n(t,y_0,\blank) \|_{\BSet} \\
  & \leq 
    \frac{1}{k! \sigma_i(y_0)}\| \partial_{y_i}^{k} u_n(t,\blank,\blank)
    \|_{C^0_\sigma(\Gamma,\XSet)}
    \leq 
    \frac{\delta_i(t)^k}{\sigma_i(y_0)}
    \| D_{i}(t,\blank) \|_{C^0_\sigma(\Gamma)} < \infty,
  \end{aligned}
\end{equation}
where we used the following inequality, valid for any 
$f \in C^0_{\sigma}(\Gamma,\XSet)$,
\[
  \| f \|_{C^0_{\sigma}(\Gamma,\XSet)}  \geq \sigma_i(s)  
  \| f(s,\blank) \|_{C^0_{\sigma^*_i}(\Gamma_i^*,\XSet)}
  =  \sigma_i(s)  \| f(s,\blank) \|_{\BSet},
  \qquad
  \quad s \in \Gamma_i,
\]
\begin{figure}
  \centering
  \includegraphics{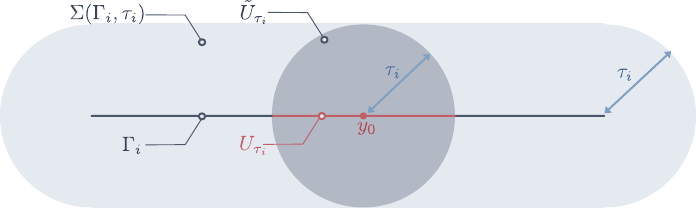}
  \caption{Sketch of the domains of analyticity of various functions used in the
    analycitity proof. The mapping
  $\psi_t$ of \cref{thm:r-analyticity} is defined on $\Gamma_i$ (grey horizontal line) and is $\RSet$-analytic on
  the interval $U_{\tau_i} \subset \Gamma_i$ of the real line (red). We use partial derivatives
$\partial^k_y u$ defined on $\Gamma_i \subset \RSet$, and the bounds of
\cref{thm:AnalytLRNFAlt} to define Fréchet derivatives in an
appropriate Banach space and obtain analyticity of $\psi_t$ on $U_{\tau_i}$. The function $\tilde
\psi_t$ in the proof of \cref{cor:Sigma} is defined on $\CSet$ and is
$\CSet$-analytic in $\tilde U_{\tau_i}$ (dark grey).
Without upgrading our definition of derivatives we prove that $\tilde \psi_t$ is the
analytic extension of $\psi_t$ on $\tilde U_{\tau_i}$. Finally, we prove the
existence of a function $\tilde u_n$ that extends $\psi_t$ analytically in the whole
$\Sigma(\Gamma_i,\tau_i)$ (light grey).}
  \label{fig:analyticitySketch}
\end{figure}
in conjunction with bound \cref{eq:unkBoundC0} of \cref{thm:AnalytLRNFAlt}. We
therefore know that $m_k$ is a symmetric operator in $\mathcal{M}^k(\Gamma_i,\BSet)$,
the set of $k$-linear operators on $\Gamma_i^k$ to $\BSet$ \cite[Proposition
4.1.2]{buffoni_analytic_2003} and by \cref{eq:psiExpansion} we have, for almost every $y\in\Gamma_i$,
\[
  \psi_t(y) -\sum_{j=0}^{k} m_j(y-y_0)^j = R_k(y_0,y,t,\blank)\in\BSet.
\]
We will now show that the sequence $\{ R_k \}$ converges in $\BSet$ as $k \to
\infty$. As an intermediate step, we bound $R_k$ in the $\XSet$-norm using the
Taylor's theorem bound for real-valued functions of a single variable, obtaining
\[
  \big\| R_k(y_0,y,t,y_i^*) \big\|_{\XSet} \leq 
  \frac{|y-y_0|^{k+1}}{(k+1)!} \sup_{s \in [y_0,y]} \| \partial_{y_i}^{k+1} u(t,s,y_i^*) \|_{\XSet}
\]
With the view of bounding $R_k$ in the $\BSet$-norm, we multiply each side of the
previous inequality by the strictly positive number $\tilde \sigma_i(y_0,y) = \{
  \min\sigma_i(s) \colon s \in [y_0,y] \}$, and obtain
\[
  \begin{aligned}
  \tilde \sigma_i(y_0,y) \big\| R_k(y_0,y,t,y_i^*) \big\|_{\XSet} 
  & \leq \frac{|y-y_0|^{k+1}}{(k+1)!} \tilde \sigma_i(y_0,y) 
  \sup_{s \in [y_0,y]} \| \partial_{y_i}^{k+1} u(t,s,y_i^*) \|_{\XSet} \\
  & \leq \frac{|y-y_0|^{k+1}}{(k+1)!} 
  \sup_{s \in [y_0,y]} \sigma_i(s)\| \partial_{y_i}^{k+1} u(t,s,y_i^*) \|_{\XSet},
  \end{aligned}
\]
hence, by \cref{thm:AnalytLRNFAlt},
\[
  \begin{aligned}
  \big\| R_k(y_0,y,t,\blank) \big\|_{\BSet} 
  & \leq 
    \frac{1}{\tilde \sigma_i(y_0,y)}\frac{|y-y_0|^{k+1}}{(k+1)!} \sup_{s' \in \Gamma_i^*}  \sigma_i^*(s') \sup_{s \in [y_0,y]} \sigma_i(s) 
    \| \partial_{y_i}^{k+1} u(t,s,s') \|_{\XSet} \\
  & \leq 
    \frac{1}{\tilde \sigma_i(y_0,y)}\frac{|y-y_0|^{k+1}}{(k+1)!}\sup_{s' \in \Gamma_i^*}  \sup_{s \in \Gamma_i}  \sigma_i^*(s')\sigma_i(s)     \| \partial_{y_i}^{k+1} u(t,s,s') \|_{\XSet} \\
  & \leq 
  \frac{\bigl( |y-y_0| \delta_i(t) \bigr)^{k+1}}{\tilde \sigma(y_0,y)} \| D_i(t,\blank)
  \|_{C^0_\sigma(\Gamma)}.
  \end{aligned}
\]
We can now draw a few conclusions: firstly, owing to the previous inequality, the
power series
\[
  \psi_t(y) = \sum_{j=0}^{\infty} m_j(y-y_0)^j
\]
converges for all $y$ in the interval $U_{\tau_i} = \{ |y-y_0|  \leq \tau_i \}
\subset \Gamma_i$, with $0 < \tau_i < 1/\delta_i(t)$. Secondly, we return to the
estimate \cref{eq:FrechDerEstimate}, which gives
 \[
   \sup_{k \geq 0 } \tau_i^k \| m_k \| < \infty, \qquad 0 < \tau_i <
   \delta_i(t).
 \]
 The hypotheses of \cite[Proposition 4.3.4]{buffoni_analytic_2003} are verified and
 we conclude that for any $0 < \tau < \delta_i(t)$: (i) the mapping $\psi_t$ is
 $\RSet$-analytic on $U_{\tau_i}$ \cite[Definition
 4.3.1]{buffoni_analytic_2003}; (ii) it holds $\psi_t \in C^\infty
 (U_{\tau_i},\BSet)$; and (iii) $m_k$ is the $k$th Fréchet derivative $d^k
 \psi_t[y_0]$ of $\psi_t$ at $y_0 \in U_{\tau_i}$. The proof is now complete.
\end{proof}

We have proved $\RSet$-analyticity of the $C^0_{\sigma_i}(\Gamma_i,\XSet)$-valued
solution $y_i \mapsto u_n(t,y_i,\blank)$ on an interval in $\Gamma_i \subset \RSet$.
To that effect, we transplanted bounds on partial derivatives of the real-valued
solution $\partial_{y_i}^k u_n$, to bounds on the Fréchet derivative of the mapping
above. We now extend this $\RSet$-analytic mapping to a $\CSet$-analytic mapping,
defined on an open set of the complex plane. Importantly, this extension can be
carried out without investigating Fréchet differentiability of Banach-space valued
functions of complex variables, on the field $\CSet$, for which we have thus far
proved no result.

\begin{corollary}[to \cref{thm:r-analyticity}; $\CSet$-analytic extension of the
  projmected solution with values in $\XSet$]\label{cor:Sigma}
Assume the hypothesis of \cref{thm:AnalytLRNFAlt}, fix $n,m \in \NSet$, $i\in
\mathbb{N}_m$ and $t \in J$. The solution of the projected
problem $u_n(t,y_i,y_i^*)$ as an $\RSet$-analytic function of $y_i$, 
$u_n \colon \Gamma_i \subseteq \RSet \to C^0_{\sigma^*_i}(\Gamma_i^*,\XSet)$
admits a $\CSet$-analytic extension $\tilde u_n \colon \CSet \to
C^0_{\sigma^*_i}(\Gamma_i^*,\XSet)$ in the region of the complex plane 
\[
\Sigma(\Gamma_i,\tau_i)=\{ z\in\mathbb{C} : \dist(\Gamma_i,z)\leq \tau_i\},
\]
for any $0<\tau_i<1/\delta_i(t)$.
\end{corollary}
\begin{proof} Pick $y_0 \in \Gamma_i$, and let $\BSet_{\CSet}$ be the Banach space
  $C^0_{\sigma^*_i}(\Gamma_i^*,\XSet)$ on the field $\CSet$. Introduce the mappings
\begin{equation}\label{eq:pseriesNew}
  \tilde \psi_k \colon \CSet \to \BSet_{\CSet}, \qquad 
  z \mapsto \sum_{j=0}^k \frac{1}{j!} 
  \partial_{y_i}^j u_n(t,y_0,\blank) (z-y_0)^j 
  \qquad k  \geq 1.
\end{equation}
In passing we note that $\tilde \psi_k$ defined above and $\psi_t$ in
\cref{thm:r-analyticity} have different domains (a subset of $\CSet$ and $\RSet$,
respectively) but also different codomains: the former has values in $\BSet_{\CSet}$,
and the latter in $\BSet$, that is $C^0_{\sigma^*_i}(\Gamma_i^*,\XSet)$ on the field
$\RSet$. 

Proceeding as in the proof of \cref{thm:r-analyticity}, and replacing the absolute value
$|\blank|_{\RSet}$ appearing in the $\XSet$ norms by the absolute value
$|\blank|_{\CSet}$, we bound $\| \partial_{y_i}^j u_n(t,y_0,\blank)
(z-y_0)^j/j! \|_{\BSet_{\CSet}}$, and find
\[
  \lim_{k \to \infty} \tilde \psi_k(z) \to \tilde \psi_t(z)
  \qquad \text{as $k \to \infty$}
  \qquad \text{for all $z \in \tilde U_{\tau_i} $}
\]
for some $\BSet_{\CSet}$-valued function $\tilde \psi_t$, where $\tilde U_{\tau_i} = \{ z \in  \CSet \colon | z-y_0| \leq \tau_i\}$, with $0 < \tau_i
< 1/\delta_i(t)$. 

For any $0 < \tau_i < \delta_i(t)$, \Cref{thm:r-analyticity} defines $\psi_t$, a
$\BSet$-valued function that is $\RSet$-analytic in $U_{\tau_i} \subset
\tilde U_{\tau_i}$. Above we proved the existence of $\tilde \psi_t$,
a $\BSet_{\CSet}$-valued function that is $\CSet$-analytic in $\tilde U_{\tau_i}$ and
such that
\[
  \tilde \psi_t(y) = \sum_{k =0}^\infty \frac{1}{k!} d^k \psi[y_0](y-y_0)^k=\psi_t(y)
  = u_n(t,y,\blank) \qquad \text{for all $y \in U_{\tau_i}$}.
\]
We therefore found an analytic extension to the solution $u_n$ in $\tilde
U_{\tau_i}$. The steps above can be repeated for any $y_0 \in \Gamma_i$, providing an
extension $\tilde u_n$ to $u_n$ in $\Sigma(\Gamma,\tau_i)$ for all $0 < \tau_i <
1/\delta_i(t)$.
\end{proof}

The results of \cref{thm:r-analyticity,cor:Sigma} concern analyticity of the
time-dependent function $u_n(t,\blank)$. \Cref{thm:AnalytLRNFAlt} also provides
\cref{eq:unkBoundC0J}, a bound valid for $t \in J= [0,T]$, and it is possible to
obtain analyticity of $u_n$ as a $C^0(J,\XSet)$-valued solution, as the following
result shows.
\begin{theoremE}[$\CSet$-analyticity of the semidiscrete solution with values in $C^0(J,\XSet)$]
  \label{thm:C-analyticity-C0}
  Assume the hypotheses of \cref{thm:AnalytLRNFAlt}, and fix $n \in \NSet$ and $i \in
  \NSet_m$. The solution of the projected problem $u_n(\blank,y_i,y_i^*)$ as an
  $\RSet$-analytic function of $y_i$, $u_n \colon \Gamma_i \subseteq \RSet \to
  C^0_{\sigma^*_i}(\Gamma_i^*,C^0(J,\XSet))$ admits a $\CSet$-analytic extension 
  $\tilde u_n \colon \CSet \to C^0_{\sigma^*_i}(\Gamma_i^*,C^0(J,\XSet))$ in the
  region of the complex plane
  \[
    \Sigma(\Gamma_i,\tau_i)=\{ z\in\mathbb{C} : \dist(\Gamma_i,z)\leq \tau_i\},
  \]
  for any $0<\tau_i<1/\| \delta_i \|_{C^0(J)}$.
\end{theoremE}
\begin{proofE}The proof is a straightforward amendment of the ones in
  \cref{thm:r-analyticity,cor:Sigma}. We amend the definition of $\BSet$ as $\BSet =
  C^0_{\sigma_i^*}(\Gamma_i^*,C^0(J,\XSet))$ over $\RSet$, and let
  \[
    \psi \colon \Gamma_i \subset \RSet \to \BSet, \qquad y_i \mapsto
    u_n(\blank,y_i,\blank).
  \]
  From \cref{eq:psiExpansion} we derive
\[
  \psi(y) -\sum_{j=0}^{k} m_j(y-y_0)^j = R_k(y_0,y,\blank,\blank)\in\BSet,
\]
with the updated definitions $m_0 (y-y_0)^0 = \psi(y_0)$, and
\[
  m_k \colon 
    (\gamma_1,\ldots,\gamma_k) \mapsto 
    \frac{1}{k!}\partial_{y_i}^{k} u_n(\blank,y_0,\blank) \gamma_1 \cdots \gamma_k,
    \qquad k  \geq 1,
\]
whose norm is bounded using \cref{eq:unkBoundC0J} and proceeding as in
\eqref{eq:FrechDerEstimate}
\[
 \| m_k \| 
 \leq 
  \frac{1}{k! \sigma_i(y_0)}\| \partial_{y_i}^{k} u_n \|_{C^0_\sigma(\Gamma,C^0(J,\XSet))}
 \leq \frac{1}{\sigma_i(y_0)} 
 \| \delta_i \|^k_{C^0(J)}
 \| D_{i} \|_{C^0_\sigma(\Gamma,C^0(J))} 
 < \infty.
\]
The operator $m_k$ is symmetric and in $\mathcal{M}^k(\Gamma_i,\BSet)$,
\cite[Proposition 4.1.2]{buffoni_analytic_2003}. Proceeding as in the proof of
\cref{thm:r-analyticity}, we introduce $\tilde \sigma_i(y_0,y) = \{  \min\sigma_i(s)
\colon s \in [y_0,y] \}$ and estimate
  \[
    \begin{aligned}
      \big\| R_k(y_0,y,\blank,\blank) \big\|_{\BSet} 
        & \leq 
        \frac{1}{\tilde \sigma_i(y_0,y)}\frac{|y-y_0|^{k+1}}{(k+1)!}
        \sup_{s' \in \Gamma_i^*}  \sup_{s \in \Gamma_i}  \sigma_i^*(s')\sigma_i(s) \|
        \partial_{y_i}^{k+1} u(\blank,s,s') \|_{C^0(J,\XSet)} \\
        & \leq 
        \frac{\bigl( |y-y_0| \, \| \delta_i \|_{C^0(J)} \bigr)^{k+1}}{\tilde \sigma(y_0,y)} 
        \| D_i \|_{C^0_\sigma(\Gamma,C^0(J))}
    \end{aligned}
  \]
  hence the power series $\psi(y) = \sum_{j=0}^{\infty} m_j(y-y_0)^j$ converges for
  all $y$ in the interval $U_{\tau_i} = \{ |y-y_0|  \leq \tau_i \}$ with $0 < \tau_i
< 1/\| \delta_i\|_{C^0(J)}$, the following bound holds
 \[
   \sup_{k \geq 0 } \tau_i^k \| m_k \| < \infty, \qquad 0 < \tau_i < 1/\| \delta_i\|_{C^0(J)},
 \]
 and by~\cite[Proposition 4.3.4]{buffoni_analytic_2003} 
 the mapping $\psi \colon \Gamma_i \to \BSet$ is $\RSet$-analytic on $U_{\tau_i}$ 
 \cite[Definition 4.3.1]{buffoni_analytic_2003} with
 \[
   d^k \psi_t[y_0] \colon \Gamma_i^k \subset \RSet^k \to
   C^0_{\sigma_i^*}(\Gamma_i^*,C^0(J,\XSet))
   \qquad (\gamma_1,\ldots,\gamma_k) \mapsto 
   \partial_{y_i}^{k} u_n(\blank,y_0,\blank) \gamma_1 \cdots \gamma_k.
 \] 
 The existence of an analytic extension $\tilde u_n$ to $u_n$ follows steps
 identical to the ones in \cref{cor:Sigma}, upon setting $\BSet_C =
 C^0_{\sigma^*_{i}}(\Gamma^*_i,C^0(J,\XSet))$ on the field $\CSet$, amending the
 sequence $\{ \tilde \psi_k \}$ 
\begin{equation}
  \tilde \psi_k \colon \CSet \to \BSet_{\CSet}, \qquad 
  z \mapsto \sum_{j=0}^k \frac{1}{j!} 
  \partial_{y_i}^j u_n(\blank,y_0,\blank) (z-y_0)^j 
  \qquad k  \geq 1,
\end{equation}
and replacing $\delta_i(t)$ by $\| \delta_i \|_{C^0(J)}$.
\end{proofE}

\subsection{Error estimate for the spatial projection, stochastic collocation
scheme}\label{ssec:totalError} 

We are ready to give an a priori estimate for the stochastic collocation scheme's
total error. We use results in \cite{babuskaStochasticCollocationMethod2007} for PDEs with random data, as the
estimates for the interpolation in the $y$ direction do not depend on the particular
problem but only rely on analyticity of the data and solution, and we combine it with
the theory developed in \cite{avitabileProjectionMethodsNeural2023} for neural fields. 
\begin{theorem}[Total error bound]\label{thm:totalBound}
  Assume the hypotheses of \cref{thm:AnalytLRNFAlt}, and \cref{hyp:subgaussian}
  (sub-Gaussianity). There exist positive constants $C_\textnormal{sc}, r_1, \ldots, r_m$ independent
  of $n$ and $q$, such that, for all $t \in J$ it holds 
  \begin{equation}\label{eq:totalErrorBound}
    \begin{aligned}
    \|u(t,\blank)-u_{n,q}(t,\blank)\|_{L^2_\rho(\Gamma)\otimes \XSet}
    & \leq 
    \|u-u_{n,q}\|_{L^2_\rho(\Gamma)\otimes C^0(J,\XSet)} \\
    & \leq 
    e^{\bar K_w T} \|u-P_n u\|_{L^2_\rho(\Gamma)\otimes C^0(J, \XSet)} +
    C_\textnormal{sc}\sum_{i=1}^m \beta_i(q_i)\, e^{-r_i q_i^{\theta_i}},
    \end{aligned}
  \end{equation}
  where, if $\Gamma_i$ is bounded
  \[
   r_i = \log \left[ \frac{2\tau_i}{|\Gamma|} 
       \left( 1+\sqrt{1+ \frac{|\Gamma_i|^2}{4\tau_i^2}}\right) \right],
   \qquad
	 \theta_i=\beta_i=1,
  \]
  and if $\Gamma_i$ is unbounded
  \[
     r_i =\tau_i \zeta_i,
     \qquad \theta_i=\frac{1}{2}, 
     \qquad \beta_i \in O(\sqrt{q_i}) \text{ as $ q_i \to \infty $},
  \]
  $\bar K_w$ being defined in \cref{hyp:analyticityLRNF_Pn}, 
  $\tau_i$ in \cref{thm:C-analyticity-C0}, and
  $\zeta_i$ in \cref{hyp:subgaussian}.
\end{theorem}

\begin{proof} 
  To prove the first inequality we note that, if $f$ is a function in
  $L^2_\rho(\Gamma,C^0(J,\XSet))$ then for all $t \in J$ it holds
  \[
    \| f(t,\blank) \|^2_{L^2_\rho(\Gamma) \otimes \XSet} 
    =  \int_{\Gamma}\| f(t,y) \|_{\XSet}^2 \, \rho(y)\,dy
    \leq 
    \int_{\Gamma} \| f(\blank,y) \|_{C^0(J,\XSet)}^2 \, \rho(y)\,dy
    = \| f \|_{L^2_\rho(\Gamma) \otimes C^0(J,\XSet)}^2,
  \]
  hence we proceed to verify that $u, u_{nq} \in L^2_\rho(\Gamma,C^0(J,\XSet))$ and
  infer
  \begin{equation}\label{eq:unqHomTBound}
    \|u(t,\blank)-u_{n,q}(t,\blank)\|_{L^2_\rho(\Gamma)\otimes \XSet}
    \leq 
    \|u-u_{n,q}\|_{L^2_\rho(\Gamma)\otimes C^0(J,\XSet)} \qquad t \in J.
  \end{equation}
  Since the hypotheses of \cite[Corollary 6.2 with $p =2$]{avitabile2024NeuralFields}
  hold we have $u, u_n \in L^2_\rho(\Gamma,C^0(J,\XSet))$. Further, by
  \cref{lem:ProjAnalyticity} $u_n \in C^0_\sigma(\Gamma,C^0(J,\XSet))$ and 
  the sub-Gaussianity hypothesis (\cref{hyp:subgaussian}) guarantees that the
  embedding $C^0_\sigma(\Gamma,C^0(J,\XSet)) \subset
  L^2_\rho(\Gamma,C^0(J,\XSet))$ is continuous (see \cite[page 1019, discussion preceeding Lemma
  4.2]{babuskaStochasticCollocationMethod2007}). We recall that $u_{n,q}(x,t,y) =
  (I_q u_n(x,t,\blank)(y)$ by definition \cref{eq:unqDef}, thus we interpret
  $I_q$ as an operator on $C^0_\sigma(\Gamma,C^0(J,\XSet))$ to
  $L^2_\rho(\Gamma,C^0(J,\XSet))$ and by \cite[Lemma 4.3 and its proof taking $V =
  C^0(J,\XSet)$]{babuskaStochasticCollocationMethod2007} there exists a constant
  $C_q$ such that
  \[
    \| u_{n,q} \|_{L^2_\rho(\Gamma,C^0(J,\XSet))} =
    \| I_q u_n \|_{L^2_\rho(\Gamma,C^0(J,\XSet))} \leq C_q \| u_n
    \|_{C^0_{\sigma}(\Gamma,C^0(J,\XSet))}
    < \infty.
  \]
  We have thus verified $u, u_{n,q} \in L^2_\rho(\Gamma,C^0(J,\XSet))$, hence
  \eqref{eq:unqHomTBound} holds. By the triangle inequality
  \[
    \|u-u_{n,q}\|_{L^2_\rho(\Gamma)\otimes C^0(J,\XSet)} 
    \leq 
    \|u-u_{n}\|_{L^2_\rho(\Gamma)\otimes C^0(J,\XSet)} 
    +
    \|u_{n}-u_{n,q}\|_{L^2_\rho(\Gamma)\otimes C^0(J,\XSet)} 
    =: I + II.
  \]

  Using \cref{thm:determErrorBound} and \cref{hyp:analyticityLRNF_Pn} for $k =0$ we
  estimate the spatial discretisation error
  \[
    \begin{aligned}
     I^2 
     & = 
     \int_{\Gamma}\| u(\blank,y) - u_{n}(\blank,y)\|_{C^0(J,\XSet)}^2 \rho(y)\,dy \\
     & \leq  
     \int_\Gamma \bigl(e^{T \| P_n W(y) \|}\bigr)^2 \|u(t,y)-P_n u (t,y)\|^2_\XSet\, \rho(y) \,dy 
     \leq 
      \bigl(e^{T \bar K_w}\bigr)^2 \| u - P_n u
      \|^2_{L^2_\rho(\Gamma,C^0(J,\XSet))},
    \end{aligned}
  \]
  which gives the first term on the right-hand side of \cref{eq:totalErrorBound}. For
  the stochastic collocation error the following estimate holds
\[
  II
   \leq 
     C_\textrm{sc}\sum_{i=1}^m \beta_i(q_i)\, e^{-r_i q_i^{\theta_i}}.
\]
with constants given in the theorem statement. This estimate 
follows identical steps to the proof of Theorem 4.1 in \cite[page
1024]{babuskaStochasticCollocationMethod2007}, except that: (i) it is carried out for
$u_n$, not $u$; (ii) it involves functions with values in the Banach space $V =
C^0(J,\XSet)$, and (iii) it uses the analyticity result in
\cref{thm:C-analyticity-C0}.
We omit the
details of the proof of this bound for brevity, and we stress that, with respect to
\cite{babuskaStochasticCollocationMethod2007}, we have different radii of
analyticity $\tau_i$, and we treat simultaneously the Banach spaces $\XSet \in
\{C(D), L^2(D)\}$, but the proof steps are unchanged.
\end{proof}

The bound in \cref{thm:totalBound} is useful also as an upper bound on the first and
second moments of the error, as we can use the upcoming \cref{lem:momentsError} with
$v = u$ and $z = u_{n,q}$. \Cref{lem:momentsError} is an adaptation of \cite[Lemmas
4.7 and 4.8]{babuskaStochasticCollocationMethod2007} to the problem under
consideration.

\begin{lemmaE}[Errors for first and second moments] \label{lem:momentsError}
  If $v,z \in L^2_\rho(\Gamma,\XSet)$, then
  \begin{align}
    & \bigl\| \mean_\rho (v-z) \bigr\|_{\XSet}
      \leq \|  v-z \|_{L^2_\rho(\Gamma) \otimes \XSet}, 
    && \label{eq:vXBound}\\
    & \| \mean_\rho (v^2-z^2) \|_{\XSet}
       \leq \| v + z \|_{L^2_\rho(\Gamma) \otimes \XSet}
           \| v - z \|_{L^2_\rho(\Gamma) \otimes \XSet}, 
    && \XSet = C(D)
     \label{eq:vzCDBound}\\
    & \| \mean_\rho (v^2-z^2) \|_{L^1(D)}
       \leq \| v + z \|_{L^2_\rho(\Gamma) \otimes \XSet}
           \| v - z \|_{L^2_\rho(\Gamma) \otimes \XSet}, 
    && \XSet = L^2(D),
     \label{eq:vzL2Bound}
  \end{align}
  whereas, if $v,z \in L^2_\rho(\Gamma,C^0(J,\XSet))$, then
  \begin{align}
    & \bigl\| \mean_\rho (v-z) \bigr\|_{C^0(J,\XSet)}
      \leq \|  v-z \|_{L^2_\rho(\Gamma) \otimes C^0(J,\XSet)}, 
    && \label{eq:vC0Bound}\\
    & \| \mean_\rho (v^2-z^2) \|_{C^0(J,\XSet)}
       \leq \| v + z \|_{L^2_\rho(\Gamma) \otimes C^0(J,\XSet)}
           \| v - z \|_{L^2_\rho(\Gamma) \otimes C^0(J,\XSet)}, 
    && \XSet = C(D)
     \label{eq:vzC0CDBound}\\
    & \| \mean_\rho (v^2-z^2) \|_{C^0(J,L^1(D))}
       \leq \| v + z \|_{L^2_\rho(\Gamma) \otimes C^0(J,\XSet)}
           \| v - z \|_{L^2_\rho(\Gamma) \otimes C^0(J,\XSet)}, 
    && \XSet = L^2(D),
     \label{eq:vzC0L2Bound}
  \end{align}
\end{lemmaE}
\begin{proofE}  If $\XSet = C(D)$, then \cref{eq:vXBound} is derived using H\"older's
  inequality on $f = v -z$, as follows
  \[
    \begin{aligned}
    \| \mean_\rho f \|_{\XSet} 
    &
    =  \sup_{x \in D} \biggl|\int_{\Gamma} f(x,y)  \rho(y) \, dy \biggr|
      \leq   \sup_{x \in D} \int_{\Gamma} |f(x,y)|  \rho(y) \, dy \\
    & \leq   \int_{\Gamma} \|f(\blank,y)\|_{\XSet} \, 1_\Gamma(y) \rho(y) \, dy  
    \leq \biggl[ \int_{\Gamma} \|f(\blank,y)\|^2_{\XSet}  \rho(y) \, dy \biggr]^{1/2} 
           \biggl[ \int_{\Gamma} 1_\Gamma(y)  \rho(y) \, dy \biggr]^{1/2} \\
    & = \| f \|_{L^2_\rho(\Gamma,\XSet)},
    \end{aligned}
  \]
  whereas for $\XSet = L^2(D)$ the bound \cref{eq:vXBound} is found via the estimate (see also \cite[page 1934]{Zhang.2012}),
  \[
    \begin{aligned}
    \| \mean_\rho f \|^2_{\XSet} 
    & = 
        \int_{D} \biggl| \int_{\Gamma} f(x,y) \rho(y)\,dy \biggr|^2 \,dx
      = \int_{D} \bigl\langle f(x,\blank), 1_\Gamma \bigr\rangle^2_\rho \,dx \\
     & 
      \leq 
      \biggl[ \int_{D} \| f(x,\blank) \|^2_{L^2(\Gamma)} \rho(y)\,dy \biggr] 
      \biggr[ \int_{D} \rho(y) \,dy \biggl]
      =
      \int_{\Gamma} \int_{D} | f(x,y) |^2 \rho(y) \,dx \,dy  =
      \| f \|^2_{L^2_\rho(\Gamma) \otimes \XSet}.
    \end{aligned}
  \]

  To prove \cref{eq:vzCDBound} we set $\XSet \in C(D)$ and derive
  \[
    \begin{aligned}
    \| E_\rho(v^2 -z^2) \|_{\XSet} 
    &= 
      \sup_{x \in  D} \int_{\Gamma}| v(x,y)+ z(x,y)| \, |v(x,y)-z(x,y)| \rho(y)\,dy \\
    & \leq 
       \int_{\Gamma}\| v(\blank,y)  + z(\blank,y)\|_{\XSet} 
                   \, \|v(\blank,y) - z(\blank,y)\|_{\XSet} \rho(y)\,dy \\
    & \leq 
    \biggl[
      \int_{\Gamma}\| v(\blank,y)  + z(\blank,y)\|_{\XSet} \rho(y)\,dy 
    \biggr]^{1/2}
    \biggl[
      \int_{\Gamma}\| v(\blank,y)  - z(\blank,y)\|_{\XSet} \rho(y)\,dy 
    \biggr]^{1/2} \\
    & \leq 
    \| v + z \|_{L^2_\rho(\Gamma) \otimes \XSet}
    \| v - z \|_{L^2_\rho(\Gamma) \otimes \XSet}, 
    \end{aligned}
  \]
  whereas the bound \cref{eq:vzL2Bound}  is proved in \cite[Lemma
  4.8]{babuskaStochasticCollocationMethod2007}. 
  To prove bounds in the $\|  \blank \|_{C^0(J,\XSet)}$ norm it is useful to note
  that, if $f \in L^2_\rho(\Gamma,C^0(J,\XSet))$, then for all $t \in J$ it holds
  \[
    \| f(t,\blank) \|^2_{L^2_\rho(\Gamma) \otimes \XSet} 
    =  \int_{\Gamma}\| f(t,y) \|_{\XSet}^2 \, \rho(y)\,dy
    \leq 
    \int_{\Gamma} \| f(\blank,y) \|_{C^0(J,\XSet)}^2 \, \rho(y)\,dy
    = \| f \|_{L^2_\rho(\Gamma) \otimes C^0(J,\XSet)}^2.
  \]
  Combining this bound with \cref{eq:vXBound} we derive \cref{eq:vC0Bound} jointly
  for $\XSet \in \{ C(D), L^2(D) \}$,
  \[
    \| \mean_\rho v \|_{C^0(J,\XSet)} 
    = \sup_{t \in J} \| \mean_{\rho} v(t,\blank) \|_{\XSet}
      \leq \sup_{t \in J} \| v(t,\blank) \|_{L^2_\rho(\Gamma) \otimes \XSet}
      \leq \| v \|_{L^2_\rho(\Gamma) \otimes C^0(J,\XSet)}.
  \]
  \Cref{eq:vzC0CDBound} is derived using \cref{eq:vzCDBound} and the property
  derived above for $f \in L^2_\rho(\Gamma,C^0(J,\XSet))$,
  \[
    \begin{aligned}
    \| E_\rho(v^2 -z^2) \|_{C^0(J,\XSet)} 
    & \leq 
      \sup_{t \in J} \bigl\| E_\rho \bigl[ v^2(t,\blank) 
                                    - z^2(t,\blank) \bigr] \bigr\|_{\XSet} \\
    &
      \leq 
      \sup_{t \in J}
      \biggl[
      \| v(t,\blank) + z(t,\blank) \|_{L^2_\rho(\Gamma) \otimes \XSet}
      \,
      \| v(t,\blank) - z(t,\blank) \|_{L^2_\rho(\Gamma) \otimes \XSet}
      \biggr] \\
    &
      \leq 
      \| v + z \|_{L^2_\rho(\Gamma) \otimes C^0(J,\XSet)}
      \,
      \| v - z \|_{L^2_\rho(\Gamma) \otimes C^0(J,\XSet)},
    \end{aligned}
  \]
  and similarly \cref{eq:vzC0L2Bound} is a consequence of  \cref{eq:vzL2Bound}
  \[
    \begin{aligned}
    \| E_\rho(v^2 -z^2) \|_{C^0(J,L^1(D))} 
    & \leq 
      \sup_{t \in J} \bigl\| E_\rho \bigl[ v^2(t,\blank) 
                                    - z^2(t,\blank) \bigr] \bigr\|_{L^1(D)} \\
    &
      \leq 
      \sup_{t \in J}
      \biggl[
      \| v(t,\blank) + z(t,\blank) \|_{L^2_\rho(\Gamma) \otimes \XSet}
      \,
      \| v(t,\blank) - z(t,\blank) \|_{L^2_\rho(\Gamma) \otimes \XSet}
      \biggr] \\
    &
      \leq 
      \| v + z \|_{L^2_\rho(\Gamma) \otimes C^0(J,\XSet)}
      \,
      \| v - z \|_{L^2_\rho(\Gamma) \otimes C^0(J,\XSet)},
    \end{aligned}
  \]
\end{proofE}

\subsection{An example of estimating the total error bound using
\cref{thm:totalBound}}\label{ssec:exampleTotalBound} 
We conclude this section by showing how to use \cref{thm:totalBound} in a concrete
example.
\begin{example}[\Cref{ex:FEMColl}, continued]\label{ex:totalErrorBound}
  Consider a RLNF with random data
\begin{equation}\label{eq:vAnaExampleCont}
  v(x,\omega_v) = b_0(x) + \sum_{j \in \NSet_m}  b_j(x) Y_j(\omega_g),
\quad Y_i \stackrel{i.i.d}{\sim}\mathcal{U}[-\alpha,\alpha], 
\quad i \in \NSet^m, 
\quad x \in D = [-1,1],
\end{equation}
\end{example}
in which the functions $\{ b_j \}_j$ as well as the other, \textit{deterministic}, neural field
data $w$ and $g$ are chosen so that, for any fixed $y \in \Gamma_v = [-\alpha,\alpha]$ it holds
$u(\blank,\blank,y) \in C^0(J,C^2(D))$. Further assume the problem is discretise
using a Finite-Element Collocation scheme, as in \cref{ex:FEMColl}. The projections
$P_n$ of this scheme is such that~\cite[Equation 2.3.48]{atkinson1997}
\[
  \| z - P_n z \|_{\infty } 
  \leq \frac{|D|^2}{8n^2} \| v'' \|_{\infty }
  = \frac{1}{2n^2} \| v'' \|_{\infty }, 
  \qquad v \in C^2(D).
\]

We now reason as follows: by \cite[Corollaries 5.1 and 6.2]{avitabile2024NeuralFields} $u \in
L^2_\rho(\Gamma_v, C^0(J,\XSet) )$ and $u_n \in L^2_\rho(\Gamma_v, C^0(J,\XSet_n) )$,
respectively, with $\rho(y) \equiv (2\alpha)^{-1}$, $\XSet = \bigl(C^0(J,D), \| \blank
\|_{\infty }\bigr)$, and $\XSet_n$ as in \cref{ex:FEMColl}.
We now estimate
\[
  \begin{aligned}
  \| u - P_n u \|^2_{L^2_\rho(\Gamma) \otimes C^0(J,\XSet)} 
  & = \int_{\Gamma} \biggl[
      \max_{(x,t) \in J \times D} \bigl| u(x,t,y) - (P_n u (\blank,t,y))(x) \bigr|
                  \biggr]^2
      \rho(y)\,dy \\
  &\leq \biggl(\frac{1}{2n^2}\biggr)^2 \int_{\Gamma} \Biggl[
      \max_{t\in J} \biggl\|  \frac{\partial^2 u}{\partial x^2}(\blank,t,y) \biggr\|
                  \Biggr]^2
      \rho(y)\,dy \\
  & =
    \biggl(\frac{1}{2n^2}\biggr)^2 
    \biggl\| \frac{\partial^2 u}{\partial x^2} \biggr\|^2_{L^2_\rho(\Gamma,C^0(J,\XSet))}
  \leq 
    \biggl(\frac{1}{2n^2}\biggr)^2 
    \biggl\| \frac{\partial^2 u}{\partial x^2} \biggr\|^2_{L^2_\rho(\Gamma,C^0(J,C^2(D)))}
  \end{aligned}
 \]
 We combine the estimate above with \cref{eq:vC0Bound} in \cref{lem:momentsError} and
 \cref{eq:totalErrorBound} in \cref{thm:totalBound}, use the fact that $\Gamma_v$ is
 bounded (hence $q_i = \theta_i = 1$), and that the kernel $w$ is deterministic
 (hence we can take $\bar K_w = \| w \|_{\WSet}$)
 to arrive at
 \begin{equation} \label{eq:exampleErrorSplit}
   \bigl\| \mean_\rho (u-u_{nq}) \bigr\|_{C^0(J,\XSet)}
   \leq 
    C_1 n^{-2} +
    C_2 \sum_{i=1}^m e^{-r_i q_i},
    \qquad C_1 = \frac{e^{T \| w \|_{\WSet}}}{2} \biggl\| \frac{\partial^2 u}{\partial x^2} \biggr\|^2_{L^2_\rho(\Gamma,C^0(J,C^2(D)))}
 \end{equation}
 and hence the error decays as an $O(n^{-2})$ as $n \to \infty $, and exponentially
 as $q_i \to \infty$, for any $i \in \NSet_m$.

\section{Considerations on nonlinear neural fields}\label{sec:nonlinear} 
\Cref{sec:linearAnalysis} discusses analyticity of solutions to the parametrised
problem $u_n'(t,y) = P_n N(t,u_n(t,y),y)$ with initial condition $u_n(0,y) = P_n v(y)$
when $\LSet = 1$ (linear case). Under the boundedness assumption on the projectors
$P_n$, \cref{hyp:analyticityLRNF} ensures that the mapping $ y_i \mapsto
N(\blank,\blank,(y_i, y_i^*))$ is analytic, and \cref{thm:AnalytLRNFAlt} provides
estimates determining the analyticity radius $\delta_i(t)$ of the mapping $y_i \to
u_n(\blank,y_i,y_i^*)$, in the appropriate function space. 

A natural question arises as to how to adapt this approach in the nonlinear case
($\LSet = 0$). The analyticity of solutions to generic Cauchy problems on Banach
spaces is studied in \cite[Theorem 4.11, Section 4.11, pages 165--166]{zeidler1985a}.
If the mapping $y \mapsto N(\blank,\blank,y)$ is affine, conditions for analyticity
of the solution are also given in~\cite{hansenSparseAdaptiveApproximation2013}. With
reference to the continuous parametrised problem
\[
  u'(t,y) = N(t,u(t,y),y), \qquad (t,y) \in J \times \Gamma,  
  \qquad 
  u(0,y) = v(y) \qquad y \in \Gamma
\]
we deduce that if $(t,u,y) \mapsto N(t,u,y)$ is analytic on $V \subseteq \RSet \times
\XSet \times \Gamma$, then the mapping $(t,v,y) \mapsto u(t,v,y)$, in which we
indicate dependence of the solution on the initial condition $v$, is analytic on an
open $Z \subset V$. 

If we pursued a study analogous to the linear case, however, appealing to such result
would be insufficient, because bounding the total error as in \cref{thm:totalBound}
requires an estimate of the radius of analyticity $\tau$ (and this is precisely what
\cref{thm:AnalytLRNFAlt,thm:C-analyticity-C0} do for the linear case).
We have also seen in \cref{rem:contractivity,sssec:kthBoundsContractive} that, in the
case of non-contracting dynamics, such radius shrinks as $t$ increases.

The dynamics of general nonlinear problems is deeply influenced by so-called
bifurcation points. As a simple example, consider a family of equilibria of the
neural field problem, $u_*(y_i)$ which are in generically parametrised in one of the
variables $y_i$. The dynamics of small perturbations to the rest state $u_*(y_i)$ is
governed by a linear problem, and the spectrum of the associated linear operators $A(y_i)$
determines, roughly speaking, whether perturbations decay or grow. The spectrum of
$A(y_i)$ thus changes with the parameter value, and a frequent scenario in nonlinear
problems is that the spectrum of $A(y_i)$ crosses the imaginary axis as $y_i$ varies
across a so-called bifurcation point, thereby determining a behavioural change in the
solution, which may pass from a contracting to an expanding one, and vice versa. 

In other words, in nonlinear problems the domains $\Gamma_i$ are generically
enclosing bifurcation points, which mark sharp transitions in the mapping $y_i
\mapsto u(t,y_i)$ for $t$ fixed. Thus, in nonlinear problems the expectation is that
$A(y_i)$ will generally have eigenvalues on the right-half complex plane, for which
we know that analyticity radii shrink as $t$ grows. Consequently, we can hope to
observe numerically fast convergence rates on long time intervals only in cases where
$\Gamma_i$ does not enclose bifurcations, with the spectrum of $A(y_i)$ staying on
the left-half complex plane for all $y_i \in \Gamma_i$.

We remark that these considerations are valid for generic nonlinear problems, not
just the neural field presented here. In this paper, rather than pursuing the
computation of analyticity radii for nonlinear projected problem, which seems
possible and is expected to give shrinking analyticity radii, we employ stochastic
collocation on nonlinear problems and present numerical evidence of convergence
behaviour in \cref{sec:numerics}.

\section{Numerical results}
\label{sec:numerics}
We include in this section a collection of numerical experiments to test our scheme.
All computations run on a laptop computer with 2 or 4 cores, except the convergence
plots of \cref{subsec:NumericalEx3}, which were run on ADA, an HPC facility at VU
Amsterdam, using 32 cores.

\subsection{Problem 1: linear case, uniformly distributed parameter in $\RSet$} \label{ssec:NumericalEx1} 
As first example we consider a linear neural field
with one-dimensional, non-affine noise in the external input $g$. We choose a system
that meets the hypotheses of \cref{thm:AnalytLRNFAlt} with $\XSet = C(D)$ (see
\cref{ex:gAnaExample} for
checking analyticity of the random input), the sub-gaussianity hypothesis, and to
which the theory of \cref{sec:linearAnalysis} for the error analysis applies. The
problem's solution and QOI are known in analytic
form, allowing to compare theory and numerical experiments.

The neural field is posed on the spatio-temporal domain $ J \times D = [0,T] \times
[-L,L]$ with $T = L = 1$, deterministic data
 \begin{equation}\label{eq:prob1DataDet} 
	w(x,x') = xx', 
  \qquad  f(x)=x, 
  \qquad v(x) = \sin(4\pi x).
 \end{equation}
 and a one-dimensional noise random input of the form $g(x,t,Y(\omega_w))$, with $Y \sim \mathcal{U}[\alpha,\beta]$ and
 \begin{equation}\label{eq:prob1DataRand} 
	g(x,t,y)=e^{ty} \Big[ (y+1) \sin(4\pi x)+\frac{x}{2\pi}\Big],
 \end{equation}
 for which the density $\rho(y) = (\beta - \alpha)^{-1}$ is defined on the compact $\Gamma_g = [\alpha,\beta]$, leading to 
\[ 
  u(x,t,y)=e^{yt}\sin(4\pi x), \qquad 
  \mathbb{E}_{\rho} [ u(x,t,\blank)] = 
  \frac{e^{\beta t}-e^{\alpha t}}{t(\beta-\alpha)}
  \sin (4\pi x),
  \qquad 
  (x,t,y) \in D \times J \times \Gamma_g.
\]

\begin{figure}\label{fig:linearUniform}
\centering
	\includegraphics{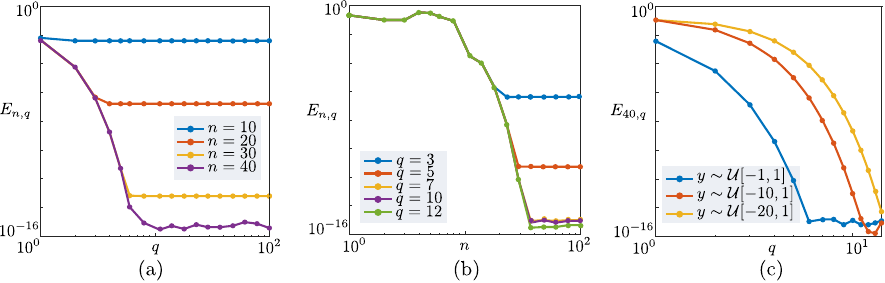}
\caption{Spectral convergence of the stochastic colllocation method paired to
  Chebyshev collocation in space, for the RLNF problem with data
  \crefrange{eq:prob1DataDet}{eq:prob1DataRand}. (a) Error $E_{n,k}$, defined in
  \cref{eq:numError} as a function of $q$, for
  various values of $n$, showing that the error decays expoentially until an
  $n$-dependent plateau, whose value is lower when $n$ increases. (b)
  Error $E_{n,k}$, as a function of $n$, for various values of
  $q$. (c) For fixed $n = 40$ (which gives machine accuracy in the spatial error),
  the error $E_{40,q}$ is exponentially convergent and increases when the variance of
  the distribution $\mathcal{U}[\alpha,\beta]$ is larger. Parameters $T = 1$, $L=1$
  and, in  (a,b), $\alpha = -2$, $\beta = 0.5 $.}
\end{figure}

The quantity of interest we approximate is the expectation of the spatial
profile at final time, $\mathbb{E}_{\rho} [u(x,T,\blank)]$. The problem is
discretised in space using a Chebyshev spectral collocation scheme with $n$ nodes, and
Clenshaw-Curtis quadrature, which is spectrally convergent for fixed values of
$y$~\cite{avitabileProjectionMethodsNeural2023, avitabile2024-codes}.
For the time discretisation we used Matlab's in-built, explicit, adaptive, 4th-order
Runge-Kutta solver \textsc{ode45}.
Our theory does not account for time-discretisation error (see
\cite{avitabile2024NeuralFields} for an example of how this can be carried out in
abstract form) and we set the relative and absolute tolerance to $10^{-14}$ and
$10^{-15}$ respectively, so that spatial and stochastic collocation errors dominate
the time-integration error. Integrals are computed using $q$ Gauss-Legendre
quadrature with $q$ nodes on $\Gamma_g$, generated using \cite{lgwt}, and we
monitor the error
\begin{equation}\label{eq:numError}
  E_{n,q} :=
  \| \mean_\rho u(T,\blank) -  \mean_\rho u_{n,q}(T,\blank) \|_{\XSet} = 
  \max_{x \in D} | \mean_\rho u(x,T,\blank) -  \mean_\rho u_{n,q}(x,T,\blank) |.
\end{equation}
From \cref{thm:determErrorBound}, and reasoning as in
\cref{ex:totalErrorBound,eq:exampleErrorSplit}, we expect $E_{n,q} \leq E^{x}_n +
E^y_q$, with the spatial bound $E^x_n$ depending solely on $n$, and the stochastic
collocation bound $E^y_q$ depending solely on $q$. Differently from
\cref{ex:totalErrorBound}, though we expect both errors to decay exponentially, in
view of the spectral convergence of the Chebyshev scheme \cite[Corollary
4.3]{avitabile2024}.

\Cref{fig:linearUniform}(a) provides evidence of this behaviour, showing $E_{n,q}$
as a function of $q$ for a range of fixed $n$ values. The error decays exponentially
fast, up to a plateu dependent on $n$, signalling that $E^x_n$ dominates $E^y_q$ and further
refinements in $q$ are not useful. \Cref{fig:linearUniform}(b) shows the
complementary behaviour on $E_{n,q}$ as a function of $n$, for various values of $q$.
Finally, \cref{fig:linearUniform}(c) fixes $n = 40$, which guarantees machine
accuracy in \Cref{fig:linearUniform}(b), and shows that higher (but
still exponentially convergent) errors are attained when the variance of the distribution
$\mathcal{U}[\alpha,\beta]$ increases.

\subsection{Problem 2: nonlinear case, uniformly and normally distributed random
parameters in $\RSet^2$}\label{subsec:NumericalEx2} We
then present a nonlinear problem whose solution is not known in closed form, but with
input data that is in use in the mathematical neuroscience literature. We examine
convergence towards a highly resolved solution, in multiple parameters and possibly
unbounded domains. We set $L = 10$,
$T=1$ for the spatio-temporal domain, and take deterministic data
\begin{equation}\label{eq:prob2DataDet}
  w(x,x') = (1-\sigma_w |x-x'|)e^{-\sigma_w |x-x'|} \bigl[ A_0+A_1 \sin(\omega_A x') \bigr],
  \qquad 
  f(u)=\frac{F_0}{1+e^{-\mu(u-h)}},
\end{equation}
and random data of the form $g(x,t,Y_1(\omega_g))$, $v(x,Y_2(\omega_v))$, 
for either $Y_i \stackrel{i. i. d}{\sim} \mathcal{U}[\alpha_i,\beta_i]$ or $Y_i
\stackrel{i. i. d}{\sim} \mathcal{N}(\mu_i,\sigma_i)$, $i = 1,2$, and with functions
given by

\begin{equation}\label{eq:prob2DataRand}
	v(x,y_1)= y_1 e^{-x^2}, \qquad g(x,t,y_2) = y_2 \sin(\omega_g t) e^{-x^2/\sigma_g^2}.
\end{equation}
The system models a neural field with heterogeneous (not translation-invariant)
kernel, sigmoidal firing rate, and random pulsatile initial condition and forcing.
The numerical discretisation and methods follow \cref{ssec:NumericalEx1}, with two
modifications: firstly, we use a dense tensor-product grid for this case, and employ
Gauss-Hermite quadrature computed with \cite{gaussHermite} on the unbounded
domain $\Gamma = \RSet^2$ for normally distributed random variables; secondly,
we monitor solely the error of the stochastic collocation (hence we fix $n^* = 40$),
and we replace the analytic expectation with one computed
with $(q_1^*,q_2^*)$ nodes, that is, instead of the error \cref{eq:numError} we
monitor the discrepancy
\begin{equation}\label{eq:errorq}
  \tilde E_q :=
  \| \mean_\rho u_{n^*,q^*}(T,\blank) -  \mean_\rho u_{n^*,q}(T,\blank) \|_{\XSet},
  \qquad q = (q_1, q_2), \qquad q^* = (q^*_1, q^*_2).
\end{equation}

\Cref{fig:NonlinearUniform} (d) investigates the error using refinements in each
direction, that is, we present the error in the $(q_1,q_2,\tilde E_q)$-space. These
experiments show exponential convergence also in the 2-parameter case.
Further, in \cref{fig:NonlinearUniform}(a,b,c), we examine convergence when
variance is increased, and when the parameters are both uniformly and normally
distributed. In these nonlinear experiments we expect bifurcations to occur in the
model, as neural fields with the deterministic choices made here for $w$ and $f$ are
known to support a zoo of attracting equilibrium states, arranged in branches of
solutions bifurcating at saddle-nodes bifurcations, and having markedly different
spatial profiles \cite{avitabileSnakesLaddersInhomogeneous2015}.
In this system, orbits within a small parameter range
can diverge on finite $T$, we expect the mapping $y \mapsto u(T,\blank,y)$ to
have sharp gradients, and result in a slower convergence rate, as visible in 
\cref{fig:NonlinearUniform}(e). The experiment in \cref{fig:NonlinearUniform}(e)
explains why in \cref{fig:NonlinearUniform}(a,b,c) we see slower convergence rates
for normally distributed parmeters (for which $\Gamma = \RSet^2$ necessarily includes
the bifurcation point) and not for the uniformly distributed parameters (for which
$\Gamma = [\alpha_1,\beta_1] \times [\alpha_2,\beta_2]$ is chosen so as to
avoid the bifurcation point).

\begin{figure}\label{fig:NonlinearUniform}
\centering
	\includegraphics{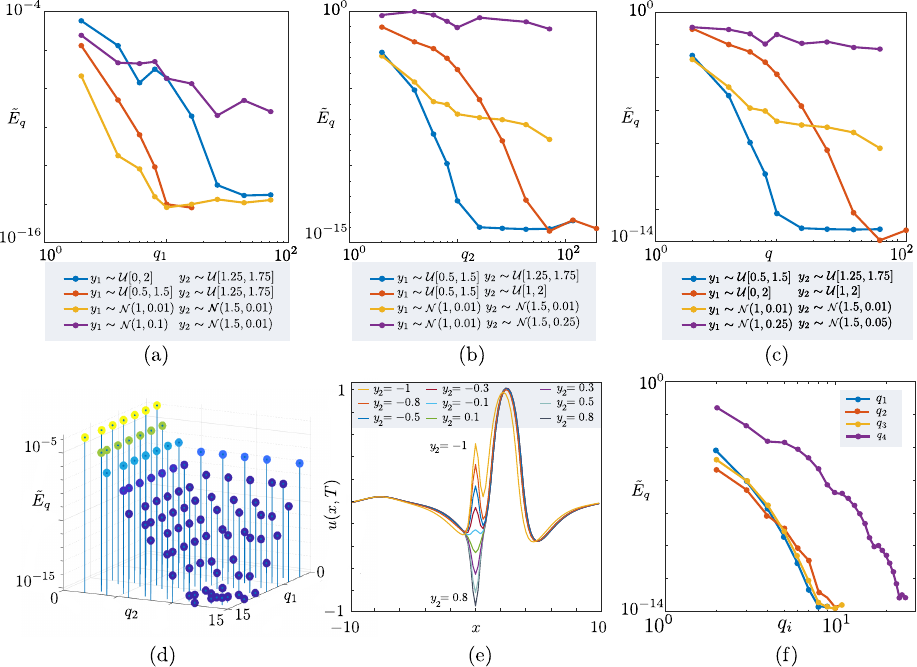}
\caption{Convergence of stochastic collocation error $\tilde E_q$ \cref{eq:errorq},
  for the NRNF problem in \crefrange{eq:prob2DataDet}{eq:prob2DataRand} with $2$- and
  $4$-dimensional random noise. (a) We fix the number of collocation points in the
  $y_2$-direction and refine the tensor grid along $y_1$-direction. We set
  $q^*_2=10$ for the normal distribution and $q^*_2=64$ for the uniform ones. The
  number $q^*_1$ is taken as large as possible to test convergence, but differs for
  each curve. In (b), we fix $q^*_1=16$ in both normal and uniform settings, and
  increase $q_2$. In (c), we increase $q_1=q_2=q$ simultaneously, and we plot $q$. In
  (a--c) we note that the error in $y_2$ dominates over the one in $y_1$, and slower
  convergence is observed when the variance increases. In (d), the noise in the external
  input and the initial condition depends on $y_1\sim
  \mathcal{U}[1.25,1.75]$ and $y_2\sim \mathcal{U}[0.5,1.5]$ and we show an
  experiment in which $q^*=(15,15)$ is kept fixed for all simulations. $(e)$ We show
  nine solution profiles at final time $T$ of the NRNF, for $y_1=1$ and ten different
  values of the parameter $y_2\in[-1, 0.8]$, giving evidence of a bifurcation in
  this parameter range, and explaining the lower convergence rates seen in (a--c) for
  normally-distributed parameters (see main text). In (f) we refine one grid
  direction at a time for the NRNF with 4 parameters, as given in
  \cref{subsec:NumericalEx3}. Other parameter values are:
  $\sigma_{\omega}=1$,
  $F_0=1$, $\mu=10$, $h=0.3$, $\omega_g=1$ and $\sigma_g=0.4$, $n^* = 40$, $T = 1$.}
\end{figure}

\subsection{Problem 3: nonlinear case, uniformly distributed random parameters in
$\RSet^4$}\label{subsec:NumericalEx3} In a final test we amend Problem 2 and make
the variables $A_0$ and $F_0$ uniformly distributed, so that all input data is
random, with $y \in \RSet^4$. \Cref{fig:NonlinearUniform}(f) displays $\tilde E_q$ as we
refine points one direction at a time, and confirms that similar convergence rates to
the case $y \in \RSet^2$ can be attained in this case.

\section{Conclusions}\label{sec:conclusions} 
In this paper we have introduced a scheme for forward UQ in neural field models,
using stochastic collocation. The method adapts stochastic collocation from PDE
literature to integro-differential equations arising in mathematical neuroscience,
and it combines with generic spatial discretisations for the problem. The method
presented here can be readily implemented to perform UQ, upon coding the
spatially-projected forward simulation problem within existing UQ software such as
\textsc{Sparse Grids Matlab Kit}~\cite{piazzola.tamellini:SGK} and
\textsc{UQLab}~\cite{marelliUQLabFrameworkUncertainty2014}.

It seems possible to study variations of stochastic collocation which, in the PDE
literature, overcome some of the limitations of the method, and are
suitable for a large number of parameters, in particular sparse grid, sparse
Galerkin, and anisotripic methods \cite{
nobileSparseGridStochastic2008a,
nobileAnisotropicSparseGrid2008,
  hansenSparseAdaptiveApproximation2013,
hoangSparseTensorGalerkin2013a,
ernstConvergenceSparseCollocation2018,
adcockSparsePolynomialApproximation2022}. Further, we expect that a similar approach to
the one taken here, which splits the spatial and stochastic error for the problem,
will help us analysing Monte Carlo Finite Element schemes and its variants (in particular we should
be able to adapt directly the methods in \cite[Chapter 9]{Lord:2014ir}). An immediate
application of the theory presented in this paper is the use of stochastic
collocation for Bayesian inverse problems in neural fields, which can follow its PDE
analogues \cite{marzoukStochasticCollocationApproach2009} and is promising for
low-dimensional parametrised problems.

\section*{Acknowledgements} We are grateful to Dirk Doorakkers and Hermen Jan Hupkes,
for discussing conditions that guarantee contractivity of parameter-dependent
families of nonlinear operators \cref{hyp:contractivity,lem:contractivityCondition}. DA and
FC acknowledge
support from the National Science Foundation under Grant No. DMS-1929284 while the
authors were in residence at the Institute for Computational and Experimental
Research in Mathematics in Providence, RI, during the “Math + Neuroscience:
Strengthening the Interplay Between Theory and Mathematics" programme.

\appendix 
\section{Perturbative result on semigroup operator}
The proof of $C^0_\sigma$-regularity for the solution uses the following perturbative
result on operator semigroups
\begin{lemma}[Perturbation of semigroup operators]\label{lem:semigroupPert}
  Let $A,B \in BL(\XSet)$, where $\XSet$ is a Banach space. The
  uniformly continuous semigroups $( e^{tA})_{t \geq 0}$, $( e^{tB})_{t \geq 0}$,
  generated on $\XSet$ by $A$ and $B$, respectively, satisfy 
  \[
    \| e^{tA}- e^{tB} \| \leq  t \| A - B \| e^{t \| A-B \|} e^{t \min (\| A \|, \| B \|)},
    \qquad t \in \RSet_{\geq 0}.
  \]
\end{lemma}
\begin{proof} By \cite[Chapter III, Section 1, Corollary 1.7]{engelShortCourseOperator2006} it holds
  \begin{equation}\label{eq:engelBound}
    (e^{tB}-e^{tA})v = \int_{0}^{t} e^{(t-s)A} (B-A) e^{sB} v\,d s,
    \qquad v \in \XSet,  \qquad t \in \RSet_{\geq 0}.
  \end{equation}
  Taking norms, and recalling $\| e^{tA} \| \leq e^{t\| A \|}$ for
  any $A \in BL(\XSet)$ and $t \in \RSet_{ \geq 0}$ we get
  \[
  \begin{aligned}
    \| (e^{tB}-e^{tA})v \| & \leq  \| B-A \| \| v \| \int_{0}^{t} e^{t\| A \|}
                                                     e^{s(\| B \| - \| A \|)} \,d s \\
                           & \leq  \| B-A \| \| v \| e^{t\| A \|} \int_{0}^{t} 
                                                     e^{s\| B - A \|} \,d s \\
                           & <  t\| A-B \| e^{t\| A - B \|} e^{t\| A \|} \| v \|, 
  \end{aligned}
  \]
  hence
  \[
    \| e^{tA}- e^{tB} \| \leq  t \| A - B \| e^{t \| A-B \|} e^{t \| A \|},
    \qquad t \in \RSet_{\geq 0}.
  \]
 Swapping $A$ and $B$ in \cref{eq:engelBound} we obtain
  \[
    \| e^{tA}- e^{tB} \| \leq  t \| A - B \| e^{t \| A-B \|} e^{t \| B \|},
    \qquad t \in \RSet_{\geq 0},
  \]
  and combining the last two bounds we prove the statement.
\end{proof}
\section{Further proofs}
\printProofs

\bibliography{references}
\bibliographystyle{siamplain}
 
\end{document}